\documentclass[reqno,twoside,11pt]{amsart}

\usepackage{amsmath,amsfonts,calrsfs,fullpage,amssymb,color,verbatim,eucal,yfonts,mathrsfs,marginnote, mathtools}

\newtheorem{Theorem}{Theorem}[section]
\newtheorem{Definition}[Theorem]{Definition}
\newtheorem{Proposition}[Theorem]{Proposition}
\newtheorem{Lemma}[Theorem]{Lemma}

\newtheorem{Remark}[Theorem]{Remark}

\newtheorem{Hypothesis}[Theorem]{Hypothesis}
\makeatletter
\@addtoreset{equation}{section}

\makeatother
 
\def\R{\mathbb R}
\def\N{\mathbb N}

\def\C{\mathbb C}

\def\eps{\varepsilon}
\def\ds{\displaystyle}

\title{\bf Schauder theorems for a class of (pseudo-)differential operators on finite and infinite dimensional state spaces}\date{}

\author[A. Lunardi]{Alessandra Lunardi}
\address{
Dipartimento di Scienze Matematiche, Fisiche e Informatiche\\
Universit\`a di Parma\\
Parco Area delle Scienze, 53/A\\
43124 Parma, Italy}
\email{alessandra.lunardi@unipr.it}

\author[M. R\"ockner]{Michael R\"ockner}
\address{Fakultat f\"ur Mathematik\\
 Bielefeld Universit\"at, D-33501 Bielefeld, Germany;
 Academy of Mathematics and Systems Science, CAS, Beijing 100190, China 
 }
\email{roeckner@math.uni-bielefeld.de}



\begin{document}

 \begin{abstract}  
We prove maximal regularity results in H\"older and Zygmund spaces for linear stationary and evolution equations  driven by a   class of differential and pseudo-differential operators $L$, both in finite and in infinite dimension. The assumptions are given in terms of the semigroup generated by $L$. We cover the cases of fractional Laplacians and Ornstein-Uhlenbeck operators with fractional diffusion in finite dimension, and several types of local and nonlocal Ornstein-Uhlenbeck operators, as well as the Gross Laplacian and its   fractional    powers,  in infinite dimension. 
 \end{abstract}

 \maketitle

\section{Introduction}
This paper is devoted to maximal regularity results in H\"older and Zygmund spaces for linear stationary and evolution equations  driven by a     class of differential and pseudo-differential operators $L$, both in finite and in infinite dimension. The underlying space $X$  is  any separable real Banach space, that may be either $\R^N$ or infinite dimensional. 

The operators  $L$ under consideration are the generators of the so called {\em generalized Mehler semigroups}, namely semigroups of operators  in the space $C_b(X)$ of the  continuous and bounded functions from $X$ to $\R$  that may be represented as
 \begin{equation}
 \label{P_t}
 P_tf(x) = \int_{X} f(T_tx + y)\mu_t(dy), \quad t\geq 0, \; f\in C_b(X). 
 \end{equation}
Here $T_t$ is a strongly continuous semigroup of bounded operators  on  $X$, and $\{\mu_t:\; t\geq 0\}$ is a family of Borel probability measures in $X$ such that $\mu_0= \delta_0$ (the Dirac measure at $0\in X$), $t\mapsto \mu_t$ is weakly continuous in $[0, +\infty)$ and 
 \begin{equation}
 \label{semigruppo}
 \mu_{t+s} = (\mu_t \circ T_s^{-1})\ast \mu_s, \quad t, s >0 . 
  \end{equation}
Such a condition is necessary and sufficient for $P_t$ be a semigroup  (namely, $P_{t+s} = P_t \circ P_s$ for $t$, $s\geq 0$), even in the space $\mathcal{B}_b(X)$ of the bounded, Borel measurable  functions  $f:X\rightarrow\R$.

Then for every $f\in C_b(X)$ the function $(t,x)\mapsto P_tf(x)$ is continuous in $[0, +\infty)\times X\mapsto \R$, and this allows to define a closed operator $L$ in $C_b(X)$ through its resolvent,  
 \begin{equation}
 \label{L}
 R(\lambda, L)f (x) = \int_0^{\infty} e^{-\lambda t}P_tf(x)\,dt, \quad \lambda >0, \;f\in C_b(X), \; x\in X. 
  \end{equation}
$L$ is called the {\it generator } of $P_t$, although it is not the infinitesimal generator in the standard sense since $P_t$ is not strongly continuous in $C_b(X)$, in general.
 
Though   this paper's    results and techniques of proof are purely analytic, let us briefly recall the probabilistic framework in which generalized Mehler semigroups occur. In fact, they are the transition semigroups of solution processes to the following type of stochastic differential equations (meant in the weak or mild sense) on $X$:
 \begin{equation}
 \label{SDE}
dX(t)=AX(t)dt + dY(t), \quad  t>0; \; X(0)=x,
\end{equation}
where $A:D(A)\subset X\rightarrow X$ is the infinitesimal generator of $T_t$,  and $Y(t)$, $t\geq0$, is a Levy process in $X$, i.e. a stochastic process in $X$ with cadlag paths starting at 0, defined on a probability space $(\Omega,\mathcal{F},\mathbb{P})$, and having stationary and independent increments. It is characterized by a negative definite function $\lambda:X^*\rightarrow \mathbb{C}$ (where $X^*$ is the dual space of $X$), satisfying
\begin{align}\label{eq:exp}
\int_\Omega  e^{i \xi(Y(t)(\omega)) }\mathbb{P}(d\omega)=\exp(-  t  \lambda(\xi)),\quad \xi\in X^*, \;t>0.
\end{align}
Then   the transition semigroup for the solution $X(t,x)$ of \eqref{eq:exp} (called ``Ornstein-Uhlenbeck process on $X$" in the case that $Y(t)$ is a Wiener process,  and  ``Ornstein-Uhlenbeck process on $X$ with jumps" if $Y(t)$  is a more general Levy process)
 is given by $P_t$ as in \eqref{P_t}, i.e., for f $\in\mathcal{B}_b(X)$, $x\in X$, $t>0$,
\begin{align}
\int_\Omega f(X(t,x)(\omega))\mathbb{P}(d\omega)=P_tf(x),
\end{align} 
where $X(t,x)$, $t\geq0$, denotes the (weak or mild) solution of \eqref{eq:exp} with $X(0,x)=x$ $\mathbb{P}$-a.s.\\
 We  then have an explicit formula for the Fourier transforms of $\mu_t$, $t>0$, in terms of $\lambda$ and $T_t$,  namely
\begin{align}\label{def:mu-t}
\hat\mu_t(\xi):=\int_X \exp(i  \xi(z))\,\mu_t(dz)= \exp \left(-\int_0^t\lambda(T^*_s\xi)\,ds \right),\quad \xi\in X^*,\; t>0,
\end{align}
where $T^*_t$ denotes the dual  semigroup of $T_t$.

There have been a number of papers on generalized Mehler semigroups and their related Ornstein-Uhlenbeck processes with jumps. We refer e.g. to \cite{BRS, ACM, DL, DLScS, FR, LR1, LR2, ZL, OR, ORW, SchSun} and the references therein.\\
Now let us come back to the main results of this paper,  which are purely analytic. 
 What we prove  are maximal H\"older and Zygmund regularity results both for the stationary equation 
\begin{equation}
\label{eq:staz}
\lambda u(x) - Lu(x) = f(x), \quad x\in X, 
\end{equation}
namely for the function $u=R(\lambda, L)f$ defined in \eqref{L}, and for the mild solutions of evolution problems, given by 
\begin{equation}
\label{mild}
v(t,x) = P_tf(x) +\int_0^t P_{t-s}g(s, \cdot)(x)\,ds , \quad 0\leq t\leq T, \;x\in X, 
\end{equation}
with continuous and bounded $f$, $g$. 

Of course, we need some ``regularity" hypothesis on the measures $\mu_t$ in connection with the semigroup $T_t$. 
Specifically, we assume that there exists a Banach space $H\subset X$ such that $  T_t   (H)\subset H$, and such that each $\mu_t$ is Fomin differentiable  along   $T_t(H)$, namely for every $h\in H$, $t>0$ there exists $\beta_{t,h}\in L^1(X, \mu_t)$ such that
\begin{equation}
\label{def:Fomin}
\int_X \frac{\partial f}{\partial (T_th)}(x)\,\mu_t(dx) = - \int_X \beta_{t,h}(x) f(x)\, \mu_t(dx), \quad f\in C^1_b(X). 
\end{equation}
Moreover we assume that  there exist $C>0$, $\omega\in \R$, $\theta >0$ such that 
\begin{equation}
\label{ipotesi_intro}\|T_th\|_H \leq Ce^{\omega t}\|h\|_H, \quad \| \beta_{t,h}\|_{L^1(X, \mu_t)}\leq \frac{Ce^{\omega t}}{t^\theta }\|h\|_H, \quad t>0, \,h\in H. 
\end{equation}

These assumptions are satisfied in several remarkable examples. We consider the following ones. 

\vspace{3mm}

\hspace{3mm} (a) In finite dimension, with $X=H=\R^N$,  they are satisfied by the heat semigroup with  $\theta = 1/2$, by the semigroups generated by the powers $-(-\Delta)^s$ for $s\in (0,1)$, and more generally by Ornstein-Uhlenbeck semigroups with fractional diffusion, 
$$\mathcal Lu (x) =  \frac{1}{2}(\text{Tr}^s(QD^2u) )(x)  - \langle Bx, \nabla u(x)\rangle , \quad x\in \R^N, $$
where $Q$ is any symmetric positive definite matrix, $B$ is any matrix, and   Tr$^s(QD^2)$ is the pseudo-differential operator with symbol $\langle Q\xi, \xi\rangle^s$, $s\in (0, 1)$. The semigroup $T_t$ is now $e^{-tB}$, and the measures $\mu_t$ are given by $\mu_t(dx)= g_t(x)dx$, with  $g_t\in W^{1,1}(\R^N)$, so that $\mu_t$ is Fomin differentiable along all directions, and \eqref{ipotesi_intro} holds with $H=\R^N$ and $\theta = 1/(2s)$. See Sections \ref{subs:Laplacian}, \ref{sect:OU}. 

\vspace{3mm}

\hspace{3mm} (b)  
In infinite dimension they are satisfied by a class of smoothing (strong Feller) Ornstein-Uhlenbeck semigroups, still with $H=X$, that includes the ones 
considered in \cite{DPZ}, 
and by a class of not strong Feller Ornstein-Uhlenbeck semigroups, that includes the classical Ornstein-Uhlenbeck semigroup used in the Malliavin Calculus, and other non symmetric Ornstein-Uhlenbeck semigroups such as in \cite{VN,VNW}; here 
 $H$ is the Cameron-Martin space of a reference Gaussian measure $\mu$. In all these cases the measures $\mu_t$ are Gaussian, and we have $\theta= 1/2$, see Section \ref{sect:OUinf_dim}.   In Section \ref{sect:NLOUO} we consider nonlocal perturbations  of the generator   of a specific strong Feller Ornstein-Uhlenbeck semigroup and show that \eqref{def:Fomin} and \eqref{ipotesi_intro} also hold in such a case, still with $H=X$ and $\theta =1/2$.
Moreover, when $X$ is a Hilbert space endowed with a centered Gaussian measure $\mu$ and $H$ is the Cameron-Martin space of $\mu$, \eqref{def:Fomin} and \eqref{ipotesi_intro} are satisfied by the semigroup generated by the Gross Laplacian $G$, again with $\theta = 1/2$, and by the semigroups generated by  $-(-G)^s$ with $s\in (0, 1)$ and $\theta = 1/(2s)$, in which case the measures $\mu_t$ are mixtures of measures. See Section \ref{sect:Gross}.   In Section \ref{sub:nonGauss} we show that some nonlocal versions of the classical Ornstein-Uhlenbeck semigroup from Malliavin calculus still satisfy our assumptions. 

\vspace{3mm}

 Our techniques are  independent of the dimension of the state space $X$, and the most important and newest part of the paper is in the infinite dimensional case.  Indeed, several familiar tools  in finite dimension, such as 
Calderon-Zygmund theory, Fourier transform, and  the uncountable consequences of local compactness, are not available in infinite dimension, as well as any translation invariant reference measure such as the Lebesgue measure.  

 Needless to say, maximal regularity results are very rare in infinite dimension. A few $L^p$ maximal regularity results, with $p\in (1, +\infty)$,  have been proved  for certain Ornstein-Uhlenbeck  stationary equations; in these cases the solution to \eqref{eq:staz} belongs to a suitable $W^{2,p}$ space with respect to an invariant Gaussian measure $\mu$ whenever $f\in L^p(X, \mu)$. After the pioneering Meyer inequalities for the classical Ornstein-Uhlenbeck operator (\cite{Meyer}, see also \cite[Sect. 5.6]{Boga}), maximal $L^p$ regularity  for a more general class of Ornstein-Uhlenbeck equations was proved  in \cite{ChMG1,ChMG2,MVN}. Concerning non Gaussian measures, the only available results are for $p=2$,   about  (nontrivial) perturbations of certain Ornstein-Uhlenbeck equations (\cite{DPL1,CF}); here $\mu$ is an invariant Gibbs  (= weighted Gaussian) measure. For $p=2$ some of the above results have been extended to the case where the whole $X$ is replaced by a good domain $\mathcal O \subset X$, with generalized Dirichlet or Neumann boundary conditions (\cite{DPL2,DPL3,Cappa,CF1}). 

 Also the literature about maximal H\"older regularity in infinite dimension is very scarce, dealing mainly with Ornstein-Uhlenbeck equations or with equations driven by the Gross Laplacian, see e.g.  \cite{DPZ,C,CL} and the references therein. More details are in  Sections \ref{sect:OUinf_dim}, \ref{sect:Gross}. Moreover, Schauder estimates for some nontrivial perturbations of a specific Ornstein-Uhlenbeck operator in the space $X= C([0,1])$ were  proved in \cite{cdp}.

In our general setting, $P_t$ is smoothing along $H$: for every $f\in C_b(X)$ and $t>0$, $P_tf$ has continuous Gateaux derivatives of any order along $H$, and for every $(h_1, \ldots , h_n)\in H^n$ we have
\begin{equation}
\label{stimaderivate_intro}
\left| \frac{\partial^n P_tf}{ \partial h_1  \ldots  \partial h_n}(x)\right| \leq C_n\left(1 + \frac{1}{t^{n\theta}} \right) \prod_{j=1}^n\|h_j\|_H \|f\|_{\infty}, \quad t>0, \; x\in X. 
\end{equation}
On the other hand, in general $P_tf$ is not Gateaux differentiable along other subspaces than $H$. Therefore, any regularity result is expressed in terms of regularity along $H$. The H\"older spaces that we use are in fact defined by 
$$ C^{\alpha}_H(X) = \left\{ f\in C_b(X): \; [f]_{C^{\alpha}_H(X)} := \sup_{x\in X, \,h\in H\setminus \{0\}} \frac{|f(x-h)- f(x)|}{\|h\|_H^{\alpha}} <+\infty \right\}, 
$$
$$ \|f\|_{C^{\alpha}_H(X)} =   \|f\|_{\infty} +  [f]_{C^{\alpha}_H(X)} , $$
for $\alpha \in (0,1)$. In the case $H=X$ this is the usual space of bounded and $\alpha$-H\"older continuous functions from $X$ to $\R$. 

The Schauder type regularity results for \eqref{eq:staz} are the following, 

\vspace{3mm}

 (i)  If $1/\theta \notin \N$,  for every $\lambda >0$ and $f\in C_b(X)$ 
the solution $u$  to \eqref{eq:staz} belongs to $C^{ 1/\theta}_H(X)$, and there is $C(\lambda)$ independent of $f$ such that $\|u\|_{C^{ 1/\theta}_H(X)} \leq C(\lambda)\|f\|_{\infty}$. 

 (ii) If $\alpha\in (0, 1)$ and  $\alpha + 1/\theta \notin \N$, for every $\lambda >0$ and $f\in C^{\alpha}_H(X)$ 
the solution $u$ to \eqref{eq:staz} belongs to $C^{\alpha + 1/\theta}_H(X)$ and  there is $C(\lambda, \alpha)$ independent of $f$ such that $\|u\|_{C^{\alpha+ 1/\theta}_H(X)} \leq C(\lambda, \alpha)\|f\|_{C^{\alpha}_H(X)}$. 

\vspace{3mm}
Here, for $\sigma\in (0,1)$ and $k\in \N$, $C^{\sigma + k}_H(X)$ denotes  the space of  all  continuous and bounded functions from $X$ to $\R$ that possess continuous and bounded Gateaux derivatives of any order $\leq k$ along $H$, and such that all the $k$-th order derivatives belong to $C^{\sigma}_H(X)$, endowed with its natural norm. If $H=X$ this is the space of the $k$ times Gateaux differentiable functions with continuous and bounded Gateaux derivatives of any order $\leq k$, and such that all the $k$-th order derivatives are $\alpha$-H\"older continuous. 

The exponents $1/\theta$ in (i), and $\alpha + 1/\theta$ in (ii) are optimal, in the sense that  they cannot be replaced by $1/\theta +\varepsilon$, $\alpha + 1/\theta + \varepsilon$ respectively,  for any $\varepsilon >0$. 

In the critical cases  $\alpha + 1/\theta = k\in \N$ (with $\alpha =0$ in statement (i), $\alpha\in (0,1)$ in statement (ii)) we do not expect that the solution to \eqref{eq:staz} has bounded partial derivatives of order $k$, in general. The simplest counterexample is given by the Laplacian in finite dimension. The heat semigroup in $\R^N$ has the representation \eqref{P_t} with $T_t=I$, $\mu_t(dx) = (4\pi t)^{-N/2}\exp(-|x|^2/4t)dx$, so that it
satisfies \eqref{def:Fomin} and \eqref{ipotesi_intro} with $X=H=\R^N$ and $\theta = 1/2$; however  it is well known that for every $\lambda >0$ the solution to 
$\lambda u - \Delta u = f$ with $f\in C_b(\R^N)$ is not twice continuously differentiable and its first order derivatives are not Lipschitz continuous in general, if $N>1$. They belong to the Zygmund space $Z^1(\R^N)$, namely they satisfy $|D_ku(x+2h) -2D_ku(x+h) + D_ku(x)| \leq C|h|$ for every $k= 1, \ldots, N$, $x$, $h\in \R^N$, with $C$ independent of $x$ and $h$. We extend this result to our general setting, introducing the Zygmund spaces $Z^n_H(X)$ for $n\in \N$ and showing that in the above critical cases the solution to \eqref{eq:staz} belongs to $Z^k_H(X)$. 

Similar results are proved for the mild solutions to Cauchy problems with continuous and bounded data, 
\begin{equation}
\label{Cauchy_intro}
\left\{ \begin{array}{l}
v_t(t,x)  = Lv(t,\cdot)(x)  + g(t,x), \quad t\in [0,T], \; x\in X, 
\\
\\
v(0, \cdot) = f, \end{array}\right. 
\end{equation}
namely for the functions $v$ defined by 
$$v(t,x) = P_tf(x) + \int_0^t P_{t-s}g(s, \cdot)(x)\,ds , \quad  t\in [0,T], \; x\in X, $$
with $f\in C_b(X)$, $g\in C_b([0,T]\times X)$. Our assumptions  are not strong enough to guarantee that $v$ is differentiable with respect to $t$, so it is just a mild solution and not a classical one. Moreover, time-space Schauder estimates such as the standard ones for the heat equation
are not available in general. For instance, they are not available when $L$ is  the classical one-dimensional Ornstein-Uhlenbeck operator $Lu(x)= u''(x) - xu'(x)$, as a consequence of \cite{DPL}. 
So, our Schauder and Zygmund regularity results concern only space regularity.  
More precisely, we introduce the space $C^{0,\alpha +k}_H([0,T]\times X)$ for $\alpha\in (0, 1)$, $k\in \N \cup \{0\}$, consisting of the bounded continuous functions $g:[0,T]\times X\mapsto \R$ such that $g(t, \cdot)\in C^{\alpha +k}_H(X)$ for every $t\in [0,T]$,  $\|g\|_{C^{0, \alpha +k}_H([0,T]\times X)} :=
\sup_{t\in [0,T]} \|g(t, \cdot)\|_{C^{\alpha +k}_H(X)} <+\infty$, and if $k \neq 0$, all the Gateaux derivatives $\partial^j g/ \partial h_1 \ldots \partial h_j $ with $j\leq k$ and $h_1, \ldots , h_j \in H$ are continuous in $[0,T]\times X$. We prove that 

\vspace{3mm}

 (i)  If $1/\theta \notin \N$,  for every  $f\in C^{1/\theta}_H(X)$ and $g\in C_b([0,T]\times X)$, 
the function $v$ 
belongs to $C^{0, 1/\theta}_H([0,T]\times X)$, and there is $C(T)$ independent of $f$ and $g$  such that $\|v\|_{C^{ 0, 1/\theta}_H([0,T]\times X)} \leq $ $C(T)$ $(\|f\|_{C^{1/\theta}_H(X)} + \|g\|_{\infty} )$. 

 (ii) If $\alpha\in (0, 1)$ and  $\alpha + 1/\theta \notin \N$, for every   $f\in C^{\alpha + 1/\theta}_H(X)$ and $g\in C^{0,\alpha }_H([0,T]\times X)$
the function $v$  belongs to $C^{0, \alpha + 1/\theta}_H([0,T]\times X)$ and  there is $C(T, \alpha)$ independent of $f$ and $g$ such that $\|v\|_{C^{0, \alpha+ 1/\theta}_H([0,T]\times X)} \leq C(T, \alpha)( \|f\|_{C^{\alpha +1/\theta}_H(X)} + \|g\|_{ C^{0,\alpha }_H([0,T]\times X)} )$. 

\vspace{3mm}

In the critical cases $\alpha + 1/\theta \in \N$ we obtain Zygmund space regularity results, as in the stationary case. 

The proofs rely on estimates \eqref{stimaderivate_intro} and on the better estimates for $\alpha\in (0, 1)$, 
\begin{equation}
\label{stimaderivateHolder_intro}
\left| \frac{\partial^n P_tf}{ \partial h_1  \ldots  \partial h_n}(x)\right| \leq C_{n, \alpha}\left(1 + \frac{1}{t^{(n-\alpha)\theta}} \right) \prod_{j=1}^n\|h_j\|_H \|f\|_{C^{\alpha}_H(X)}, \quad t>0, \; x\in X, \; f\in C^{\alpha  }_H(X), 
\end{equation}
through a procedure  that employs interpolation techniques such as in the recent paper \cite{CL} where the classical Ornstein-Uhlenbeck operator in infinite dimension is considered. 

Both in the stationary and in the evolution case the general results are applied to the above mentioned examples, and yield old and new maximal regularity results. Comparisons with the literature dealing with H\"older and Zygmund maximal regularity   for such examples    are given in Sections \ref{subs:Laplacian}, \ref{sect:OU}, \ref{sect:OUinf_dim},  \ref{sect:Gross}.   To give a complete account on all the contributions to Schauder theory in finite dimension is beyond the scope of this paper.

 Finally, we would like to explain the motivation to study Schauder estimates
for this class of partial (pseudo-) differential equations in infinitely many variables. 
The interest in such PDEs has risen enormously in recent years, because they occur 
as forward and backward Kolmogorov equations for stochastic PDEs, an area that 
has become one of the major directions of research in probability theory and, in 
particular, in stochastic analysis. See e.g.  \cite{DPZ, C, DPBirkh, BKRS, BDPRS} and the references therein. Solving any of the two provides a way to obtain the 
time marginal laws of the solution to the SPDE in a purely analytic way, without having to 
solve the SPDE itself. In many important cases the time marginal laws determine 
the solutions to the SPDE completely and the latter can be reconstructed from the former.
Hence to understand such PDEs in infinite dimensions becomes important and regularity  
results for their solutions mandatory. In particular, because of the already mentioned lack of a Lebesgue 
measure on infinite dimensional spaces and since high order Sobolev spaces are not embedded in spaces of continuous functions (even for Gaussian measures), Schauder regularity  
appears to be more feasible here. 

In the evolutionary case the class of PDEs considered 
in this paper are just (after time reversal) the Kolmogorov backward equations corresponding to SPDEs of
type \eqref{SDE}, whose solutions are infinite dimensional Ornstein-Uhlenbeck processes with 
L\'evy noise, as explained above. These have been used as model cases in many other 
aspects and are used in our paper as a model case to understand Schauder theory in 
infinite dimension. It turns out that already in this case new and typical infinite dimensional 
phenomena occur, as e.g. regularity of solutions can only hold along subspaces, 
which are intrinsically linked to the type of noise in equation  \eqref{SDE} and its relation to the 
possibly unbounded operator $A $ in its drift. This is expressed in pure analytical language through the differentiability properties of the measures $\mu_t$ in relation to the semigroup generated by $A$; see conditions \eqref{def:Fomin} and \eqref{ipotesi_intro} above. Clearly, if one does not fully understand 
such phenomena in this model case, one has no chance to develop a 
Schauder theory for more general Kolmogorov equations in infinite dimension.  The next step 
would then be to look at perturbations of the situation studied in this paper, for example by adding a first order part to the Ornstein-Uhlenbeck type Kolmogorov operators considered here, which is given by a nonlinear vector field in the Banach space $X$. This is the object of a paper already in 
preparation by the two authors. A further step would then be to perturb the higher order part in a
``geometrically comparable" way, similarly to  what is usually done in finite dimension, in going from the Laplacian to  strictly elliptic operators.   

 The structure of this paper is as follows. In Section \ref{sect:notation} we mainly fix notations. In Section \ref{sect:schauzyg} we introduce our hypotheses, state and prove our main results described above. In particular, we prove an explicit formula for the $n$-th G\^{a}teaux derivative of $P_tg$ for $g\in C_b(X)$ in Proposition \ref{representation}. Sections \ref{sect:finite} and \ref{sect:infinite} are devoted to examples in finite and infinite dimensions respectively.

\section{Notation and preliminaries}
\label{sect:notation}
 
Below, $X$, $Y$ are Banach spaces. If we write $X\subset Y$ this means that $X$ is contained in $Y$ with continuous embedding. By $\mathcal L(X)$,  $\mathcal L(X, Y)$ we denote the spaces of the linear bounded operators from $X$ to $X$, from $X$ to $Y$, respectively. 

Let $\mathcal B_b(X; Y)$ and $C_b(X; Y)$ denote the space of all bounded Borel measurable (resp. bounded  continuous)   functions $F:X\mapsto Y$, endowed with the sup norm $\|F \|_{\infty}:= \sup_{x\in X} \|F(x)\|_Y$. 
If $X=\R$ we set $\mathcal B_b(X; \R) = \mathcal B_b(X)$ and  $C_b(X; \R):= C_b(X)$. 

We use the standard notation for partial derivatives along elements of $X$: 
for any fixed $v$, $x\in X$ and $F:X\mapsto Y$, we say that $F$ is differentiable along $v$ at $x$ if there exists the limit
$\lim_{t\to 0} ( F(x+  t    v) - F(x))/t$. In this case the limit is denoted by $\partial   F    (x)/\partial v$.

In this paper we shall  consider spaces of functions that enjoy regularity properties only along certain directions. They are defined as follows. 

Let  $H\subset X$ be a Banach space. If $F:X\mapsto Y$ is differentiable along every $h\in H$, and the mapping  $h\mapsto \partial F/\partial h(x)$ belongs to  $\mathcal L(H, Y)$, $F$ is called $H$-Gateaux differentiable at $x$. Such a mapping is called $H$-Gateaux derivative of $  F    $ at $x$, and denoted by $D_HF(x)$.  
If in addition  
$$\lim_{\|h\|_H\to 0} \frac{\|F(x+h) - F(x) - D_HF(x)(h)\|_Y}{\|h\|_H} =0, $$
$F$ is called $H$-Fr\'echet differentiable at $x$.

Note that in  the case $H=X$, these are the usual notions of Gateaux and Fr\'echet differentiable functions at $x$.  

We shall consider also higher order derivatives. We identify ${\mathcal L}(H, {\mathcal L}(H,\R) )$ 
 with the space of the bilinear continuous functions from $H^2$ to $\R$;  more generally,  denoting by ${\mathcal L}^{n}(H)$ the space of all $n$-linear continuous mappings from $H^n$ to $\R$,  we identify ${\mathcal L}(H, {\mathcal L}^{n-1}(H))$ with ${\mathcal L}^{n}(H)$.

 Let $f:X\mapsto \R$ be  $H$-Gateaux (resp. $H$-Fr\'echet) differentiable at every $x\in X$.  If the mapping $D_Hf:X\mapsto {\mathcal L}(H,\R) $ is in  turn $H$-Gateaux (resp. $H$-Fr\'echet) differentiable at $x_0$, its $H$-Gateaux (resp. $H$-Fr\'echet) derivative belongs to ${\mathcal L}(H, {\mathcal L}(H,\R) )= {\mathcal L}^{2}(H)$ and is  denoted by $D^2_Hf(x_0)$. 
For  $n\in \N$, $n\geq 2$, $n$ times $H$-Gateaux (resp. $H$-Fr\'echet) differentiable functions are defined by recurrence. If $f:X\mapsto \R$ is $n-1$ times $H$-Gateaux (resp. $H$-Fr\'echet) differentiable at every $x\in X$, and the mapping $D^{n-1}_Hf:X\mapsto {\mathcal L}^{n-1}(H)$ is 
$H$-Gateaux (resp. $H$-Fr\'echet) differentiable at $x_0$, we say that $f$ is $n$ times $H$-Gateaux (resp. $H$-Fr\'echet) differentiable at $x_0$, and the derivative of $D^{n-1}_Hf$ at $x_0$ is denoted by $D^{n}_Hf(x_0)$. 

The space $C^n_H(X)$ consists of all continuous and bounded functions $f:X\mapsto \R$ having  $H$-Gateaux derivatives up to the order $n$,  such that for every $k=1, \ldots n$ and for every  $(h_1, \ldots, h_k)\in H^k$, the mapping $X\mapsto \R$, $x\mapsto D^k_Hf(x)(h_1, \ldots, h_k)$ is continuous and bounded.  It is endowed with the norm
$$\|f\|_{C^n_H(X)} = \|f\|_{\infty} + \sum_{k=1}^n \sup_{x\in X}\|D^k_Hf(x)\|_{ {\mathcal L}^{k}(H)}, $$
where, for every $n$-linear continuous function $T:H^k\mapsto \R$, 
$$\|T\|_{ {\mathcal L}^{k}(H)} := \sup\bigg\{ \frac{|T(h_1, \ldots, h_k)|}{\|h_1\|_H\cdots \|h_k\|_H}: \; h_i\in H\setminus \{0\} \bigg\}. $$
We note that by the multilinear version of the uniform boundedness principle (see \cite{SB}, \cite{B}) we have that if $f\in C^n_H(X)$, then indeed $\|f\|_{C^n_H(X)}<\infty$. Furthermore, we remark that if $f\in C^n_H(X)$, for every $x\in X$ and $h\in H$ the function $t\mapsto f(x+th)$ is in $C^n(\R)$, and we have
\begin{equation}
\label{teofondcalc}
\begin{array}{lll}
(i)  & \text{if}\;n=1, & f(x+h) - f(x) = \ds \int_0^1 D_Hf(x+\sigma h)(h)  \, d\sigma    , \;     \text{and therefore}
\\
\\
& &    |f(x+h) - f(x)| \leq  \sup_{y\in X}\|D_Hf(y)\|_{ H^*}\|h\|_H;   
\\
\\
(ii)&  \text{if}\;n>1, & (D^{n-1}_{  H   }f(x+h)- D^{n-1}_Hf(x))(h_1, \ldots, h_{n-1})  =  \ds\int_0^1 D^{n}_Hf(x+\sigma h)(h_1, \ldots, h_{n-1}, h)d\sigma. 
\end{array}
\end{equation}

For $\alpha\in (0, 1)$ we set
$$C^{\alpha}_H(X; Y) := \bigg\{F\in C_b(X; Y):  \, [F]_{C^{\alpha}_H(X, Y)} := \sup_{x\in X, \,h\in H\setminus\{0\}}\frac{\|F(x+h) - F(x)\|_Y}{\|h\|_H^{\alpha }} <+\infty \bigg\} $$
and we endow $C^{\alpha}_H(X; Y)$ with the norm
$$\|F\|_{C^{\alpha}_H(X; Y)} :=  \|F\|_{\infty} + [F]_{C^{\alpha}_H(X, Y)}$$
For $\alpha =1$, instead of Lipschitz continuity we shall consider a weaker condition, called Zygmund continuity. We set
$$\begin{array}{l}
Z^1_H(X, Y) := 
\\
\\
\ds \bigg\{F\in C_b(X; Y):  \, [F]_{Z^1_H(X, Y)} := \sup_{x\in X, \,h\in H\setminus \{0\}}\frac{\| F(x+2h)-2F(x+h) + F(x)\|_Y}{\|h\|_H} <+\infty \bigg\} \end{array}$$
and we endow $Z^1_H(X; Y)$ with the norm
$$\|F\|_{Z^1_H(X, Y)} :=  \sup_{x\in X}\|F(x)\|_Y + [F]_{Z^1_H(X, Y)} .  $$
If $H=X$ we drop the subindex $H$ and we use the more standard notations $C^{n}_b(X, Y)$ for $n\in \N$, $C^{\alpha}_b(X, Y)$ for $\alpha\in (0,1) $, $Z^1_b(X, Y)$.

If $Y=\R$, we set $C^{n}_H(X; \R) =: C^{n}_H(X)$ for $n\in \N$, $C^{\alpha}_H(X; \R) =: C^{\alpha}_H(X)$ for $\alpha\in (0,1) $, $Z^1_H(X; \R) =: Z^1_H(X)$. Higher order H\"older and Zygmund spaces of real valued functions will also be used; they are defined in a natural way, as follows. 
 
For $\alpha\in (0, 1)$ and $n\in \N$ we set
\begin{equation}
\label{def:higherHolder}
\begin{array}{c}
C^{\alpha +n}_H(X) := \{ f\in C^n_H(X): \: D^n_Hf \in C^{\alpha}_H(X, \mathcal L^n(H))\}, 
\\
\\
\|f\|_{C^{\alpha +n}_H(X)} := \|f\|_{C^{n}_H(X)} + [D^n_Hf]_{C^{\alpha}_H(X, \mathcal L^n(H))}
\end{array}
\end{equation}
and  for $n\in \N$, $n\geq 2$, 
\begin{equation}
\label{def:higherZygmund}
\begin{array}{c}
Z^{n}_H(X) := \{ f\in C^{n-1}_H(X): \: D^{n-1}_Hf \in Z^{1}_H(X, \mathcal L^{n-1}(H))\}, 
\\
\\
\|f\|_{Z^{ n}_H(X)} := \|f\|_{C^{n-1}_H(X):} + [D^{n-1}_Hf]_{Z^{1}_H(X, \mathcal L^{n-1}(H))}
\end{array}
\end{equation}

In the next lemma we collect some  properties of the above defined spaces, that are easy extensions of known properties in the case $H=X$, and that will be used later.

\begin{Lemma}
\label{Le:H-Holder}
Let $X$, $Y$ be Banach spaces.
\begin{itemize}
\item[(i)] For every $\alpha \in (0,1)$ and $F\in C^{1}_H(X; Y) $ we have
\begin{equation}
\label{Jalpha}
[F]_{C^{\alpha}_H(X; Y)} \leq 2^{1-\alpha }   \sup_{z\in X} \|D_HF(z)\|_{ \mathcal L(H,Y)}  ^{\alpha} \|F\|_{C_b(X;Y)}^{1-\alpha}. 
\end{equation}
\item[(ii)] If $F:X\mapsto Y$ is $H$-Gateaux differentiable and $D_HF$ is continuous at $x\in X$, then $F$ is $H$-Fr\'echet differentiable at $x$. 
\item[(iii)] If  $f\in C^2_H(X)$ we have 
 \begin{equation}
 \label{Zygmund2}
|f(x+2h) -2f(x+h)  +f(x)| \leq \sup_{y\in X} \|D^2_Hf(y)\|_{\mathcal L^2(H)} \|h\|^2_H, \quad  x\in X, \; h\in H. 
 \end{equation}
\end{itemize}
\end{Lemma}
\begin{proof}
Let  $F\in  C^{1}_H(X; Y)$. For every $x\in X$, $h\in H$, the function $\psi: \R\mapsto Y$,  $\psi(t) := F(x+th)$ is continuously differentiable, and $\psi'(t) = D_HF(x+th)(h)$. 
Therefore we have
\begin{equation}
\label{fond_calcolo}
F(x+h) - F(x) = \int_0^1 D_HF(x+ \sigma h)(h)d\sigma
\end{equation}
so that 
\begin{equation}\label{Zygmund3}
\|F(x+h) - F(x)\|_Y   \leq \sup_{z\in X}\|D_HF(z)\|_{\mathcal L(H, Y)}\|h\|_H. 
\end{equation}
Of course, we also have 
$$\|F(x+h) - F(x)\|_Y \leq 2 \sup_{z\in X}\| F(z)\|_{Y}. $$
Consequently, 
$$\|F(x+h) - F(x)\|_Y \leq (\sup_{z\in X}\|D_HF(z)\|_{\mathcal L(H, Y)}\|h\|_H)^{\alpha}  (2 \sup_{z\in X}\| F(z)\|_{Y})^{1-\alpha}$$
and statement (i) follows. 

Let us prove (ii). Using again \eqref{fond_calcolo} we get, for every $h\in H$, 
$$\begin{array}{lll}\|F(x+h) - F(x) - D_HF(x)(h)\|_Y &  = & \ds   \bigg\| \int_0^1 (D_HF(x+ \sigma h) -  D_HF(x))(h)d\sigma \bigg\|_Y
\\
\\
& \leq &  \sup_{v\in H, \, \|v\|_H \leq \|h\|_H } \|D_HF(x+v)-D_HF(x)\|_{\mathcal L(X, Y)}\|h\|_H 
\end{array}$$
so that, recalling that $H\subset X$, 
$\|F(x+h) - F(x) - D_HF(x)(h)\|_Y = o(\|h\|_H)$ as $\|h\|_H\to 0$, and $F$ is  $H$-Fr\'echet differentiable at $x$. 

Let now $f\in C^2_H(X)$. 
Applying thrice \eqref{fond_calcolo}, for every $x\in X$ and $h\in H$ we get 
$$\begin{array}{l}
f(x+2h) -2f(x+h)  +f(x) = \ds \int_0^1 (D_Hf(x+(1+\sigma) h)-D_Hf(x+ \sigma h))(h)d\sigma 
\\
\\
\ds = \int_0^1\int_0^1 D^2_Hf(x+(\tau + \sigma)h)(h,h)d\tau \,d\sigma\end{array}$$
and estimating in an obvious way statement (iii) follows.  
\end{proof}

A Borel probability measure $\mu$  in $X$ is called {\em Fomin differentiable} along $v\in X$ if for every Borel set $A$ the incremental ratio 
$ (\mu (A+tv) - \mu(A))/t  $ has finite limit as $t\to 0$. Such a limit is called $d_v\mu(A)$; $d_v\mu$ is a signed measure and denoting the translated measure $\mu_v(A):= \mu(A+v)$ \ by $\mu_v$  we have  
\begin{equation}
\label{misura_limite}
\lim_{s\to 0}\left\| \frac{\mu_{sv}- \mu}{s} - d_v\mu\right\| =0, 
\end{equation}
where $\|\cdot\|$ denotes the total variation norm. 

Moreover, $d_v\mu$ is absolutely continuous with respect to $\mu$. The density $\beta^{\mu}_v\in L^1(X, \mu)$ is called {\em Fomin derivative} or {\em logarithmic derivative} of $\mu$ along $v$, and it satisfies
\begin{equation}
\label{Fomin}
\int_X \frac{\partial f}{\partial v}\,\mu(dx) = - \int_X \beta^{\mu}_vf\,\mu(dx), \quad f\in C^1_b(X). 
\end{equation}
 By \cite[Thm. 3.6.8]{BogaDiff}, this equality characterizes  Fomin differentiability, in the sense that if \eqref{Fomin} holds for some $\beta^{\mu}_v\in L^1 (X, \mu)$ and for every $f\in C^1_b(X)$, then $\mu$  is Fomin differentiable along $v$.

If $\mu$ is Fomin differentiable along two 
vectors $v$, $w$, then it is Fomin differentiable along any linear combination of $v$ and $w$, and we have $d_{\lambda_1 v + \lambda_2w}\mu = \lambda_1 d_v\mu + \lambda_2 d_w\mu$; therefore
\begin{equation}
\label{somma}
 \beta^{\mu}_{\lambda_1 v + \lambda_2w}= \lambda_1 \beta^{\mu}_{v} + \lambda_2  \beta^{\mu}_{w}, \quad \lambda_1, \;\lambda_2\in \R .
\end{equation}

The proofs of these statements may be found in  \cite[Chapter 3]{BogaDiff}. We refer to  \cite{BogaDiff} for the general theory of differentiable measures    $\mu$, and to the basic properties of the measures $d_v\mu$.     
 
 
 \section{Schauder and Zygmund regularity}\label{sect:schauzyg}
 

Under the only assumptions that $T_t\in {\mathcal L}(X)$ and $\mu_t$ is a Borel probability measure for every $t$, the operators $P_t$ defined in \eqref{P_t} map $C_b(X)$ into itself and we have
\begin{equation}
\label{contrazione}
\|P_tf\|_{\infty} \leq \|f\|_{\infty}, \quad t>0, \; f\in C_b(X). 
\end{equation}
The weak continuity of $t\mapsto \mu_t$ yields that for every $ f\in C_b(X)$ the function $ [0, +\infty)\times X$, $(t,x)\mapsto P_tf(x)$ is continuous, by \cite[Lemma 2.1]{BRS}. 
Consequently, the operators $F_\lambda$ in right-hand side of \eqref{L}  are one to one, and since $P_t$ is a semigroup they satisfy the resolvent identity $F(\lambda ) - F(\mu) = (\mu - \lambda)F(\lambda) F(\mu)$. By the general spectral theory, there exists a unique closed operator $L:D(L) \subset  C_b(X)\mapsto  C_b(X)$ such that $F(\lambda) = R(\lambda, L)$. The domain $D(L)$ is just the range of $F(\lambda )$, for every $\lambda >0$.  

The leading assumptions of the paper are the following. 
 
\begin{Hypothesis}
\label{Hyp1}
 For every $t>0$ there exists a subspace $\{0\} \neq H_t\subset X$ such that $\mu_t$ is Fomin differentiable along every $h\in H_t$. 
 \end{Hypothesis}
 
According to the notation of Section \ref{sect:notation}, for every $v\in H_t$ we denote by $\beta^{\mu_t}_{v}$ the Fomin derivative of $\mu_t$ along $v$.

\begin{Hypothesis}
\label{Hyp2}
There exists a Banach space $H \subset X$, and constants $M$, $C$, $\theta>0$, $\omega \in \R$ such that 
 \begin{equation}
 \label{fundamental}
\left\{\begin{array}{ll} 
(i) &   T_t (H)\subset H \; ,  \;\|T_th\|_H \leq Me^{\omega t}\|h\|_H, \quad t>0, \;h\in H, 
\\
\\
(ii) &  T_t(H)\subset H_t,\;  \|\beta^{\mu_t}_{T_th}\|_{L^1(X, \mu_t)} \leq \ds \frac{Ce^{\omega t}}{t^{\theta}}\|h\|_H, \quad t>0, \; h\in H. 
\end{array}\right. 
\end{equation}
\end{Hypothesis}

\subsection{Properties of $P_t$ and estimates}

The starting point of our analysis is the next proposition, which shows that each $P_t$ is  smoothing along suitable directions. 

\begin{Proposition}
\label{representation}
\begin{itemize}
\item[(i)] Let $g\in \mathcal{B}_b(X)$ and $t>0$. Then $P_tg$ is $H$-Gateaux differentiable with bounded H-Gateaux derivative, and 
\begin{equation}
\label{hderivative0}
D_H P_tg(x)(h) = - \int_X g(T_tx + y)\beta ^{\mu_t}_{T_th}(y)\,\mu_t(dy), \quad t>0, \; h\in H. 
\end{equation}
\item[(ii)]
 Let $g\in\ C_b(X)$ and $t>0$, $n\in\N$. Then $P_tg\in C^n_H(X)$; for all $h_1,...,h_n\in H$ we have
\begin{align}\label{ident1}
&D^n_HP_tg(x)(h_1,...,h_n)\\ \notag
&=(-1)^n\int_X\cdot\cdot\cdot\int_X g(T_tx+T_{\frac{n-1}{n}t}y_1+\cdot\cdot\cdot+T_{\frac t n}y_{n-1}+y_n)\\
&\quad\quad\quad\quad\beta^{\mu_{t/n}}_{T_{t/n}h_n}(y_n)\cdot\cdot\cdot\beta^{\mu_{t/n}}_{T_{t/n}h_1}(y_1)\mu_{t/n}(dy_n)\cdot\cdot\cdot\mu_{t/n}(dy_1), \notag
\end{align}
 and 
\begin{equation}\label{nth-derivatives}
\|D_H^nP_tg(x)\|_{\mathcal{L}^n(H)}\leq K_n\frac{e^{\omega t}}{t^{n\theta}}\|g\|_\infty , 
\end{equation}
with $K_n:= C^nn^\theta$. 
If $\omega >0$  a better  estimate than \eqref{nth-derivatives} holds for large $t$, namely there exists  $K_n' >0$ such that 
\begin{equation}
\label{nth-derivatives1}
\|D^n_HP_tg(x)\|_{\mathcal L^n(H)} \leq K_n' \max\{ 1,  t^{-n\theta} \} \|g\|_{\infty}, \quad t>0, \; x\in X, \; g\in C_b(X). 
\end{equation}
\item[(iii)]
Let  $g\in C^1_H(X)$ and $t>0$. Then
\begin{equation}
\label{hderivative1}
D_HP_tg(x)(h) = \int_X \frac{\partial g}{\partial (T_th)}(T_tx + y) \mu_t(dy) = P_t\bigg( \frac{\partial g}{\partial (T_th)}\bigg) (x), \quad t>0, \; x\in X, \;h\in H. 
\end{equation}
If even $g\in C^n_H(X)$ for some $n\in \N$, then for all $t>0$, $x\in X$ and $h_j\in H$, $j=1, \ldots, n$, we have 
\begin{equation}
\label{hderivativen}
D_H^nP_tg(x)(h_1, \ldots, h_n) =  \int_X D^n_Hg(T_tx+y)(T_th_1, \ldots, T_th_n)\mu_t(dy)
=P_t\bigg( \frac{\partial ^ng}{\partial (T_th_n) \ldots \partial (T_th_1)}\bigg) (x), 
\end{equation}
and the function $(t,x)\mapsto D_H^nP_tg(x)(h_1, \ldots, h_n)$ is continuous in $[0, +\infty)\times X$. Moreover, 
\begin{equation}
\label{kappa-kappa}
\|P_t\|_{\mathcal L (C^n_H(X))}\leq \max\{1,  M^ne^{n\omega t} \}, \quad t>0, \; n\in \N. 
\end{equation}
\end{itemize}
\end{Proposition}
\begin{proof} 
(i) For every $x\in X$, $h\in H$ and $s\neq 0$ we have 
$$\begin{array}{lll}
\ds \frac{P_tg(x+sh) - P_tg(x)}{s} & = & \ds  \frac{1}{s} \int_X (g(T_tx + sT_th+ y) -  g(T_tx +  y))\mu_t(dy) 
\\
\\
& = &
\ds  \frac{1}{s} \left( \int_X  g(T_tx +z)( \mu_t)_{sT_th}(dz) - 
\int_X  g(T_tx +  y)\mu_t(dy)\right)\end{array}$$
so that 
$$ \left| \frac{P_tg(x+sh) - P_tg(x)}{s} + \int_X g(T_tx + y)\beta ^{\mu_t}_{T_th}(y)\,\mu_t(dy) \right|
\leq \|g\|_{\infty} \left\| \frac{ (\mu_t)_{sT_th} - \mu_t}{s} - d_{T_th}\mu_t\right\| 
 $$
that vanishes as $s\to 0$ by \eqref{misura_limite}. Therefore, $P_tg$ is differentiable along $h$ at $x$, with derivative $\partial P_tg(x)/\partial h $ given by the right-hand side of \eqref{hderivative0}.  Such a derivative is linear in $h$ by \eqref{somma} and by the linearity of $T_t$, and  by Hypothesis \ref{Hyp2}(ii) it modulus is bounded by 
$\|g\|_{\infty}Ce^{\omega t}t^{-\theta}\|h\|_H$. Therefore, $P_tg$ is $H$-Gateaux differentiable  at $x$ and \eqref{hderivative0} holds. If $g\in C_b(X)$, 
then  for every $x$, $x_0\in X$, and $h\in H$, 
\begin{equation}
\label{Gateaux}
 |(D_HP_tg(x) - D_HP_tg(x_0) )(h)| \leq \int_X | g(T_tx + y)-  g(T_tx_0 + y)|\,| \beta ^{\mu_t}_{T_th}(y)|\,\mu_t(dy)
\end{equation}
where the right-hand side vanishes as $x\to x_0$ by the Dominated Convergence Theorem. So, $D_HP_tg(\cdot)h$ is continuous on $X$, hence $P_tg\in C^1_H(X)$. 
 
\vspace{3mm}
 
(ii) Now let us prove \eqref{ident1} for $g\in C_b(X)$, $t>0$, by induction over $n\in\N$. We have just proved \eqref{ident1} for $n=1$ above. Suppose that \eqref{ident1} holds for $n\in\N$. 
By  the induction hypothesis applied to the $n$-step equipartition $0<\frac{1}{n+1}<\cdot\cdot\cdot<\frac{n}{n+1}$ of $[0,\frac{n}{n+1}]$ for $h_1,...,h_{n+1}\in H$ we have
$$\begin{array}{l}
D_H^{n}P_tg(x)(h_1,...,h_{n}) = D_H^nP_{\frac{n}{n+1}t}(P_{\frac{t}{n+1}}g)(x)(h_1,...,h_n)
\\
\\
\ds = (-1)^n\int_X\cdot\cdot\cdot\int_X P_{\frac{t}{n+1}}g(T_{\frac{n}{n+1}t}x+T_{\frac{n-1}{n+1}t}y_1+\cdot\cdot\cdot T_{\frac{t}{n+1}}y_{n-1}+y_n)\\
\\
\ds \quad\quad\quad\quad \beta^{\mu_{t/(n+1)}}_{T_{t/(n+1)}h_n}(y_n)\cdot\cdot\cdot\beta^{\mu_{t/(n+1)}}_{T_{t/(n+1)}h_1}(y_1)\mu_{t/(n+1)}(dy_n)\cdot\cdot\cdot\mu_{t/(n+1)}(dy_1).
\end{array}$$
Since we already know that $P_{\frac{t}{n+1}}g\in C_H^1(X)$, by Hypothesis \ref{Hyp2}(ii), \eqref{Zygmund3} and the Dominated Convergence Theorem
we can differentiate the right-hand side along $h_{n+1}$ interchanging the partial derivative with the multiple integrals, and using \eqref{hderivative0} we obtain
\begin{align*}
& \frac{\partial }{\partial h_{n+1}} D_H^{n}P_tg(x)(h_1,...,h_{n}) =
(-1)^{(n+1)}\int_X\cdot\cdot\cdot\int_X g(T_{t}x+T_{\frac{nt}{n+1}}y_1+\cdot\cdot\cdot+T_{\frac{t}{n+1}}y_{n}+z)\\
&\beta^{\mu_{t/(n+1)}}_{T_{t/(n+1)}h_{n+1}}(z)\beta^{\mu_{t/(n+1)}}_{T_{t/(n+1)}h_n}(y_n)\cdot\cdot\cdot\beta^{\mu_{t/(n+1)}}_{T_{t/(n+1)}h_1}(y_1)\mu_{t/(n+1)}(dz)\mu_{t/(n+1)}(dy_n)\cdot\cdot\cdot\mu_{t/(n+1)}(dy_1).
\end{align*}
The right-hand side is just $D_H^{  n+1   }P_tg(x)(h_1,...,h_{n+1})$, so that \eqref{ident1} holds for $n+1$. 

The continuity and boundedness on $X$ of the map $x\mapsto D_H^nP_tg(x)(h_1,...,h_{n})$ is obvious by \eqref{ident1}, Hypothesis \ref{Hyp2}(ii) and the Dominated Convergence Theorem. Then also \eqref{nth-derivatives} follows immediately by Hypothesis \ref{Hyp2}(ii).

Assume now that $\omega >0$. Using  \eqref{nth-derivatives}  we get for $0<t\leq 2$
$$\|D^n_HP_tg(x)\|_{\mathcal L^n(H)} \leq K_n \frac{e^{2  \omega }}{t^{n\theta}}\|g\|_{\infty},  $$
while for $t\geq 2$, writing $D^n_HP_tg= D^n_HP_1(P_{t-1}g)$ and using \eqref{contrazione} and  \eqref{nth-derivatives} with $t=1$, we get 
$$\|D^n_HP_tg(x)\|_{\mathcal L^n(H)}    \leq K_n  e^{ \omega } \|P_{t-1}g\|_{\infty} \leq K_n  e^{  \omega } \|g\|_{\infty}. $$
Putting together such estimates, we get \eqref{nth-derivatives1}.
 
 \vspace{3mm}
 
(iii) Now we prove \eqref{hderivative1}. If $g\in C^1_H(X)$, for every $s\neq 0$ and $x\in X$, as before, 
$$\frac{P_tg(x+sh) - P_tg(x)}{s} = \int_X \frac{ g(T_tx + sT_th+ y) -  g(T_tx +  y)}{s}\,\mu_t(dy)$$
and  the right-hand side converges to $\int_X \frac{\partial g}{\partial T_th}(T_tx + y) \mu_t(dy) = \int_X D_H g(T_tx + y)(h) \mu_t(dy) $ as $s\to 0$, by \eqref{Zygmund3} and the Dominated Convergence Theorem. By the definition of $P_t$, such limit coincides with $ P_t (  \partial g/\partial (T_th) ) (x) $.  

If $g\in C^n_H(X)$, formula \eqref{hderivativen} follows applying several times \eqref{hderivative1}. The proof of the continuity of 
$(t,x)\mapsto D_H^nP_{t}g(x)(h_1, \ldots, h_n)$ is similar to the proof of the continuity of $(t,x)\mapsto P_tf(x)$ of \cite[Lemma 2.1]{BRS}. Here is the argument: 

 Let $t_k\to t \in [0, +\infty)$, $x_k\to x\in X$. Then, 
$$\begin{array}{l}
D_H^nP_{t_k}g(x_k)(h_1, \ldots, h_n) - D_H^nP_{t}g(x)(h_1, \ldots, h_n) 
\\
\\
= \ds  \int_X (D^n_Hg(T_{t_k}x_k +y)(T_{t_k}h_1, \ldots, T_{t_k}h_n) - D^n_Hg(T_{t}x+y)(T_{t}h_1, \ldots, T_{t}h_n)) \mu_{t_k}(dy)
\\
\\ + \ds  \int_X  D^n_Hg(T_{t}x+y)(T_{t}h_1, \ldots, T_{t}h_n)) ( \mu_{t_k} (dy) - \mu_{t}(dy))=: I_{1,k}+I_{2,k}\end{array}$$
Since $\mu_{t_k} $ weakly converges to $\mu_t$ as $k\to \infty$ and $ D^n_Hg(T_{t}x+\cdot )(T_{t}h_1, \ldots, T_{t}h_n)$ is continuous and bounded, $I_{2,k}\to 0$ as $k\to \infty$. Still by the weak convergence, the measures $\mu_{t_k}$ are uniformly tight, namely for every $\varepsilon >0$ there is a compact set $K_{\varepsilon}\subset X$ such that $\mu_{t_k}(X\setminus K_{\varepsilon}) \leq \varepsilon$ for every $k\in \N$. Splitting $I_{1,k}$ into the sum of the integral over $K$ and the integral over $X\setminus K_{\varepsilon}$,  and using the uniform continuity of $(t,z)\mapsto D^n_Hg(z)(T_{t}h_1, \ldots, T_{t}h_n)$ on compact sets, one   gets  $\lim_{k\to \infty} I_{1,k} =0$, too.  

By
\eqref{hderivativen} and   Hypothesis \ref{Hyp2}(i) we have, for every natural number $j\leq n$ and $x\in X$, $h_1, \ldots h_j\in H$,  
$$|(D^j_HP_tg(x))(h_1, \ldots , h_j)| \leq \sup_{y\in X}\|D^j_H  g    (y)\|_{\mathcal L^j(H)}\prod_{l=1}^j\|T_th_l\|_H \leq M^je^{j\omega t} \prod_{l=1}^j\|h_l\|_H   \sup_{y\in X}\|D^j_H g (y)\|_{\mathcal L^j(H)},    $$
which yields \eqref{kappa-kappa}. 
\end{proof}

\begin{Remark}
\label{Rem:Frechet}
\em Under our general assumptions we cannot prove that  $D_HP_tg$ is continuous with values in $H^*$ (and therefore that $P_tg$ is $H$-Fr\'echet differentiable, by Lemma \ref{Le:H-Holder}(ii))   for every $t>0$ and $g\in C_b(X)$. 
\eqref{Gateaux}  implies immediately that $D_HP_tg$ is continuous  for every uniformly continuous and bounded $g$, but we prefer to deal with merely continuous rather than uniformly continuous functions. 

 If in addition the functions $\beta ^{\mu_t}_{T_th}$ belong to $L^p(X, \mu_t)$ for some $p>1$, and  for every $t>0$ there exists $C_t>0$ such that 
$\| \beta ^{\mu_t}_{T_th}\|_{L^p(X, \mu_t)} \leq C_t\|h\|_H$ for every  $h\in H$, 
using the H\"older inequality in the right-hand side of  \eqref{Gateaux} and then the Dominated Convergence Theorem yields that $D_HP_tg$ is continuous with values in $H^*$. 
In this case, throughout the paper  we could  use stronger higher order H\"older and Zygmund spaces,  obtained replacing the condition of $H$-Gateaux differentiability by $H$-Fr\'echet differentiability in the definition of the $C^n_H$ spaces. 
\end{Remark}

The behavior of $P_t$ in the H\"older spaces $C^{\alpha}_H(X)$ and in the Zygmund spaces $Z^{k}_H(X)$ is coherent with its behavior in $C_b(X)$, as the next lemma shows.

\begin{Lemma}
\label{Le:sgrHolderZygmund}
For every $t>0$ and $\alpha \in (0, +\infty)$, $k\in \N \cup \{0\}$, $P_t\in {\mathcal L}( C^{k+\alpha }_H(X))$ and there exists $c=c(k+\alpha)>0$ such that 
\begin{equation}
\label{alpha-alpha}
\|P_tf\|_{C^{k+\alpha}_H(X)} \leq c\|f\|_{C^{k+\alpha}_H(X)} , \quad t>0, \; f\in  C^{k+\alpha}_H(X). 
\end{equation}
Moreover, for every $t>0$ and $k\in \N$ , $P_t\in {\mathcal L}( Z^{k}_H(X))$ and there exists $c=c(k)>0$ such that 
\begin{equation}
\label{Zygmund-Zygmund}
\|P_tf\|_{Z^{k}_H(X)} \leq c \|f\|_{Z^{k}_H(X)} , \quad t>0, \; f\in  Z^{k}_H(X). 
\end{equation}
\end{Lemma}
\begin{proof}
Let  $\alpha \in (0,1)$ and $f\in C^{\alpha}_H(X)$, $t>0$.  
From  the representation formula \eqref{P_t} and   Hypothesis \ref{Hyp2}(i) we get, for every $x\in X$ and $h\in H$, 
\begin{equation}
\label{k=0}
\begin{array}{lll}
|P_tf(x+h) - P_tf(x)| & = & \ds \bigg| \int_X (f(T_tx + T_th +y) - f(T_tx + y))\mu_t(dy) \bigg|
\\
\\
& \leq & [f]_{C^\alpha_H(X)}\|T_th\|_H^{\alpha}  \leq M^{\alpha}e^{ \alpha \omega t}[f]_{C^\alpha_H(X)}\|h\|_H^{\alpha}, 
\end{array}
\end{equation}
so that 
$$\|P_tf\|_{C^\alpha_H(X)} \leq \|f\|_{\infty} + M^{\alpha}e^{ \alpha \omega t}[f]_{C^\alpha_H(X)}. $$
This proves \eqref{alpha-alpha} for $k=0$, in the case $\omega \leq 0$. 

 If $f\in C^{k+\alpha }_H(X)$ for some $k\in \N$, we use \eqref{hderivativen} and again  Hypothesis \ref{Hyp2}(i), that give, for every $x\in X$ and $h$, $h_1, \ldots h_k\in H$, 
$$\begin{array}{l}
|(D^k_HP_tf(x+h) - D^k_HP_tf(x))(h_1, \ldots , h_n)| =
\\
\\
=  \ds \bigg| \int_X (D^k_Hf( T_tx + T_th +y) - D^k_Hf(T_tx + y))(T_th_1, \ldots , T_th_k)\mu_t(dy) \bigg|
\\
\\
\leq  [D^k_Hf]_{C^\alpha_H(X; \mathcal L^k(H))}\|T_th\|_H^{\alpha} \prod_{j=1}^k\|T_th_j\|_H \leq M^{k+\alpha}e^{(k+ \alpha )\omega t}[D^k_Hf]_{C^\alpha_H(X; \mathcal L^k(H))}\|h\|_H^{\alpha}\prod_{j=1}^k\|h_j\|_H. 
\end{array}$$
This estimate and \eqref{kappa-kappa} yield  \eqref{alpha-alpha} for $k\in \N$, in the case $\omega \leq 0$. 

For $\omega >0$ we argue as follows. For every $f\in C^{k+\alpha}(X)$ with $\alpha\in (0,1)$ and $k\in \N \cup\{0\}$ the above estimates yield
$$\|P_t f\|_{C^{k+\alpha}(X)} \leq M^{k+\alpha}e^{(k+\alpha)\omega}\| f\|_{C^{k+\alpha}(X)} , \quad  t\leq 1, $$
while for $t>1$ we write $P_tf = P_1P_{t-1}f$,  so that $\|P_t f\|_{C^{k+\alpha}_H(X) } \leq 
\|P_1\|_{\mathcal L (C_b(X), C^{k+\alpha}_H(X))}\|f\|_{\infty}$ by \eqref{contrazione} ($P_1$ belongs to $\mathcal L (C_b(X), C^{k+\alpha}_H(X))$ because it belongs to $\mathcal L (C_b(X), C^{k+1}_H(X))$ by Proposition  \ref{representation}, and $ C^{k+1}_H(X) \subset C^{k+\alpha}_H(X)$).

The proof of estimates  \eqref{Zygmund-Zygmund}  is similar, and it is left to the reader.
\end{proof}

If $f$ is $H$-H\"older continuous  estimates \eqref{nth-derivatives} may be improved near $t=0$. 
Such improvements are crucial in the proof of our Schauder theorems. 

\begin{Proposition}
\label{stimesgrHolder}
For every  $\alpha\in (0, 1)$ and $n\in \N$ there are constants $K_{n, \alpha}>0$ such that   
\begin{equation}
\label{nth-derivativesHolder}
\|D^n_HP_tf(x)\|_{\mathcal L^n(H)} \leq K_{n, \alpha} \frac{e^{  \omega t}}{t^{(n-\alpha)\theta}} [f]_{C^\alpha_H(X)}, \quad t>0, \; x\in X, \;   f\in C^{\alpha}_H(X). \end{equation}
\end{Proposition}
\begin{proof}
The key step is to prove that  \eqref{nth-derivativesHolder} holds for $n=1$. 
We use the same argument of \cite{CL}.   Let $t>0$, $f\in C^{\alpha}_H(X)$, $h\in H\setminus \{0\}$. For every $s>0 $ we have 
 $$\begin{array}{lll}
 D_H P_tf(x)(h)  & = & \ds  \left( D_H P_tf(x)(h)- \frac{P_tf(x+sh) - P_tf(x)}{s}\right) +
   \frac{P_tf(x+sh) - P_tf(x)}{s} 
 \\
 \\
&  = & \ds \bigg( \frac{1}{s} \int_0^s (D_HP_tf(x)(h) - D_H P_tf(x+ \sigma h)(h)) d\sigma  \bigg)
 +   \frac{P_tf(x+sh) - P_tf(x)}{s} 
\\
\\
& =: & I_1(s)+I_2(s). \end{array}$$
To estimate $ I_1(s) $ we remark that  for every $k\in H$, by \eqref{hderivative0} we have
$$\begin{array}{l}
|(D_HP_tf(x+k) - D_HP_tf(x))(h)| = \ds \left| \int_X (f(T_tx + T_tk +y) - f(T_tx + y))\beta^{\mu_t}_{T_th}(y) \mu_t(dy) \right|
\\
\\
  \leq   [f]_{C^\alpha_H(X)} \|T_tk\|_H^{\alpha}  \|  \beta^{\mu_t}_{T_th}\|_{L^1(X, \mu_t)}
\leq 
[f]_{C^\alpha_H(X)}M^{\alpha}e^{ \alpha \omega t}\|k\|_H^{\alpha}\ds  \frac{C e^{\omega t}}{t^{\theta}}\|h\|_H. 
\end{array}$$
Using this estimate with $k=\sigma h$ we get 
  $$ |  I_1(s) |   \leq  \frac{1}{s} \int_0^s |D_H P_tf(x+ \sigma h)(h) - D_HP_tf(x)(h)| d\sigma   
\leq  \frac{1}{s}  \frac{CM^{\alpha}}{t^{\theta}} e^{( \alpha +1) \omega t} \int_0^s  
 \sigma^{\alpha} d\sigma \,\|h\|_H^{\alpha +1} 
[f]_{C^\alpha_H(X)}.  $$

%
On the other hand,  by \eqref{k=0}  we get  
$$|  I_2(s) | \leq M^{\alpha} e^{ \alpha   \omega t}s^{\alpha -1}  \|h\|_H^{\alpha }[f]_{C^\alpha_H(X)} . $$
Summing up, 
$$ |D_H P_tf(x)(h)| \leq \bigg(  \frac{CM^{\alpha}}{\alpha +1}  e^{( \alpha +1) \omega t}\|h\|_H^{\alpha +1} \frac{s^{\alpha}}{t^\theta} 
+ M^{\alpha} e^{ \alpha   \omega t}s^{\alpha -1}  \|h\|_H^{\alpha } \bigg) [f]_{C^\alpha_H(X)} , \quad s>0. $$
Choosing now $s= t^{\theta}e^{-\omega t}/ \|h\|_H$ we get 
$$ |D_H P_tf(x)(h)| \leq \bigg( \frac{CM^{\alpha}}{\alpha +1} + M^{\alpha}\bigg) \frac{1}{t^{(1-\alpha)\theta}}  e^{\omega t}  \|h\|_H  [f]_{C^\alpha_H(X)} , $$
which yields  \eqref{nth-derivativesHolder}  for $n=1$. 

For $n>1$ we have $D^n_HP_tf = D^n_HP_{t/2}g$, with $g= P_{t/2}f$. By \eqref{hderivative1}, 
$$ \frac{\partial ^n P_tf}{\partial h_n\cdots \partial h_1} =  \frac{\partial ^{n-1} }{\partial h_{n-1}\cdots \partial h_1}P_{t/2}\left(  \frac{\partial  P_{t/2}f}{\partial (T_{t/2}h_n)} \right)$$
so that  \eqref{nth-derivativesHolder} follows from  \eqref{nth-derivatives}  and  \eqref{nth-derivativesHolder}  with $n=1$. 
\end{proof}

 
  \subsection{Schauder and Zygmund estimates: stationary equations}
 

 In this section we use the smoothing properties of $P_t$ to deduce regularity results for the elements of $D(L)$, namely for the functions $u$ given by   \eqref{L} for some $\lambda >0$ and $f\in C_b(X)$. Estimate \eqref{contrazione}   yields immediately  
 \begin{equation}
 \label{stimasup}
 \|u\|_{\infty} \leq \frac{1}{\lambda} \|f\|_{\infty}. 
 \end{equation}
  
The first (not optimal) regularity result is a standard consequence of Propositions \ref{representation} and \ref{stimesgrHolder}.

 \begin{Proposition}
 \label{Pr:nonoptimal}
Given $\lambda >0$ and $f\in C_b(X)$, let $u=R(\lambda, L)f$.  
 \begin{itemize}
 \item[(i)] Let $\theta <1$. For every $n\in \N$ such that $n<1/\theta$, $u\in C^n_H(X)$. There exists $C=C(\lambda)>0$, independent of $f$, such that 
 \begin{equation}
 \label{C^n0}
 \|u\|_{C^n_H(X)}\leq C\|f\|_{\infty}
 \end{equation}
  \item[(ii)] Let $\alpha\in (0,1)$ be such that $\alpha + 1/\theta >1$. For every $f\in C^{\alpha}_H(X)$ and for every $n\in \N$ such that $n<\alpha +1/\theta$, $u\in C^n_H(X)$. There exists $C= C(\lambda, \alpha)>0$, independent of $f$, such that 
 \begin{equation}
 \label{C^nalpha}
 \|u\|_{C^n_H(X)}\leq C\|f\|_{C^{\alpha}_H(X)}
 \end{equation}
\end{itemize}
\end{Proposition}
\begin{proof}
The proof is in two steps. First we consider the case $\lambda >\omega$, and then, if $\omega >0$, the case $\lambda\in (0,  \omega]$.

\noindent 
{\em First step}: $\lambda > \omega$. 
Estimate \eqref{nth-derivatives}  yields, for every $k\in \{1, \ldots, n\}$, 
 \begin{equation}
 \label{stima_integrando0}  e^{-\lambda t}\|D^k_HP_tf(x)\|_{\mathcal L^k(H)} \leq   e^{-\lambda  t}
  K_{k} \frac{e^{  \omega t}}{t^{k\theta}} \|f\|_{ \infty}, \quad t>0, \; x\in X, \; f\in C_b(X), 
   \end{equation}
and if $\alpha \in (0, 1)$,  \eqref{nth-derivativesHolder} yields, for every $k\in \{1, \ldots, n\}$, 
 \begin{equation}
 \label{stima_integrando}  e^{-\lambda t}\|D^k_HP_tf(x)\|_{\mathcal L^k(H)} \leq   e^{-\lambda t}
  K_{k, \alpha} \frac{e^{  \omega t}}{t^{(k-\alpha)\theta}} [f]_{ C^{\alpha }_H(X)}, \quad t>0, \; x\in X,\;  f\in C^{\alpha}_H(X). 
   \end{equation}
The right-hand sides of \eqref{stima_integrando0} and \eqref{stima_integrando} belong to $L^1(0, +\infty)$ because $    \lambda> \omega$,  and $k\theta \in (0, 1)$ in \eqref{stima_integrando0}, $ (k-\alpha)\theta \in (0, 1)$ in \eqref{stima_integrando}. Therefore $u$ is $n$ times $H$-Gateaux differentiable at every $x\in X$, and for every $h_1, \ldots , h_k\in H$ with $k\in \{1, \ldots, n\}$  we have
$$D^k_Hu(x)(h_1, \ldots , h_k) = \int_0^{\infty}  e^{-\lambda t}D^k_HP_tf(x)(h_1, \ldots , h_k)dt. $$
\eqref{stima_integrando0} and  \eqref{stima_integrando} imply respectively, for every $x\in X$ and $k\in \{1, \ldots, n\}$, 
 $$\|D^k_Hu(x)\|_{\mathcal L^k(H)} \leq  K_{k}\int_0^{\infty}  \frac{e^{  \omega t -\lambda t}}{t^{k\theta}} dt \, \|f\|_{\infty}
 \leq  \frac{ K_{k} \Gamma( 1- k\theta)}{(\lambda - \omega )^{ 1- k\theta} }\|f\|_{\infty} ,  $$
and 
 $$\|D^k_Hu(x)\|_{\mathcal L^k(H)} \leq  K_{k, \alpha}\int_0^{\infty}  \frac{e^{  \omega t -\lambda t}}{t^{(k-\alpha)\theta}} dt \, [f]_{ C^{\alpha }_H(X)}
 \leq  \frac{K_{k, \alpha} \Gamma( 1- (k-\alpha)\theta)}{(\lambda- \omega )^{ 1- (k-\alpha)\theta} } [f]_{ C^{\alpha }_H(X)},  $$
for $\alpha >0$, $f\in  C^{\alpha }_H(X)$ (here, $\Gamma$ is the Euler function). 
In both cases, since for every $t>0$ the function $x\mapsto D^k_HP_tf(x)(h_1, \ldots , h_k)$ is continuous  by Proposition \ref{representation}, estimates \eqref{stima_integrando0},  \eqref{stima_integrando} and the Dominated Convergence Theorem imply that $x\mapsto D^k_Hu(x)(h_1, \ldots , h_k)$ is continuous   for $k=1, \ldots, n$. 
Therefore, $u\in C^n_H(X)$ and 
\begin{equation}
 \label{stimaC^n0}
 \|u\|_{C^n_H(X)} 
 \leq \frac{\|f\|_{\infty} }{  \lambda   }+ \bigg(\sum_{k=1}^n  K_{k } \frac{ \Gamma( 1- k\theta)}{(\lambda - \omega )^{ 1- k\theta} } \bigg)\|f\|_{\infty}, 
 \end{equation}
so that \eqref{C^n0} holds with $C= 1 + \sum_{k=1}^n  K_{k }   \Gamma( 1- k\theta)/(\lambda- \omega )^{ 1- k\theta} $. 
In the case that  $\alpha\in (0,1)$ and $f\in C^{\alpha}_H(X)$,  we get
\begin{equation}
 \label{stimaC^n}
 \|u\|_{C^n_H(X)} 
 \leq  \frac{\|f\|_{\infty} }{  \lambda   } + \bigg(\sum_{k=1}^n  K_{k, \alpha} \frac{ \Gamma( 1- (k-\alpha)\theta)}{(\lambda- \omega )^{ 1- (k-\alpha)\theta} }\bigg) [f]_{ C^{\alpha }_H(X)}, 
 \end{equation}
so that \eqref{C^nalpha} holds with $C= 1  /\lambda    + \sum_{k=1}^n  K_{k ,\alpha}   \Gamma( 1- (k-\alpha)\theta)/(\lambda- \omega )^{ 1- (k-\alpha)\theta} $. 
 
\vspace{2mm}
  
\noindent {\em Second step}: $\omega >0$, $\lambda \in (0, \omega]$. 

In this case the statement follows from Step 1 by a perturbation argument. Indeed, since $\lambda u - Lu = f$, 
we have $(\omega +1)u - Lu = (\omega +1 -\lambda )u + f$. The right-hand side belongs to $C_b(X)$, and its sup norm is bounded by $((\omega +1 -\lambda )/\lambda +1)\|f\|_{\infty}$, by \eqref{stimasup}. So, statement (i) follows from Step 1. 

Concerning statement (ii),  it is sufficient to prove that $u\in C^{\alpha}_H(X)$, with 
$\|u\|_{C^{\alpha}_H(X)} \leq C \|f\|_{C^{\alpha}_H(X)} $ for some $C>0$, and to use Step 1 as above. 
This is a simple consequence of Lemma  \ref{Le:sgrHolderZygmund}. Indeed, using \eqref{alpha-alpha}  with $k=0$ we get 
$$|u(x+h) - u(x)|    \leq   \int_0^{\infty} e^{-\lambda t}|P_tf(x+h) - P_tf(x)| \,dt \leq   \int_0^{\infty} e^{-\lambda   t}c\|h\|_H^{\alpha}[f]_{C^\alpha_H(X)}dt
= \frac{c}{\lambda  } \|h\|_H^{\alpha}[f]_{C^\alpha_H(X)}. $$
%
%
\end{proof}
 
 Notice that for $n\geq 1/\theta$ in case (i) and for $n\geq  \alpha + 1/\theta$ in case (ii), the arguments used above do not work, since the functions 
 $t\mapsto t^{-n\theta}$, $t\mapsto t^{-(n-\alpha)\theta}$, respectively,  are not integrable near $0$, and  \eqref{stima_integrando0}, \eqref{stima_integrando} are not helpful to conclude that $D^n_Hu(x)$ exists.

 Optimal regularity results are provided by the next theorems. The first one deals with H\"older regularity, and the second one with Zygmund regularity. 
 
 \begin{Theorem}
 \label{Th:Schauder_ell}
 Let $\lambda >0$, $f\in C_b(X)$ and let $u=R(\lambda, L)f$. The following statements hold. 
  \begin{itemize}
 \item[(i)] If $1/\theta \notin \N$ then $u\in C^{1/\theta}_H(X)$. There exists $C= C(\lambda)>0$, independent of $f$, such that 
 \begin{equation}
 \label{maggSchauder_ell0}
 \|u\|_{ C^{  1/\theta}_H(X)} \leq C  \|f\|_{\infty}. 
 \end{equation}
  \item[(ii)] If $\alpha\in (0,1)$ with $\alpha + 1/\theta \notin \N$ and $f\in C^\alpha_H(X)$ then $u\in C^{\alpha + 1/\theta}_H(X)$ and there  exists $C= C(\lambda, \alpha)>0$, independent of $f$, such that 
 \begin{equation}
 \label{maggSchauder_ell}
 \|u\|_{ C^{\alpha + 1/\theta}_H(X)} \leq C  \|f\|_{ C^{\alpha }_H(X)}. 
 \end{equation}
\end{itemize}
\end{Theorem}
\begin{proof}
Let  $n\in \N\cup \{0\}$ be the integral part of $\alpha +1/\theta$, with $\alpha =0$ in the case of  statement (i) and $\alpha \in (0, 1)$ in the case of statement (ii). If $n=0$, $u\in C_b(X)$ and \eqref{stimasup} holds. If $n>0$, 
we already know, by Proposition \ref{Pr:nonoptimal}, that $u\in C^n_H(X)$, and that estimate \eqref{stimaC^n0} (resp. estimate \eqref{stimaC^n}) holds. 
 
We have to prove  that $D^n_Hu$ belongs to $C^{\alpha +1/\theta   -n}_H(X, \mathcal L^n(H))$. 
As in Proposition \ref{Pr:nonoptimal}, it is sufficient to consider the case $\lambda >\omega$. If $\omega >0$, the case $\lambda\in (0,  \omega]$ is recovered by the same argument used in Step 2 of Proposition \ref{Pr:nonoptimal}.

We treat separately the cases $n>0$ and $n=0$.

Let  $n=0$. This implies that $\theta >1$ in statement (i), and $(1-\alpha)\theta >1$ in statement (ii). 
For every fixed $h$,  we split $ u = a_h  + b_h $, where
\begin{equation}
\label{a,b0}
a_h(y) =   \int_0^{\|h\|^{1/\theta}_H} e^{-\lambda t}  P_tf(y) \,dt, \quad  b_h(y) =   \int_{\|h\|^{1/\theta}_H}^{\infty} e^{-\lambda t}  P_tf(y) \,dt, \quad y\in X. 
\end{equation}
So, for every $x\in X$ we have 
$$ 
| a_h(x+h) - a_h( x) | \leq \int_0^{\|h\|^{1/\theta}_H} e^{-\lambda t}  |P_t f(x+h) - P_t f(x)|\, dt 
\leq \int_0^{\|h\|^{1/\theta}_H} 2\|f\|_{\infty}  dt = 2\|h\|^{1/\theta}_H \|f\|_{\infty}. 
 $$

To estimate $| b_h(x+h) - b_h(x) | $ we 
remark  that by \eqref{teofondcalc}(i) and \eqref{nth-derivatives} with $n=1$ for every $t>0$ we have 
\begin{equation}
\label{eq:1}
\| P_tf(x+h) -  P_tf(x)\|  \leq \sup_{y\in X} \|D_H P_tf(y)\|_{H^*} \|h\|_H
\leq K_{1 } \frac{  e^{ \omega t} }{  t^{ \theta}   }  \|h\|_H \|f\|_{\infty}, 
\end{equation}
which yields 

%
$$\begin{array}{lll}
 | b_h(x+h) - b_h(x) | & \leq & \ds  \int_{\|h\|^{1/\theta}_H}^{\infty} e^{-\lambda t}  | P_t f(x+h) - P_t f(x)|\, dt \leq    \int_{\|h\|^{1/\theta}_H}^{\infty}  \frac{K_{1 }}{t^{ \theta}} dt\,\|h\|_H
 \|f\|_{\infty}
 \\
\\
& \leq & \ds
 \frac{K_{1 }}{ \theta -1} \|h\|_H^{ 1/ \theta  } \|f\|_{\infty}. 
  \end{array}$$
Summing up,   $u\in C^{1/\theta}_H(X)$, and 
$$[u]_{C^{ 1/\theta}_H(X)} \leq  \bigg(2+  \frac{K_{1 }}{  \theta -1}\bigg)\|f\|_{\infty}.  $$
This estimate and \eqref{stimasup} give \eqref{maggSchauder_ell0} with $C(\lambda ) = 2+ K_{1 }/( \theta -1) + 1/\lambda$, in the case that $\theta >1$.

If  $\alpha \in (0,1)$ and $f\in C^{\alpha}_H(X)$ we use    \eqref{k=0}    and we get 
$$\begin{array}{l}
| a_h(x+h) - a_h(x) |   \leq   \ds  \int_0^{\|h\|^{1/\theta}_H} e^{-\lambda t}  | P_t f(x+h) - P_t f(x)|\, dt 
\\
\\
  \leq   \ds \int_0^{\|h\|^{1/\theta}_H}  e^{-(\lambda -\alpha \omega)t} M^{\alpha}\|h\|^{\alpha}_H [f]_{C^{\alpha}_H(X)} dt
\leq M^{\alpha}\|h\|^{\alpha + 1/\theta}_H [f]_{C^{\alpha}_H(X)}. 
\end{array}$$
To estimate $ | b_h(x+h) - b_h(x) | $  we use  \eqref{nth-derivativesHolder} with $n=1$, that gives  
\begin{equation}
\label{eq:2}
\| P_tf(x+h) -  P_tf(x)\|  \leq \sup_{y\in X} \|D_H P_tf(y)\|_{H^*} \|h\|_H
\leq K_{1, \alpha} \frac{  e^{  \omega t} }{  t^{( 1-\alpha)\theta}   }  \|h\|_H [f]_{C^\alpha_H(X)}, 
\end{equation}
which yields   
$$ | b_h(x+h) - b_h(x) |   \leq   \int_{\|h\|^{1/\theta}_H}^{\infty}  \frac{K_{1, \alpha}}{t^{(1-\alpha)\theta}} dt\,\|h\|_H
 [f]_{C^\alpha_H(X)}
  \leq     \frac{K_{1, \alpha}}{ (1-\alpha)\theta -1} \|h\|_H^{ 1/ \theta   + \alpha  }  [f]_{C^\alpha_H(X)}. $$
Summing up, we obtain   $u\in C^{\alpha + 1/\theta}_H(X)$, and 
$$[u]_{C^{\alpha + 1/\theta}_H(X)} \leq  \bigg(M^{\alpha} +  \frac{K_{1, \alpha}}{ (1-\alpha)\theta -1}\bigg)   [f]_{C^\alpha_H(X)}. $$
This estimate, together with \eqref{stimasup}, yield    \eqref{maggSchauder_ell}, with $C(\lambda)= M^{\alpha} +  K_{1, \alpha}/((1-\alpha)\theta -1) + 1/\lambda$, in the case that $\alpha + 1/\theta <1$.

For $n=[ \alpha + 1/\theta ] \geq 1$ the procedure is similar, just with different notations and constants. We already know from Proposition \ref{Pr:nonoptimal} that $u\in C^n_H(X)$, and we have to show that $D^n_Hu \in C^{\alpha + 1/\theta -n}_H(X, \mathcal L^n(H))$, with $\alpha=0$ as far as  statement (i) is concerned, and $\alpha \in (0,1)$ as far as statement (ii) is concerned. 

For every fixed $h$, $h_1, \ldots , h_n\in H$ we split $D^n_Hu(y)(h_1, \ldots, h_n) = a_h(y) + b_h(y)$, where now
\begin{equation}
\label{a,b}
\left\{ \begin{array}{lll}
a_h(y) &: = & \ds  \int_0^{\|h\|^{1/\theta}_H} e^{-\lambda t} D^n_HP_tf(y)(h_1, \ldots, h_n) \,dt, \quad y\in X, 
\\
\\
b_h(y) &: =  & \ds \int_{\|h\|^{1/\theta}_H}^{\infty} e^{-\lambda t} D^n_HP_tf(y)(h_1, \ldots, h_n) \,dt, \quad y\in X. 
\end{array}\right. 
\end{equation}
Let  us prove that statement (i) holds. In this case we have  $f\in C_b(X)$, $n\theta \in (0, 1)$, $(n+1)\theta >1$. 
Recalling that $ \omega  - \lambda<0$, estimate \eqref{nth-derivatives} yields
$$\begin{array}{lll}
| a_h(x+h) - a_h(x) | & \leq & | a_h(x+h)| + |a_h(x) | \leq  \ds 2 
K_{n } \int_0^{\|h\|^{1/\theta}_H}\frac{e^{( \omega  - \lambda )t}}{t^{n\theta}} dt \prod_{j=1}^n\|h_j\|_H \|f\|_{\infty}
\\
\\
& \leq  & \ds \frac{2K_{n } }{ 1-n\theta} \|h\|_H^{(1-n\theta)/\theta} \prod_{j=1}^n\|h_j\|_H  \|f\|_{\infty}. 
\end{array}$$
To estimate $| b_h(x+h) - b_h(x) | $ we apply \eqref{teofondcalc}(ii) to the function $P_tf$, and using   \eqref{nth-derivatives} we get 
\begin{equation}
\label{eq:3}
\|D_H^n P_tf(x+h) - D_H^nP_tf(x)\|_{\mathcal L^n(H)} \leq \sup_{y\in X} \|D_H^{n+1}P_tf(y)\|_{\mathcal L^n(H)} \|h\|_H
\leq K_{n+1 } \frac{  e^{\omega  t} }{  t^{(n+1 )\theta}   } \|f\|_{\infty} \|h\|_H , 
\end{equation}
which yields (since $\omega  -\lambda<0$)
$$ 
| b_h(x+h) - b_h(x) |   \leq    \int_{\|h\|^{1/\theta}_H}^{\infty}  \frac{K_{n+1 }\|h\|_H }{t^{(n+1 )\theta}} dt  \prod_{j=1}^n\|h_j\|_H \|f\|_{\infty}
 \leq   \frac{K_{n+1 } \|h\|_H^{ 1/ \theta - n  } }{ (n+1 )\theta -1}\prod_{j=1}^n\|h_j\|_H \|f\|_{\infty}. 
$$
Summing up we get
\begin{equation}
 \label{stimaHolder0}
| (D^n_Hu(x +h) -  D^n_Hu(x))(h_1, \ldots, h_n)| \leq C_1 \|h\|_H^{1/\theta -n   } \prod_{j=1}^n\|h_j\|_H \|f\|_{\infty} 
  \end{equation}
with 
$$C_1 =  \frac{2K_{n } }{ 1-n\theta} +  \frac{K_{n+1 }}{ (n+1 )\theta -1}. $$
Therefore, $D^n_Hu \in C^{1/\theta -n  }_H(X; \mathcal L^n(H))$ and 
$[D^n_Hu]_{C^{1/\theta -n  }_H(X; \mathcal L^n(H))} \leq C_1 \|f\|_{\infty}$. This estimate and \eqref{stimaC^n0} give \eqref{maggSchauder_ell0} for $n\geq 1$. 

Let us prove that statement (ii) holds. Now we have  $f\in C^{\alpha}_H(X)$ with  $\alpha \in (0, 1)$, $(n-\alpha)\theta \in (0, 1)$, $(n+1-\alpha)\theta >1$.  
Estimate \eqref{nth-derivativesHolder} yields
$$\begin{array}{lll}
| a_h(x+h) - a_h(x) | & \leq & | a_h(x+h)| + |a_h(x) | \leq  \ds 2 
K_{n, \alpha} \int_0^{\|h\|^{1/\theta}_H}\frac{e^{(    \omega -\lambda )t}}{t^{(n-\alpha)\theta}} dt \prod_{j=1}^n\|h_j\|_H [f]_{C^\alpha_H(X)}
\\
\\
& = & \ds \frac{2K_{n, \alpha} }{ 1-(n-\alpha)\theta} \|h\|_H^{(1-(n-\alpha)\theta)/\theta} \prod_{j=1}^n\|h_j\|_H [f]_{C^\alpha_H(X)}. 
\end{array}$$
To estimate $| b_h(x+h) - b_h(x) | $ we use again \eqref{teofondcalc}(ii) and by   \eqref{nth-derivatives} we get 
\begin{equation}
\label{eq:4}
\|D_H^n P_tf(x+h) - D_H^nP_tf(x)\|_{\mathcal L^n(H)} \leq \sup_{y\in X} \|D_H^{n+1}P_tf(y)\|_{\mathcal L^n(H)} \|h\|_H
\leq\frac{  K_{n+1, \alpha }  e^{ \omega  t} }{  t^{(n+1 -\alpha)\theta}   }  [f]_{C^\alpha_H(X)} \|h\|_H , 
\end{equation}
which yields 
$$\begin{array}{lll}
| b_h(x+h) - b_h(x) | & \leq &    \ds   \int_{\|h\|^{1/\theta}_H}^{\infty}  \frac{K_{n+1, \alpha }}{t^{(n+1-\alpha)\theta}} dt\,\|h\|_H
 \prod_{j=1}^n\|h_j\|_H  [f]_{C^\alpha_H(X)}
\\
\\
& \leq  & \ds \frac{K_{n+1 }}{ (n+1 -\alpha)\theta -1} \|h\|_H^{ 1/ \theta - n  +\alpha } \prod_{j=1}^n\|h_j\|_H  [f]_{C^\alpha_H(X)}. 
\end{array}$$
Summing up we get
\begin{equation}
 \label{stimaHolder}
| (D^n_Hu(x +h) -  D^n_Hu(x))(h_1, \ldots, h_n)| \leq C_2 \|h\|_H^{1/\theta -n +\alpha } \prod_{j=1}^n\|h_j\|_H  [f]_{C^\alpha_H(X)}
  \end{equation}
with 
$$C_2 =  \frac{2K_{n, \alpha} }{ 1-(n-\alpha)\theta} +  \frac{K_{n+1, \alpha}}{ (n+1-\alpha)\theta -1}. $$
Therefore, $D^n_Hu \in C^{1/\theta -n  +\alpha}_H(X; \mathcal L^n(H))$ and 
$[D^n_Hu]_{C^{1/\theta -n  +\alpha}_H(X; \mathcal L^n(H))} \leq C_2  [f]_{C^\alpha_H(X)}$. This estimate and \eqref{stimaC^n} give \eqref{maggSchauder_ell} in the case $n\geq 1$. 
 \end{proof}

If $  1/\theta =k \in \N$ we do not expect that $u\in C^{ 1/\theta}_H(X)$ whenever $f\in C_b(X)$. The simplest counterexample is $X=H= \R^N$ with $N>1$, $L=\Delta$. In this case \eqref{fundamental} is satisfied with $\theta = 1/2$ (see Sect. \ref{subs:Laplacian}) and it is well known that  the equation $\lambda u -\Delta u =f$ has not solutions in $C^2_b(\R^N)$ (and even not in $C^{1}_b(\R^N)$ with Lipschitz gradient) for every $f\in C_b(\R^N)$. The best regularity result in this scale of spaces is  in Zygmund spaces.

\begin{Theorem}
\label{Zygmund_ell}
 Let $\lambda >0$, $f\in C_b(X)$ and let $u=R(\lambda, L)f$. Then
  \begin{itemize}
 \item[(i)] If $1/\theta = k\in \N$,  $u \in Z^{k}_H(X)$, and there exists $C= C(\lambda )>0$, independent of $f$, such that 
 \begin{equation}
 \label{maggZygmund_ell0}
 \|u\|_{ Z^{k}_H(X)} \leq C  \|f\|_{\infty}. 
 \end{equation}
  \item[(ii)] 
If  $\alpha\in (0,1)$ and  $\alpha + 1/\theta = k \in \N$,   for every  $f\in C^{\alpha}_H(X)$ the function $u$   belongs to $Z^{k}_H(X)$, and there exists $C= C(\lambda, \alpha)>0$, independent of $f$, such that 
 \begin{equation}
 \label{maggZygmund_ell}
 \|u\|_{ Z^{k}_H(X)} \leq C  \|f\|_{ C^{\alpha }_H(X)}. 
 \end{equation}
\end{itemize}
 \end{Theorem}
 \begin{proof}
We proceed as in the proof of Theorem \ref{Th:Schauder_ell}, with due modifications. So, it is enough to prove that the statement holds if $\lambda >\omega$. The case where $\omega >0$ and $\lambda\in (0, \omega]$ will follow as in Step 2 of Proposition \ref{Pr:nonoptimal}. 
 
First we prove statements (i) and (ii) in the case $k=1$. 

We already know that  $u\in C_b(X)$, with $\|u\|_{\infty} \leq \|f\|_{\infty}/\lambda$. To show that 
 $u\in Z^1_H(X)$, for every fixed $h\in H$  we consider again  the functions $a_h$ and $b_h$ defined in \eqref{a,b0}, such that $u= a_h+b_h$. 

Let us prove statement (i), in the case $\theta =k=1$. For every $x\in X$ we have
  $$\begin{array}{lll}
 | a_h(x+2h) - 2 a_h(x+h) + a_h(x)| &  \leq & \ds  \int_0^{\|h\|_H} e^{-\lambda t} |P_tf(x+2h) - 2 P_tf(x+h) +  P_tf(x )|\,dt
\\
\\
&  \leq  &  \ds 4 \int_0^{\|h\|_H}  e^{-\lambda t}   \|f\|_{\infty} dt = 4 \|h\|_H   \|f\|_{\infty}. 
\end{array}$$   
To estimate  $b_h(x+2h) - 2 b_h(x+h) + b_h(x)$ we use \eqref{teofondcalc} twice, that gives
$$\begin{array}{l}
 P_tf(x+2h) - 2 P_tf(x+h) +  P_tf(x )= \ds  \int_0^1 D_HP_tf(x+(1+\sigma)h)(h)d\sigma -  \int_0^1 D_HP_tf(x+ \sigma h)(h)d\sigma 
\\
\\
= \ds \int_0^1 \int_0^1 D^2_HP_tf(x+(\tau +\sigma)h)(h,h)d\tau \,d\sigma
\end{array}$$
so that, by \eqref{nth-derivatives} with $n=2$, 
\begin{equation}
\label{perZ0}
|P_tf(x+2h) - 2 P_tf(x+h) +  P_tf(x )| \leq \sup_{y\in X}\|D^2_HP_tf(y)\|_{\mathcal L^2(H)}\|h\|^2_H\leq K_2\frac{e^{  \omega t}}{t^{2 }} \|f\|_{\infty} \|h\|^2_H. 
\end{equation}
Therefore, 
$$\begin{array}{lll}
|b_h(x+2h) - 2 b_h(x+h) + b_h(x)|
& \leq  & \ds \int_{\|h\|_H}^{\infty} e^{-\lambda t} |P_tf(x+2h) - 2 P_tf(x+h) +  P_tf(x )|\,dt
\\
\\
 & \leq  & \ds \int_{\|h\|_H}^{\infty} e^{-\lambda t} K_2\frac{e^{  \omega t}}{t^{2 }} \|f\|_{\infty}   \|h\|_H^2dt   
 \leq K_2\|h\|_H   \|f\|_{\infty}. 
\end{array}$$
Summing up, 
$$|u(x+2h) - 2 u(x+h) + u(x)| \leq (4+K_2)\|h\|_H   \|f\|_{\infty}, $$
so that $u\in Z^1_H(X)$ and \eqref{maggZygmund_ell0} holds with $C = 1/\lambda + 4+K_2$. So, statement (i) is proved for $\theta =1$. 
Concerning statement (ii),  when $\alpha + 1/\theta =1$ and  $f\in C^{\alpha}_H(X)$   we have by    \eqref{k=0}   
$$\begin{array}{l}
| a_h(x+2h) - 2 a_h(x+h) + a_h(x)|
\\
\\
\ds \leq  \int_0^{\|h\|_H^{1/\theta}} e^{-\lambda t} ( |P_tf(x+2h) - P_tf(x+h)|  + |P_tf(x+h) - P_tf(x )|)\,dt
\\
\\
\ds \leq 2 \int_0^{\|h\|_H^{1/\theta}}  e^{-\lambda t} M^{\alpha} e^{\alpha \omega t} [f]_{C^{\alpha}_H(X)}\|h\|_H^{\alpha}dt \leq  2 M^{\alpha}  \|h\|_H [f]_{C^{\alpha}_H(X)}
\end{array}$$
while \eqref{perZ0} has to be replaced (using \eqref{nth-derivativesHolder} with $n=2$) by 
\begin{equation}
\label{eq:5}
|P_tf(x+2h) - 2 P_tf(x+h) +  P_tf(x )| \leq \sup_{y\in X}\|D^2_HP_tf(y)\|_{\mathcal L^2(H)}\|h\|^2_H\leq 
\frac{K_{2, \alpha}e^{  \omega t}}{t^{(2-\alpha)\theta }}  [f]_{C^{\alpha}_H(X)} \|h\|^2_H, 
\end{equation}
and therefore, recalling that $(2-\alpha)\theta = 1+\theta$, 
$$\begin{array}{l}
|b_h(x+2h) - 2 b_h(x+h) + b_h(x)|
\ds \leq  \int_{\|h\|_H^{1/\theta}}^{\infty} e^{-\lambda t} |P_tf(x+2h) - 2 P_tf(x+h) +  P_tf(x )|\,dt
\\
\\
\ds \leq  \int_{\|h\|_H^{1/\theta}}^{\infty} e^{-\lambda t} K_{2, \alpha}\frac{e^{ \omega t}}{t^{(2-\alpha)\theta }}  [f]_{C^{\alpha}_H(X)} \|h\|^2_Hdt
 \leq \frac{ K_{2, \alpha}}{\theta} \|h\|_H [f]_{C^{\alpha}_H(X)}.  
\end{array}$$
Summing up, 
$$|u(x+2h) - 2 u(x+h) + u(x)| \leq \left( 2 M^{\alpha} + \frac{ K_{2, \alpha}}{\theta}\right)\|h\|_H  [f]_{C^{\alpha}_H(X)}, $$
so that $u\in Z^1_H(X)$ and \eqref{maggZygmund_ell} holds with $C = 1/\lambda +  2 M^{\alpha} +  K_{2, \alpha}/\theta$. So, statement (ii) is proved for $\alpha +1/\theta =1$. 

In the case that $k>1$ (we recall that $k= 1/\theta$ in statement (i), $k= \alpha + 1/\theta$ in statement (ii)), we know from Proposition \ref{Pr:nonoptimal} that $u\in C^{k-1}_H(X)$ and that estimates  \eqref{C^n0}, \eqref{C^nalpha} hold with $n=k-1$. What we have to prove is that $D^{k-1}_Hu \in 
 Z^1(X, \mathcal L^{k-1}(H))$, and to estimate its $Z^1$ norm in terms of $f$. To this aim, fixed any $h, h_1, \ldots, h_{k-1}\in H$, for every $y\in X$  we split 
 $D^{k -1}_Hu(y)(h_1, \ldots, h_{k-1})$ as $a_h(y) + b_h(y)$, where now
 \begin{equation}
\label{a,bZygmund}
\begin{array}{lll}
a_h(y) &: = & \ds  \int_0^{\|h\|^{1/\theta}_H} e^{-\lambda t} D^{k-1}_HP_tf(y)(h_1, \ldots, h_{   k-1   }) \,dt, \quad y\in X, 
\\
\\
b_h(y) &: =  & \ds \int_{\|h\|^{1/\theta}_H}^{\infty} e^{-\lambda t} D^{k-1}_HP_tf(y)(h_1, \ldots,  h_{   k-1   }) \,dt, \quad y\in X. 
\end{array}
\end{equation}
So, for every $x\in X$  we have 
\begin{equation}
\label{quelcheserve}
\begin{array}{l}
|(D^{k-1}_Hu(x+2h)- 2D^{k-1}_Hu(x+h) + D^{k-1}_Hu(x ))(h_1, \ldots, h_{k-1}) | \leq  
\\
\\
= |a_h(x+2h) - 2 a_h(x+h) + a_h(x)| + 
| b_h(x+2h) - 2 b_h(x+h) + b_h(x)|. 
 \end{array}
 \end{equation}
By the definition of $a_h$ we get
\begin{equation}
\label{|a_h|}
\begin{array}{l}
| a_h(x+2h) - 2 a_h(x+h) + a_h(x)|
\\
\\
\ds \leq  \int_0^{\|h\|_H^{1/\theta}} e^{-\lambda t} |(D^{k-1}_HP_tf(x+2h) - 2D^{k-1}_HP_tf(x+h) + D^{k-1}_HP_tf(x ))(h_1, \ldots, h_{k-1})\,dt , \end{array}
\end{equation}
To estimate the right-hand side  we observe that 
$$\begin{array}{l}
 |(D^{k-1}_HP_tf(x+2h) - 2D^{k-1}_HP_tf(x+h) + D^{k-1}_HP_tf(x ))(h_1, \ldots, h_{k-1})| 
 \\
 \\
 \leq 4 \sup_{y\in X} \|D^{k-1}_HP_tf(y)\|_{\mathcal L^{k-1}(H)}
 \prod_{j=1}^{k-1}\|h_j\|_H, \end{array}$$
which is bounded by $4 K_{k-1}  e^{ \omega t} t^{-(k-1)\theta} \prod_{j=1}^{k-1}\|h_j\|_H \|f\|_{\infty}$ thanks to 
 \eqref{nth-derivatives},  and by $4 K_{k-1, \alpha }$ $ e^{ \omega t}$ $ t^{-(k-1-\alpha )\theta} \prod_{j=1}^{k-1}\|h_j\|_H [f]_{C^{\alpha}_H(X)}$ thanks to   \eqref{nth-derivativesHolder} if $f\in C^{\alpha}_H(X)$ with $\alpha \in (0, 1)$. 
Therefore,  the right-hand side of \eqref{|a_h|} is bounded by
$$ \int_0^{\|h\|_H^{1/\theta}}\frac{ 4 K_{k-1}}{t^{(k-1)\theta}}dt \prod_{j=1}^{k-1}\|h_j\|_H \|f\|_{\infty}  =  4 kK_{k-1}   \|h\|\prod_{j=1}^{k-1}\|h_j\|_H \|f\|_{\infty}$$ %
if $k = 1/\theta$, and by 
$$\int_0^{\|h\|_H^{1/\theta}}  \frac{4K_{k-1, \alpha} }{t^{(k-1 -\alpha)\theta}}dt \prod_{j=1}^{k-1}\|h_j\|_H [f]_{C^{\alpha}_H(X)} = 4(k-\alpha) K_{k-1, \alpha}
 \|h\|_H \prod_{j=1}^{k-1}\|h_j\|_H [f]_{C^{\alpha}_H(X)}, $$
if  $ f\in C^{\alpha}_H(X)$ with $\alpha \in (0, 1)$ and 
$k= \alpha + 1/\theta$. 
Moreover, by the definition of $b_h$ we get 
\begin{equation}
\label{|b_h|}
\begin{array}{l}
| b_h(x+2h) - 2 b_h(x+h) + b_h(x)|
\\
\\
\ds \leq  \int_ {\|h\|_H^{1/\theta}}^{\infty}  e^{-\lambda t} |(D^{k-1}_HP_tf(x+2h) - 2D^{k-1}_HP_tf(x+h) + D^{k-1}_HP_tf(x ))(h_1, \ldots, h_{k-1})\,dt. 
\end{array}
\end{equation}
To estimate the right-hand side  we recall that for every $t>0$, $x\in X$, $h \in H$,  by \eqref{Zygmund2}  we have 
$$ \begin{array}{l}
|(D^{k-1}_HP_tf(x+2h) - 2D^{k-1}_HP_tf(x+2h) + D^{k-1}_HP_tf(x ))(h_1, \ldots, h_{k-1}) |
\\
\\
\leq \sup_{y\in X} \|D^{k+1}_HP_tf(y)\|_{\mathcal L^2(H)} \|h\|^2_H\prod_{j=1}^{k-1}\|h_j\|_H\end{array}$$
which is respectively bounded by $K_{k+1}  e^{  \omega t}  t^{-(k+1)\theta} \|h\|^2_H\prod_{j=1}^{k-1}\|h_j\|_H\|f\|_{\infty}$ due to   \eqref{nth-derivatives}, and by 
$K_{k+1, \alpha}  e^{  \omega t}  t^{-(k+1-\alpha )\theta} \|h\|^2_H\prod_{j=1}^{k-1}\|h_j\|_H[f]_{C^{\alpha}_H(X)}$ if  $f\in C^{\alpha}_H(X)$, due to \eqref{nth-derivativesHolder}. 
Therefore, the right-hand side of \eqref{|b_h|}  is bounded by 
$$ \int_ {\|h\|_H^{1/\theta}}^{\infty} \frac{  K_{k+1}}{t^{(k+1)\theta}}dt \|h\|^2_H \prod_{j=1}^{k-1}\|h_j\|_H \|f\|_{\infty} = k K_{k+1}  \|h\|\prod_{j=1}^{k-1}\|h_j\|_H \|f\|_{\infty}, $$
if $k = 1/\theta$, and by 
$$ \int_ {\|h\|_H^{1/\theta}}^{\infty}  \frac{K_{k+1, \alpha}}{t^{(k+1-\alpha)\theta}}dt \|h\|^2_H \prod_{j=1}^{k-1}\|h_j\|_H  [f]_{C^{\alpha}_H(X)} = (k-\alpha) K_{k+1, \alpha}
 \|h\|_H \prod_{j=1}^{k-1}\|h_j\|_H [f]_{C^{\alpha}_H(X)}, $$
if  $ f\in C^{\alpha}_H(X)$ with $\alpha \in (0, 1)$ and 
$k= \alpha + 1/\theta$. 
Summing up,  the left-hand side of \eqref{quelcheserve} is bounded by 
 $$k(4  K_{k-1} + K_{k+1})\prod_{j=1}^{k-1}\|h_j\|_H \|f\|_{\infty} \|h\|, $$
 if $1/\theta = k$, and by  
 $$(k-\alpha)(4  K_{k-1, \alpha} + K_{k+1, \alpha})\prod_{j=1}^{k-1}\|h_j\|_H  [f]_{C^{\alpha}_H(X)} \|h\|, $$
 if  $f\in C^{\alpha}_H(X)$ with $\alpha \in (0, 1)$ and $ \alpha + 1/\theta = k$. In both cases, this implies that 
 $D^{k-1}u \in Z^1_H(X, \mathcal L^{k-1}(H))$  (so that $u\in Z^{k}_H(X)$) with Zygmund seminorm bounded by $k(4  K_{k-1} + K_{k+1})  \|f\|_{\infty}  $in the first case,  and by    $(k-\alpha)(4  K_{k-1, \alpha} + K_{k+1, \alpha})  [f]_{C^{\alpha}_H(X)} $,    in the second case. Such estimates and \eqref{C^n0}, \eqref{C^nalpha} with $n=k-1$ yield \eqref{maggZygmund_ell0} and \eqref{maggZygmund_ell}, respectively. 
 \end{proof}

\subsection{Schauder and Zygmund estimates: evolution equations}

 This section deals with mild solutions to Cauchy problems, 
\begin{equation}
\label{Cauchy}
\left\{ \begin{array}{l}
v_t(t,x)  = Lv(t,x)  + g(t,x), \quad t\in [0,T], \; x\in X, 
\\
\\
v(0, \cdot) = f, \end{array}\right. 
\end{equation}
where $L$ is the operator defined in \eqref{L}, and $f   :X     \mapsto \R$, $g:[0,T]\times X\mapsto \R$ are bounded  continuous functions. Mild solutions are defined by 
\begin{equation}
\label{v}
v(t,x) = P_tf(x) + \int_0^t P_{t-s}g(s, \cdot)(x)ds , \quad  t\in [0,T], \; x\in X. 
\end{equation}
We already know that $(t,x)\mapsto P_tf(x)$ is continuous and bounded in $[0, +\infty)\times X$; if in addition $f\in C^n_H(X)$ for some $n\in \N$ all the derivatives $\partial ^k/\partial h_1\ldots\partial h_k (P_tf)$  with $k\leq n$ enjoy the same property, by Proposition \ref{representation}.  
We still have to  study the function
\begin{equation}
\label{v0}
v_0(t,x) := \int_0^t P_{t-s}g(s, \cdot)(x)ds =\int_0^t P_{s}g(t-s, \cdot)(x)ds, \quad  t\in [0,T], \; x\in X, 
\end{equation}
with $g\in C_b([0,T]\times X)$. Our final aim are maximal regularity results 
in H\"older and Zygmund  spaces with respect to the $x$ variable, so we introduce the relevant functional spaces. 
 
\begin{Definition} 
Let $T>0$, $\alpha> 0$. We denote by $C^{0, \alpha}_H ([0,T]\times X)$ the space of the bounded continuous functions $g:[0,T]\times X\mapsto \R$ such that $g(t, \cdot)\in C^{\alpha}_H(X)$ for every $t\in [0, T]$ and 
$$\|g\|_{C^{0,\alpha}_H ([0,T]\times X)} := \sup_{t\in [0, T]} \|g(t, \cdot)\|_{C^{\alpha}_H(X)}  < +\infty, $$
and moreover, if   $\alpha \geq 1$  , for every $(h_1, \ldots, h_k)\in H^k$, with $k\leq [\alpha]$, the functions  $(t,x)\mapsto (\partial ^{k}g/\partial h_1 \ldots \partial h_k)(t,x)$ are continuous in $[0,   T]    \times X$. 

For $k\in \N$ we denote by $Z^{0,k}_H ([0,T]\times X)$ the space of the bounded continuous functions $g:[0,T]\times X\mapsto \R$ such that $g(t, \cdot)\in Z^{k}_H(X)$ for every $t\in [0, T]$ and 
$$\|g\|_{Z^{0, k}_H ([0,T]\times X)} := \sup_{t\in [0, T]} \|g(t, \cdot)\|_{Z^{k}_H(X)}   < +\infty, $$
and, if $k\geq 2$, $g\in C^{0,k-1}_H ([0,T]\times X)$. 

If $H=X$ we drop the subindex $H$, setting $C^{0, \alpha}_b ([0,T]\times X):= C^{0, \alpha}_X ([0,T]\times X)$, $Z^{0,k} _b([0,T]\times X):= Z^{0,k}_{  X   } ([0,T]\times X)$. 
\end{Definition}

The next proposition is the evolution counterpart of Proposition  \ref{Pr:nonoptimal}. 

 \begin{Proposition}
 \label{Pr:nonoptimal_par}
For every $g\in C_b([0,T]\times X)$ the function $v_0$ defined in \eqref{v0}
is continuous, and we have 
\begin{equation}
\label{stimasup_par} 
\|v_0\|_{\infty}  \leq T\|g\|_{\infty}. 
\end{equation}
Moreover the following statements hold. 
 \begin{itemize}
 \item[(i)] Let $\theta <1$. For every $n\in \N$ such that $n<1/\theta$, $v_0\in C^{0,n}_H([0,T]\times X)$. There exists $C=C(T)>0$, independent of $g$, such that 
 \begin{equation}
 \label{C^n0_par}
 \|v_0\|_{C^{0,n}_H([0,T]\times X)}\leq C\|g\|_{\infty}.
 \end{equation}
  \item[(ii)] Let $\alpha\in (0,1)$ be such that $\alpha + 1/\theta >1$. For every $f\in C^{\alpha}_H(X)$ and for every $n\in \N$ such that $n<\alpha +1/\theta$, $v_0\in C^{0,n}_H([0,T]\times X)$. There exists $C= C(T, \alpha)>0$, independent of $g$, such that 
 \begin{equation}
 \label{C^nalpha_par}
 \|v_0\|_{C^{0,n}_H([0,T]\times X)}\leq C\|g\|_{C^{0,\alpha}_H([0,T]\times X)}
 \end{equation}
\end{itemize}
\end{Proposition}
\begin{proof}
Fix  $t$, $t_0\in [0,T]$ and  $x$, $x_0\in X$. If $t>t_0$ we have 
$$|v_0(t,x) - v_0(t_0, x_0) | \leq \int_0^{  t_0} |P_{t-s}g(s, \cdot)(x) - P_{t_0-s}g(s, \cdot)(x_0)|\,ds + \int_{t_0}^t  |P_{t-s}g(s, \cdot)(x)|\,ds . $$
Since for every $s\geq 0$ the function $(t,x)\mapsto P_{t-s}g(s, \cdot)(x)$ is continuous in $[s, +\infty)\times X$, and 
$|P_{t-s}g(s, \cdot)(x) - P_{t_0-s}g(s, \cdot)(x_0)|\leq 2\|g\|_{\infty}$, by the Dominated Convergence Theorem the first integral vanishes as $t\to t_0$, $x\to x_0$. The second integral is bounded by $(t-t_0)\|g\|_{\infty}$, so that it vanishes too as $t\to t_0$, $x\to x_0$. If $t<t_0$ we split $v_0(t,x) - v_0(t_0, x_0)= 
\int_0^{  t} (P_{t-s}g(s, \cdot)(x) - P_{t_0-s}g(s, \cdot)(x_0))\,ds + \int_{t}^{t_0}  P_{t_0-s}g(s, \cdot)(x_0)\,ds$ and we argue in the same way. So, $v_0$ is continuous. Estimate  \eqref{stimasup_par} is immediate. 

Concerning statements (i) and (ii), the proof of the fact that $v_0(t, \cdot)\in C^n_H(X)$ for every $t\in [0, T]$, and that 
$$\frac{\partial^kv_0}{\partial h_1 \ldots  \partial h_k}(t,  x)= \int_{0}^t D^k_HP_{t-s}g(s, \cdot)(x)(h_1, \ldots, h_k) \,ds , \quad k\in \{1, \ldots, n\}, t\in [0, T], \;x\in X, $$
 is quite analogous to the corresponding proof of Proposition \ref{Pr:nonoptimal},  and it is omitted. Estimates \eqref{C^n0_par} and \eqref{C^nalpha_par} follow as well as in the proof of Proposition \ref{Pr:nonoptimal}. 
 
It remains to  prove that $(t,x)\mapsto D^k_Hv_0(t, \cdot)(x)(h_1, \ldots, h_k) $ is continuous in $[0, T]\times X$ for every $ k\in \{1, \ldots, n\}$, $h_1, \ldots h_k\in H$, and this is similar to the proof of the continuity of $v_0$. For $t>t_0\in [0,T]$ and  $x$, $x_0\in X$ we split $\frac{\partial^kv_0}{\partial h_1 \ldots\partial h_k}(t,  x) - \frac{\partial^kv_0}{\partial h_1 \ldots \partial h_k}(t_0,  x_0)= I_1 + I_2$, where 
$$I_1 =  \int_{0}^{t_0} (D^k_HP_{t-s}g(s, \cdot)(x)- D^k_HP_{t_0-s}g(s, \cdot)(x_0))(h_1, \ldots, h_k) \,ds , $$
$$I_2 =\int_{t_0}^t D^k_HP_{t-s}g(s, \cdot)(x)(h_1, \ldots, h_k) \,ds. $$
Concerning $I_1$, by Proposition \ref{representation} for every $s\in [0, T]$ the function $(t,x)\mapsto D^k_HP_{t-s}g(s, \cdot)(x)(h_1, \ldots, h_k)$ is continuous in $(s, +\infty)\times X$, moreover for $0<s<t_0$ we have 
$$|D^k_HP_{t-s}g(s, \cdot)(x)- D^k_HP_{t_0-s}g(s, \cdot)(x_0))(h_1, \ldots, h_k)| \leq \varphi(s), $$
where  
$\varphi(s) = 2K_k\max\{ e^{ \omega T}, 1\}(t_0-s)^{  - k\theta   }\|g\|_{\infty}\prod_{j=1}^k\|h_j\|_H$ if $g\in C_b([0,T]\times X)$ by \eqref{nth-derivatives}, and $\varphi(s) = 2K_{k, \alpha}\max\{ e^{ \omega T}, 1\}(t_0-s)^{  -(k-\alpha)\theta    }\|g\|_{C^{0,\alpha}_H([0,T]\times X)}\prod_{j=1}^k\|h_j\|_H$  if $g\in C^{0,\alpha}_H([0,T]\times X)$ by \eqref{nth-derivativesHolder}. Both in case of statement (i) and of statement (ii), $\varphi \in L^1(0, t_0)$ and the Dominated Convergence Theorem yields that $I_1$ vanishes as $t\to t_0$, $x\to x_0$. 

Moreover  we have $| I_2|\leq  \int_{t_0}^t \psi(s)ds$, where $\psi(s) =  K_k\max\{ e^{\omega T}, 1\}(t-s)^{  - k\theta   }$ $\|g\|_{\infty}\prod_{j=1}^k\|h_j\|_H$ if $g\in C_b([0,T]\times X)$  by \eqref{nth-derivatives}, and $\psi(s) = K_{k, \alpha}\max\{ e^{\omega T}, 1\}(t-s)^{  -(k-\alpha)\theta    } \|g\|_{C^{0,\alpha}_H([0,T]\times X)}$ $\prod_{j=1}^k\|h_j\|_H$  if $g\in C^{0,\alpha}_H([0,T]\times X)$,  by \eqref{nth-derivativesHolder}. So we get $| I_2|\leq C(t-t_0)^{1-k\theta}$ in the first case, $| I_2|\leq C(t-t_0)^{1-(k-\alpha)\theta}$ in the second case;  in both cases $I_2$ vanishes as $t\to t_0$, $x\to x_0$. 

If   $t<t_0\in [0,T]$ and  $x$, $x_0\in X$ we split $D^k_Hv_0(t, \cdot)(x)(h_1, \ldots, h_k) - D^k_Hv_0(t_0, \cdot)(x_0)(h_1, \ldots, h_k)$ as above, replacing 
$\int_0^{t_0}$, $ \int_{t_0}^t $  by $\int_0^{t}$, $ \int_{t}^{t_0} $, respectively, and arguing in the same way. This ends the proof. 
\end{proof}

 \begin{Theorem}
 \label{Th:Schauder_par}
 Let  $f\in C_b(X)$, $g\in C_b([0,T]\times X)$ and let $v$ be defined by \eqref{v}. The following statements hold. 
  \begin{itemize}
 \item[(i)] If $1/\theta \notin \N$ and $f\in C^{1/\theta}_H(X)$, $g\in C_b([0,T]\times X)$,  then $v\in C^{0, 1/\theta}_H ([0,T]\times X)$. There exists $C= C(T)>0$, independent of $f$ and $g$, such that 
 \begin{equation}
 \label{maggSchauder_par0}
 \|v\|_{ C^{0, 1/\theta}_H ([0,T]\times X)} \leq C ( \|f\|_{C^{1/\theta}_H(X)} + \|g\|_{\infty}). 
 \end{equation}
\item[(ii)] If $\alpha\in (0,1)$ and  $\alpha + 1/\theta \notin \N$, $f\in C^{\alpha + 1/\theta}_H(X)$ and $g\in C^{0, \alpha}_H ([0,T]\times X)$,
   then $v\in C^{0,\alpha + 1/\theta}_H ([0,T]\times X)$. There exists $C= C(T, \alpha)>0$, independent of $f$ and $g$, such that 
 \begin{equation}
 \label{maggSchauder_par}
 \|v\|_{ C^{0,\alpha + 1/\theta}_H ([0,T]\times X)} \leq C  ( \|f\|_{ C^{\alpha + 1/\theta}_H(X)} + \|g\|_{C^{0,\alpha  }_H ([0,T]\times X)} ). 
 \end{equation}
\end{itemize}
\end{Theorem}
\begin{proof}
Both for $\alpha =0$ and for $\alpha \in (0,1)$,   for $f\in C^{\alpha+ 1/\theta}_H(X)$ 
the function $(t,x)\mapsto P_tf(x) $ belongs to $C^{0,\alpha + 1/\theta}_H ([0,T]\times X)$. Indeed, by Proposition \ref{representation}(iii), it belongs to $C^{0,n}_H  ([0,T]\times X)$ with $n= [\alpha + 1/\theta]$, while  Lemma \ref{Le:sgrHolderZygmund} yields $ P_tf \in C^{\alpha+ 1/\theta}_H(X)$ for every $t\leq T$, and $\sup_{0\leq t\leq T} \| P_tf \|_{C^{\alpha+ 1/\theta}_H(X)}\leq C \|  f \|_{C^{\alpha+ 1/\theta}_H(X)}$.   

 Therefore it is sufficient to prove that the statements hold in the case $f\equiv 0$, namely when $v=v_0$. 
Taking proposition \ref{Pr:nonoptimal_par} into account, it remains to be checked  that for every $t\in [0,T]$, $v(t, \cdot)\in C^{1/\theta}_H(X)$ in case of statement (i), $v_0(t, \cdot)\in C^{\alpha +1/\theta}_H(X)$ in case of statement (ii), with H\"older norm bounded by a constant independent of $t$. 
The proof is quite similar to the proof of Theorem \ref{Th:Schauder_ell}. 
Let  $n\in \N\cup \{0\}$ be the integral part of $\alpha +1/\theta$, with $\alpha =0$ in the case of  statement (i) and $\alpha \in (0, 1)$ in the case of statement (ii); we treat separately the cases $n>0$ and $n=0$.

Let  $n=0$.  For every fixed $h$,  we split $v = a_h  + b_h $, where for every $t\in [0, T]$, $y\in X$ 
\begin{equation}
\label{a,b_par0}
a_h(t, y) =   \int_0^{\|h\|^{1/\theta}_H \wedge t}   P_{s}g(t-s, y) \,ds,  \quad   b_h(t, y) = \ds \int_{\|h\|^{1/\theta}_H \wedge t}^{t}   P_{s}g(t-s, y) \,ds. \end{equation}
So, for every $x\in X$ and $t\in [0, T]$  we have 
$$\begin{array}{lll}
| a_h(t, x+h) - a_h(t, x) | & \leq & \ds  \int_0^{\|h\|^{1/\theta}_H}    |P_{s} g(t-s, x+h) - P_s g(t-s, x)|\, ds
\\
\\
& \leq & \ds \int_0^{\|h\|^{1/\theta}_H} 2\|g\|_{\infty}  dt = 2\|h\|^{1/\theta}_H \|g\|_{\infty}. 
\end{array}$$
If $\|h\|^{1/\theta}_H\geq t$, we have $b_h(t, x+h) - b_h(t, x) =0$. If $\|h\|^{1/\theta}_H< t$
to estimate $| b_h(t, x+h) - b_h(t, x) | $ we use \eqref{eq:1}, 
which yields 
$$\begin{array}{lll}
 | b_h(t, x+h) - b_h(t, x) | & \leq & \ds  \int_{\|h\|^{1/\theta}_H}^{t}   | P_s g(t-s, x+h) - P_s g(t-s, x)|\,   ds   
 \\
\\
& \leq & \ds
   \int_{\|h\|^{1/\theta}_H}^{\infty}  \frac{K_{1 }}{s^{ \theta}}   ds   \,\|h\|_H
 \|g\|_{\infty}
 \leq    \frac{K_{1 }}{ \theta -1} \|h\|_H^{ 1/ \theta  } \|g\|_{\infty}. 
  \end{array}$$
Summing up,   $v (t, \cdot)\in C^{1/\theta}_H(X)$, and 
$$[v_0(t, \cdot)]_{C^{ 1/\theta}_H(X)} \leq  \bigg(2+  \frac{K_{1 }}{  \theta -1}\bigg)\|g\|_{\infty}, \quad 0\leq t\leq T.  $$
This estimate and \eqref{stimasup_par} give \eqref{maggSchauder_par0} with $C(T) = 2+ K_{1 }/( \theta -1) + T$, in the case that $\theta >1$.

If  $\alpha \in (0,1)$, $\alpha + 1/\theta <1$, and $g\in C^{0,\alpha}_H([0,T]\times X)$, we use    \eqref{k=0}    and we get  
$$\begin{array}{lll}
| a_h(t, x+h) - a_h(t, x) | & \leq & \ds  \int_0^{\|h\|^{1/\theta}_H\wedge t}   | P_s g(t-s, x+h) - P_s g(t-s, x)|\, ds
\\
\\
& \leq & \ds \int_0^{\|h\|^{1/\theta}_H \wedge t} e^{\alpha \omega s}M^{\alpha}\|h\|^{\alpha}_H [g(t-s, \cdot)]_{C^{\alpha}_H(X)} ds
\\
\\
& \leq &  \ds  \max\{e^{\alpha \omega T}, 1\}  M^{\alpha}\|h\|^{\alpha + 1/\theta}_H \sup_{0\leq r\leq T}[g(r, \cdot)]_{C^{\alpha}_H(X)} . 
\end{array}$$
As before, if $\|h\|^{1/\theta}_H\geq t$ we have $b_h(t, x+h) - b_h(t, x)=0$.  If $\|h\|^{1/\theta}_H< t$, 
to estimate $ | b_h(t, x+h) - b_h(t, x) | $  we use  \eqref{eq:2}, that yields   
$$\begin{array}{lll} | b_h(t, x+h) - b_h(t, x) |  & \leq  &\ds  \int_{\|h\|^{1/\theta}_H}^{t}  \frac{K_{1, \alpha}e^{\omega s}}{s^{(1-\alpha)\theta}}  [g(t-s, \cdot)]_{C^{\alpha}_H(X)}ds\,\|h\|_H
\\
\\
&  \leq  & \ds   \max\{e^{\alpha \omega T}, 1\}   \frac{K_{1, \alpha}}{ (1-\alpha)\theta -1} \|h\|_H^{ 1/ \theta   + \alpha  } \sup_{0\leq r\leq T} [g(r, \cdot)]_{C^\alpha_H(X)}. \end{array}$$
Summing up, we obtain   $v(t, \cdot)\in C^{\alpha + 1/\theta}_H(X)$, and 
$$[v_0(t, \cdot)]_{C^{\alpha + 1/\theta}_H(X)} \leq   \max\{e^{\alpha \omega T}, 1\} \bigg(M^{\alpha} +  \frac{K_{1, \alpha}}{ (1-\alpha)\theta -1}\bigg)  \sup_{0\leq r\leq T} [g(r, \cdot)]_{C^\alpha_H(X)}, \quad 0\leq t\leq T. $$
This estimate, together with \eqref{stimasup_par}, yields    \eqref{maggSchauder_ell}, with $C(T)= T +  \max\{e^{\alpha \omega T}, 1\}(M^{\alpha} +  K_{1, \alpha}/((1-\alpha)\theta -1))$, in the case that $\alpha + 1/\theta <1$.

Let us consider now the case $n>0$. 
By Proposition \ref{Pr:nonoptimal_par}  we already know that $v_0\in C^{0,n}_H([0,T]\times X)$. 
It remains to prove  that $D^n_Hv(t, \cdot)$ is $H$-H\"older continuous with values in 
$\mathcal L^n(H)$, with exponent $1/\theta -n$ as far as  statement (i) is concerned, and with exponent $\alpha + 1/\theta -n$ as far as statement (ii) is concerned. Once again, this is done as in Theorem  \ref{Th:Schauder_ell}, 
splitting every partial derivative  $D^n_Hv(t, y)(h_1, \ldots, h_n) = a_h(t, y) + b_h(t, y)$, where   now  we set   
\begin{equation}
\label{a_par}
a_h(t, y) : =   \int_0^{\|h\|^{1/\theta}_H\wedge t}   D^n_HP_sg(t-s, \cdot)(y)(h_1, \ldots, h_n) \,ds, \quad t\in [0,T], \;y\in X, 
\end{equation}
\begin{equation}
\label{b_par}
b_h(t, y) =  \int_{\|h\|^{1/\theta}_H\wedge t }^{t}  D^n_HP_sg(t-s, \cdot)(y)(h_1, \ldots, h_n) \,ds, \quad t\in [0,T], \;y\in X. 
\end{equation}

Let  us consider statement (i). We recall that in this case we have  $g\in C_b([0,T]\times X)$, $n\theta \in (0, 1)$, $(n+1)\theta >1$. 
Estimate \eqref{nth-derivatives} yields  
$$\begin{array}{lll}
| a_h(t, x+h) - a_h(t, x) | & \leq & | a_h(t, x+h)| + |a_h(t, x) | \leq  \ds 2 
K_{n } \int_0^{\|h\|^{1/\theta}_H\wedge t }\frac{e^{  \omega  s}}{s^{n\theta}} ds \prod_{j=1}^n\|h_j\|_H \|g\|_{\infty}
\\
\\
& \leq  & \ds \max\{    e^{\omega T},   1\} 
\frac{2K_{n } }{ 1-n\theta} \|h\|_H^{(1-n\theta)/\theta} \prod_{j=1}^n\|h_j\|_H  \|g\|_{\infty}. 
\end{array}$$
To estimate $| b_h(t, x+h) - b_h(t, x) | $ when $\|h\|^{1/\theta}_H<t$ we use \eqref{eq:3},  
which yields 
$$\begin{array}{lll}
| b_h(t, x+h) - b_h(t, x) | & \leq &    \ds   \int_{\|h\|^{1/\theta}_H}^{t}  \frac{K_{n+1 }e^{  \omega  s}}{s^{(n+1 )\theta}} ds\,\|h\|_H
 \prod_{j=1}^n\|h_j\|_H \|g\|_{\infty}
\\
\\
& \leq  &\ds  \max\{   e^{\omega T},    1\} \frac{K_{n+1 }}{ (n+1 )\theta -1} \|h\|_H^{ 1  /    \theta - n  } \prod_{j=1}^n\|h_j\|_H \| g \|_{\infty}. 
\end{array}$$
Summing up we get
$$| (D^n_Hv_0(t, x +h) -  D^n_Hv_0(t, x))(h_1, \ldots, h_n)| \leq C_3 \|h\|_H^{1/\theta -n   } \prod_{j=1}^n\|h_j\|_H \|f\|_{\infty}, \quad 0\leq t\leq T,  $$
with 
$$C_3 =   \max\{    e^{\omega T},   1\}   \left(     \frac{2K_{n } }{ 1-n\theta} +  \frac{K_{n+1 }}{ (n+1 )\theta -1}   \right)   . $$
Therefore,  
$[D^n_Hv_0(t, \cdot)]_{C^{1/\theta -n  }_H(X; \mathcal L^n(H))} \leq C_3 \|f\|_{\infty}$ for every $t\in [0, T]$. This estimate and \eqref{stimasup_par}   give \eqref{maggSchauder_par0} for $n\geq 1$. 

Let us consider  statement (ii). Now we have  $g\in C^{0, \alpha}_H([0,T]\times X)$ with  $\alpha \in (0, 1)$, $(n-\alpha)\theta \in (0, 1)$, $(n+1-\alpha)\theta >1$.  
Estimate \eqref{nth-derivativesHolder} yields
$$\begin{array}{l}
| a_h(t, x+h) - a_h(t, x) |   \leq    \ds 2 
K_{n, \alpha} \int_0^{\|h\|^{1/\theta}_H\wedge t }\frac{e^{  \omega s}}{s^{(n-\alpha)\theta}} [g(t-s, \cdot)]_{C^\alpha_H(X)}ds \prod_{j=1}^n\|h_j\|_H 
\\
\\
  =   \ds  \max\{    e^{\omega T},    1\} \frac{2K_{n, \alpha} }{ 1-(n-\alpha)\theta} \|h\|_H^{(1-(n-\alpha)\theta)/\theta} \prod_{j=1}^n\|h_j\|_H  \sup_{0\leq r\leq T}[g(r, \cdot)]_{C^\alpha_H(X)}. 
\end{array}$$
To estimate $| b_h(t, x+h) - b_h(t, x) | $ for $\|h\|^{1/\theta}_H<t$ we use \eqref{eq:4}, which yields 
$$\begin{array}{l}
| b_h(t, x+h) - b_h(t, x) |   \leq     \ds   \int_{\|h\|^{1/\theta}_H}^{t}  \frac{K_{n+1, \alpha } e^{  \omega s} }{s^{(n+1-\alpha)\theta}}  [g(t-s, \cdot)]_{C^\alpha_H(X)}ds\,\|h\|_H
 \prod_{j=1}^n\|h_j\|_H 
\\
\\
  \leq    \ds  \max\{    e^{\omega T},   1\} \frac{K_{n+1 }}{ (n+1 -\alpha)\theta -1} \|h\|_H^{ 1/ \theta - n  +\alpha } \prod_{j=1}^n\|h_j\|_H \sup_{0\leq r\leq T}[g(r, \cdot)]_{C^\alpha_H(X)}. 
\end{array}$$
Summing up we get
$$| (D^n_Hv_0(t, x +h) -  D^n_Hv_0(t, x))(h_1, \ldots, h_n)| \leq C_4 \|h\|_H^{1/\theta -n +\alpha } \prod_{j=1}^n\|h_j\|_H \sup_{0\leq r\leq T}[g(r, \cdot)]_{C^\alpha_H(X)}$$
with 
$$C_4 =  \max\{    e^{\omega T},   1\}   \left(     \frac{2K_{n, \alpha} }{ 1-(n-\alpha)\theta} +  \frac{K_{n+1, \alpha}}{ (n+1-\alpha)\theta -1}   \right)   . $$
Therefore,  $[D^n_Hv_0(t, \cdot)]_{C^{1/\theta -n  +\alpha}_H(X; \mathcal L^n(H))} \leq C_4  [f]_{C^\alpha_H(X)}$  for every $t\in [0, T]$. This estimate and \eqref{stimasup_par}   give \eqref{maggSchauder_par} in the case $n\geq 1$. 
\end{proof}

\begin{Theorem}
\label{Zygmund_par}
 Let  $f\in C_b(X)$, $g\in C_b([0,T]\times X)$ and let $v$ be defined by \eqref{v}. The following statements hold. 
  \begin{itemize}
 \item[(i)] If $1/\theta = k\in \N$ and $f\in Z^{k}_H(X)$, then $v \in Z^{0,k}_H([0,T]\times X)$ and there exists $C= C(T)>0$, independent of $f$ and $g$, such that 
 \begin{equation}
 \label{maggZygmund_par0}
 \|v\|_{ Z^{0,k}_H([0,T]\times X)} \leq C ( \|f\|_{Z^{k}_H(X)} + \|g\|_{\infty}). 
 \end{equation}
  \item[(ii)] 
If  $\alpha\in (0,1)$, $\alpha + 1/\theta = k \in \N$,  $f\in Z^{k}_H(X)$ and $g\in  C^{0,\alpha}_H([0,T]\times X)$, then $v \in Z^{0,k}_H([0,T]\times X)$, and there exists $C= C(T, \alpha)>0$, independent of $f$ and $g$, such that 
 \begin{equation}
 \label{maggZygmund_par}
 \|v\|_{ Z^{0,k}_H([0,T]\times X)}   \leq C  ( \|f\|_{Z^{k}_H(X)} + \|g\|_{C^{0,\alpha}_H([0,T]\times X)} . 
 \end{equation}
\end{itemize}
 \end{Theorem}
 \begin{proof}
We know by Lemma \ref{Le:sgrHolderZygmund} that for every $f\in Z^{k}_H(X)$ the function $(t,x) \mapsto P_tf(x)$ belongs to $Z^{0,k}_H([0,T]\times X)$, and estimate \eqref{Zygmund-Zygmund} holds. So it is enough to prove that the statements hold for $f\equiv 0$, in which case $v=v_0$ defined by \eqref{v0}. 

First we prove statements (i) and (ii) in the case $k=1$. 
 
By Proposition \ref{Pr:nonoptimal_par} we already know that  $v_0\in C_b([0,T]\times X)$, with $\|v_0\|_{\infty} \leq T\|g\|_{\infty} $. To show that 
 $v_0(t, \cdot)\in Z^1_H(X)$ for every $t\in [0, T]$, for every fixed $h\in H$  we consider again  the functions $a_h$ and $b_h$ defined in \eqref{a,b_par0}, such that $v_0= a_h+b_h$. 

Let us prove statement (i), in the case $\theta =k=1$. For every $x\in X$ we have
$$\begin{array}{l}
| a_h(t, x+2h) - 2 a_h(t, x+h) + a_h(t, x)|
\\
\\
\ds \leq  \int_0^{\|h\|_H\wedge t}  |P_sg(t-s, \cdot)(x+2h) - 2 P_sg(t-s, \cdot)(x+h) +  P_sg(t-s, \cdot)(x )|\,dt
\\
\\
\ds \leq 4 \int_0^{\|h\|_H }    \|g\|_{\infty} dt = 4 \|h\|_H   \|g\|_{\infty}. 
\end{array}$$
We recall that $b_h(t, x+2h) - 2 b_h(t, x+h) + b_h(t, x)=0$ if $\|h\|_H\geq t$. 
To estimate  $b_h(t, x+2h) - 2 b_h(t, x+h) + b_h(t, x)$ if $\|h\|_H < t$
  we use \eqref{perZ0}, that yields 
$$\begin{array}{l}
|b_h(t, x+2h) - 2 b_h(t, x+h) + b_h(t, x)|
\\
\\
\ds \leq  \int_{\|h\|_H}^{t}  |P_sg(t-s, \cdot)(x+2h) - 2 P_sg(t-s, \cdot)(x+h) +  P_sg(t-s, \cdot)(x )|\,ds
\\
\\
\ds \leq  \int_{\|h\|_H}^{t}   K_2\frac{e^{ \omega s}}{s^{2 }} \|g\|_{\infty}   |h\|_H^2 ds   
\leq \max\{ e^{2\omega T}, 1\} K_2\|h\|_H   \|g\|_{\infty} . 
\end{array}$$
Summing up, 
$$|v_0(t, x+2h) - 2 v_0(t, x+h) +  v_0(t, x)| \leq (4+\max\{   e^{\omega T},    1\} K_2)\|h\|_H   \|g\|_{\infty}, $$
so that $v_0\in Z^1_H(X)$ and \eqref{maggZygmund_par0} holds with $C = T + 4+\max\{   e^{\omega T},    1\}K_2$. So, statement (i) is proved for $\theta =1$. 
Concerning statement (ii),  when $\alpha + 1/\theta =1$ and  $g\in  C^{0, \alpha}_H([0,T]\times X)$   we have by    \eqref{k=0}   
$$\begin{array}{l}
| a_h(t, x+2h) - 2 a_h(t, x+h) + a_h(t, x)|
\\
\\
\ds \leq  \int_0^{\|h\|_H^{1/\theta} \wedge t}  |P_sg(t-s, \cdot)(x+2h) - 2 P_sg(t-s, \cdot)(x+h) +  P_sg(t-s, \cdot)(x )|\,ds
\\
\\
\ds \leq 2 \int_0^{\|h\|_H^{1/\theta} \wedge t}   M^{\alpha} e^{\alpha \omega s} [g(t-s, \cdot)]_{C^{\alpha}_H(X)}\|h\|_H^{\alpha}dt \leq  2  \max\{ e^{\alpha \omega T}, 1\} M^{\alpha}  \|h\|_H \sup_{0\leq r\leq T}[g(r, \cdot)]_{C^{\alpha}_H(X)}
\end{array}$$
while to estimate $|b_h(t, x+2h) - 2 b_h(t, x+h) + b_h(t, x)|$ for $\|h\|_H^{1/\theta} <t$ we use \eqref{eq:5}, that gives 
(recalling that $(2-\alpha)\theta = 1+\theta$), 
$$\begin{array}{l}
|b_h(t, x+2h) - 2 b_h(t, x+h) + b_h(t, x)|
\\
\\
\ds \leq  \int_{\|h\|_H^{1/\theta}}^{t}  |P_sg(t-s, \cdot)(x+2h) - 2 P_sg(t-s, \cdot)(x+h) +  P_sg(t-s, \cdot)(x )|\,ds
\\
\\
\ds \leq  \int_{\|h\|_H^{1/\theta}}^{t}  K_{2, \alpha}\frac{e^{ \omega s}}{s^{(2-\alpha)\theta }}  [g(t-s, \cdot)]_{C^{\alpha}_H(X)} \|h\|^2_Hds
 \leq  \max\{   e^{\omega T},    1\} \frac{ K_{2, \alpha}}{\theta} \|h\|_H \sup_{0\leq r\leq T}[g(r, \cdot)]_{C^{\alpha}_H(X)}.  
\end{array}$$
Summing up, 
$$\begin{array}{l}
|v_0(t, x+2h) - 2 v_0(t, x+h) +  v_0(t, x)| \leq  
\\
\\
\leq \ds \left( 2  \max\{    e^{\alpha \omega T},   1\} M^{\alpha} +\max\{    e^{\omega T},    1\}  \frac{ K_{2, \alpha}}{\theta}\right)\|h\|_H  \sup_{0\leq r\leq T}[g(r, \cdot)]_{C^{\alpha}_H(X)}, \end{array}$$
so that $u\in Z^1_H(X)$ and \eqref{maggZygmund_ell}  follows. So, statement (ii) is proved for $\alpha +1/\theta =1$. 

In the case that $k>1$ (we recall that $k= 1/\theta$ in statement (i), $k= \alpha + 1/\theta$ in statement (ii)),   Proposition \ref{Pr:nonoptimal_par} yields  $v_0\in C^{0, k-1}_H([0,T]\times X)$. 
We have to prove   that $[D^{k-1}_Hv_0(t, \cdot)]_{Z^1(X, \mathcal L^{k-1}(H))}$ is bounded by a constant independent of $t$. To this aim, fixed any $h, h_1, \ldots, h_{k-1}\in H$, for every $t\in [0,T]$ and $y\in X$  we split 
 $D^{k -1}_Hv_0(t, y)(h_1, \ldots, h_{k-1})$ as $a_h(t, y) + b_h(t, y)$, where now  
 $$a_h(t, y) =  \int_0^{\|h\|^{1/\theta}_H\wedge t}  D^{k-1}_HP_sg(t-s, \cdot)(y)(h_1, \ldots,   h_{k-1}   ) \,ds$$
$$b_h(t, y) = \int_{\|h\|^{1/\theta}_H\wedge t }^{t}  D^{k-1}_HP_sg(t-s, \cdot)(y)(h_1, \ldots,   h_{k-1}   ) \,ds. $$
 We have
 $$\begin{array}{l}
| a_h(t, x+2h) - 2 a_h(t, x+h) + a_h(t, x)|
\\
\\
\ds \leq  \int_0^{\|h\|_H^{1/\theta}\wedge t}  |(D^{k-1}_H P_sg(t-s, \cdot)(x+2h) - 2D^{k-1}_HP_sg(t-s, \cdot)(x+h) 
\\
\\
\hspace{22mm} + D^{k-1}_HP_sg(t-s, \cdot)(x ))(h_1, \ldots, h_{k-1})|\,dt , \end{array}$$
and arguing as in the proof of Theorem \ref{Zygmund_ell}, we see that the right-hand side is bounded by  
$$ \int_0^{\|h\|_H^{1/\theta}}  e^{ \omega s}\frac{ 4 K_{k-1}}{s^{(k-1)\theta}}ds\prod_{j=1}^{k-1}\|h_j\|_H \|g\|_{\infty}  \leq 
\max\{   e^{\omega T},    1\}   4 kK_{k-1}   \|h\|\prod_{j=1}^{k-1}\|h_j\|_H \|g\|_{\infty}$$ 
if $k = 1/\theta$, and by 
$$\begin{array}{l}
\ds \int_0^{\|h\|_H^{1/\theta}}    e^{ \omega s}\frac{4K_{k-1, \alpha} }{s^{(k-1 -\alpha)\theta}}[g(t-s, \cdot]_{C^{\alpha}_H(X)}ds \prod_{j=1}^{k-1}\|h_j\|_H 
\\
\\
\ds \leq  \max\{   e^{\omega T},     1\} 4(k-\alpha) K_{k-1, \alpha}
 \|h\|_H \prod_{j=1}^{k-1}\|h_j\|_H \sup_{0\leq r\leq 1}[g(r, \cdot)]_{C^{\alpha}_H(X)}, \end{array}$$
if  $g\in C^{0, \alpha}_H([0,T]\times X)$ with $\alpha \in (0, 1)$ and 
$k= \alpha + 1/\theta$. 
If $\|h\|_H^{1/\theta}<t$, we estimate 
$$\begin{array}{l}
| b_h(t, x+2h) - 2 b_h(t, x+h) + b_h(t, x)|  \leq
\\
\\
\ds  \int_ {\|h\|_H^{1/\theta}}^{t}  e^{-\lambda t} |(D^{k-1}_H  P_sg(t-s, x+2h) - 2D^{k-1}_HP_sg(t-s, x+h) + D^{k-1}_HP_sg(t-s, x ))(h_1, \ldots, h_{k-1})|    \,dt
\end{array}$$
and  arguing  again as in the proof of Theorem  \ref{Zygmund_ell} we see that the right-hand side is bounded by 
$$ \int_ {\|h\|_H^{1/\theta}}^{t}   e^{ \omega s}\frac{  K_{k+1}}{s^{(k+1)\theta}} ds \|h\|^2_H \prod_{j=1}^{k-1}\|h_j\|_H \|g\|_{\infty}  \leq   \max\{    e^{\omega T},     1\} k K_{k+1}  \|h\|\prod_{j=1}^{k-1}\|h_j\|_H \|g\|_{\infty}, $$
if $k = 1/\theta$, and by 
$$\begin{array}{l} \ds  \int_ {\|h\|_H^{1/\theta}}^{t}    e^{ \omega s}\frac{K_{k+1, \alpha}}{s^{(k+1-\alpha)\theta}}[g(t-s, \cdot)]_{C^{\alpha}_H(X)}ds \|h\|^2_H \prod_{j=1}^{k-1}\|h_j\|_H 
\\
\\
\ds  \leq  \max\{    e^{\omega T},     1\} (k-\alpha) K_{k+1, \alpha}
 \|h\|_H \prod_{j=1}^{k-1}\|h_j\|_H  \sup_{0\leq r\leq 1}[g(r, \cdot)]_{C^{\alpha}_H(X)}, \end{array}$$
if  $g\in C^{0,\alpha}_H([0,T]\times X)$ with $\alpha \in (0, 1)$ and 
$k= \alpha + 1/\theta$. 
Summing up, we estimate $[D^{k-1}_Hv_0(t, \cdot)(x+2h) -2D^{k-1}_Hv_0(t, \cdot)(x+h) + D^{k-1}_Hv_0(t, \cdot)(x)] (h_1, \ldots h_{k-1})$ 
 by 
 $$ \max\{   e^{\omega T},    1\}k(4  K_{k-1} + K_{k+1})\prod_{j=1}^{k-1}\|h_j\|_H \|g\|_{\infty} \|h\|, $$
 if $1/\theta = k$, and by  
 $$ \max\{   e^{\omega T},    1\}(k-\alpha)(4  K_{k-1, \alpha} + K_{k+1, \alpha})\prod_{j=1}^{k-1}\|h_j\|_H \sup_{0\leq r\leq 1}[g(r, \cdot)]_{C^{\alpha}_H(X)} \|h\|, $$
 if  $g\in C^{0,\alpha}_H([0,T]\times X)$  with $\alpha \in (0, 1)$ and $ \alpha + 1/\theta = k$. This implies that 
 $v_0(t, \cdot) \in Z^{k}_{  H   }(X)$ with Zygmund seminorm bounded by $ \max\{   e^{\omega T},    1\}k$ $(4  K_{k-1} + K_{k+1})  \|f\|_{\infty}  $ in the first case,  and by  $ \max\{   e^{\omega T},    1\}(k-\alpha)(4  K_{k-1, \alpha} + K_{k+1, \alpha})  [f]_{C^{\alpha}_H(X)} \|h\|$, in the second case. Such estimates and \eqref{Zygmund-Zygmund}   yield \eqref{maggZygmund_par0} and \eqref{maggZygmund_par}, respectively. 
 \end{proof}

  
 \section{Examples in finite dimension}
 \label{sect:finite}

In this section  $X=  \R^N$ and $T_t = e^{tB}$ for every $t$, where $B$ is any $N\times N$ matrix, so that 
\begin{equation}
\label{PtRN}P_tf(x) = \int_{\R^N} f(e^{tB}x + y) \mu_t(dy), \quad t>0, \;f\in C_b(\R^N), \; x\in \R^N. 
\end{equation}
The measures $\mu_t$ are given by
\begin{equation}
\label{muRN}
\mu_t(dy) = g_t(y)dy, \quad t>0, 
\end{equation}
where the nonnegative functions $g_t\in L^1(\R^N)$ satisfy $g_{t+s}(x) = \int_{\R^N} g_s(x-e^{sB}y)g_t(y)dy$ for $t$, $s>0$, a.e. $x\in \R^N$, and $\|g_t\|_{L^1(\R^N)} =1$, for every $t>0$. If $B=0$ this condition is simply $g_{t+s} = g_t\star g_s$ for $s$, $t>0$. 
 
Hypotheses  \ref{Hyp1}  and  \ref{Hyp2}  are   satisfied with $H=H_t= \R^N$ provided  $g_t$ is weakly differentiable in all directions and
\begin{equation}
\label{FominRN}
\sup_{t>0} \;  t^{\theta} \bigg\| \frac{\partial g_t}{\partial x_k}\bigg\|_{L^1(\R^N)} <+\infty, \quad k=1, \ldots, n. 
\end{equation}
%

 \subsection{The Laplacian and the fractional Laplacian}
 \label{subs:Laplacian}

Strictly speaking, the results of this section are contained in the ones of both sections  \ref{sect:OU} and \ref{sect:Gross}, but we prefer to isolate them because checking our assumptions is particularly simple in this case and does not involve the technicalities needed in the more complicated situations of the next sections. 

 We recall that the heat semigroup is given by \eqref{P_t}, with $T_t=I$ for every $t$ (namely, $B=0$) and $\mu_t(dx) = g_t(x) dx$, where $g_t$ is the Gaussian kernel
$$g_t(x) =  \frac{1}{(4\pi t)^{N/2}} e^{-\frac{|x|^{2}}{4t}}, \quad x\in \R^N, \; t>0, $$ 
that satisfies \eqref{FominRN} with $\theta = 1/2$. The operator $L$ is the realization of the Laplacian in $C_b(\R^N)$, whose domain is $\{ f\in C_b(\R^N)\cap_{p>1}W^{2,p}_{loc}(\R^N): \; \Delta f\in C_b(\R^N)\}$. Schauder and Zygmund regularity results have several independent proofs by now, the present approach was outlined in \cite{L1}.  
Concerning the fractional Laplacian $-(-\Delta)^{s}$, $s\in (0,1)$, Schauder and Zygmund regularity results for stationary equations are already available. The first proof of the Schauder estimates seems to be in \cite[Cor. 2.9]{Si}.   Up-to-date references may be found  in the survey paper \cite{Stinga};     for more general classes of pseudodifferential operators including the fractional Laplacian see   \cite{DK,K} and the references therein.  However, a proof through our approach is very simple. Indeed, the associated semigroup is given by the classical subordination formula,
$$e^{-t(-\Delta)^{s}}f(x) = \int_0^{\infty} T_\sigma f(x) \eta^{(s)}_t(\sigma) d\sigma, \quad t>0, \;x\in \R^N, $$
where   $T_\sigma$ is now the heat semigroup, and    $ \eta^{(s)}_t$ is the inverse Laplace transform of $\lambda \mapsto e^{-t\lambda^{s}}$. Setting  $ \eta^{(s)}:= \eta^{(s)}_1$, we get  
$$\eta^{(s)}_t(\sigma ) = t^{-1/s}\eta^{(s)}(t^{-1/s}\sigma), \quad t, \; \sigma >0. $$
Moreover, $\eta^{(s)}$ is smooth in $(0, +\infty)$, it has positive values and it belongs to $L^{\infty}(0, +\infty) \cap W^{1,1}(0, +\infty)$. This is easily seen modifying the integral that defines $\eta^{(s)}$, to get (see e.g. \cite{Yosida})
\begin{equation}
\label{Y}
\eta^{(s)}(\sigma) = \frac{1}{\pi} \int_0^{\infty} e^{-\sigma r - r^s\cos(s\pi )}\sin(r^s\sin(s\pi))dr, \quad \sigma >0. 
\end{equation}
Therefore, $e^{-t(-\Delta)^{s}}$ takes the form  \eqref{PtRN}, with $B=0$ and 
$$ \mu_t(dy) = p_{s, t}(y) dy, $$
where  
\begin{equation}
\label{eq:Laplace}
p_{s, t}(y) = \frac{1}{t^{1/s}}\int_0^{\infty} g_\xi(y)\eta^{(s)} (t^{-1/s}\xi)d\xi, \quad y\in \R^N, \; t>0, 
\end{equation}
By homogeneity, we get
$$ p_{s, t}(y) = t^{-N/  (2s)    } p_{s, 1}(t^{-1/(2s)}y) , \quad t>0, \; y\in \R^N, $$
and such equality easily yields that $t\mapsto \mu_t$ is weakly continuous in $[0, +\infty)$. Moreover, 
$$\frac{\partial}{\partial y_k}p_{s, t}(y) =  t^{  -(N+1)/(2s)    } \frac{\partial}{\partial y_k}p_{s, 1}(t^{  -1/(2s)    }y) , \quad t>0, \; y\in \R^N, $$
which implies  
$$\int_{\R^N} \bigg| \frac{\partial}{\partial y_k}p_{s, t}(y) \bigg|\,dy =  t^{-(N+1)/(2s)} \int_{\R^N} \bigg| \frac{\partial}{\partial y_k}p_{s, 1}(t^{-1/(2s)}y) \bigg|\,dy = t^{-1/(2s)}\int_{\R^N}\bigg| \frac{\partial}{\partial z_k}p_{s, 1}(z) \bigg|\,dz. $$
From the   representation formula \eqref{eq:Laplace} we get
$$\begin{array}{l}
\ds \int_{\R^N}\bigg| \frac{\partial}{\partial z_k}p_{s, 1}(z) \bigg|\,dz = \int_{\R^N} \bigg| \int_0^{\infty} \frac{z_k e^{-|z|^2/4\xi}}{2\xi (4\pi \xi)^{N/2}} \eta^{(s)}(\xi)\,d\xi\bigg|\,dz
\\
\\
\ds = \int_0^{\infty}  \frac{ \eta^{(s)}(\xi)}{2\xi(4\pi \xi)^{N/2}} \int_{\R^N} |z_k| e^{-|z|^2/4\xi}dz \,d\xi =  \frac{ 1}{\sqrt{\pi}}  \int_0^{\infty} \frac{\eta^{(s)}(\xi)}{\sqrt{\xi}}\,d\xi. 
\end{array}$$
The last integral is finite, since $\eta ^{(s)}$ is bounded and it belongs to  $L^1(0, +\infty)$. Therefore, there is $C>0$ such that 
$$\bigg\| \frac{\partial}{\partial y_k}p_{s, t}\bigg\|_{L^1(\R^N)} \leq \frac{C}{t^{ 1/(2s)}}, \quad t>0, \; k=1, \ldots, N, $$
so that Hypotheses  \ref{Hyp1}  and  \ref{Hyp2}  are   satisfied with $X=H=\R^N$ and $\omega =0$, $\theta = 1/(2s)$. Theorems \ref{Th:Schauder_ell} and \ref{Zygmund_ell} yield

\begin{Theorem}
\label{Th:SchZygfraz}
Let $f\in C_b(\R^N)$ and $\lambda >0$, $s\in (0,1 )\setminus \{1/2 \}$. Then the equation
\begin{equation}
\label{eq:fraz}
\lambda u + (-\Delta)^{s} u = f
\end{equation}
has a unique solution $u\in C^{2s}_b(\R^N)$, and there is $C>0$, independent of $f$, such that 
$$\|u\|_{C^{2s}_b(\R^N)} \leq C\|f\|_{\infty}. $$
If $s=1/2$, equation \eqref{eq:fraz} has a unique solution in $Z^1(\R^N)$,  and there is $C>0$, independent of $f$, such that 
$$\|u\|_{Z^1(\R^N)} \leq C\|f\|_{\infty}. $$
If in addition $f\in C^{\alpha}_b(\R^N)$ with $\alpha \in (0, 1)$ and $\alpha + 2s \notin \{1, 2\}$, then $u\in C^{\alpha + 2s}_b(\R^N)$  and there is $C>0$, independent of $f$, such that 
$$\|u\|_{C^{\alpha + 2s}_b(\R^N)} \leq C\|f\|_{ C^{\alpha}_b(\R^N)}. $$
If $\alpha + 2s = k\in \{1, 2\}$, then $u\in Z^{k}_b(\R^N)$  and there is $C>0$, independent of $f$, such that 
$$\|u\|_{Z^{k}_b(\R^N)} \leq C\|f\|_{ C^{\alpha}_b(\R^N)}. $$
\end{Theorem}

 Theorems \ref{Th:Schauder_par} and \ref{Zygmund_par} yield
 
 \begin{Theorem}
\label{Th:SchZygfraz_par}
Let $s \in (0, 1)$, $\alpha \in [0, 1)$ be such that $\alpha + 2s\notin \{1, 2\}$, and let 
$f\in C^{\alpha + 2s}_b(\R^N)$, $g\in C^{0, \alpha}_b([0,T]\times \R^N)$ \footnote{For $\alpha =0$ we mean $C^{0,0}_b([0,T]\times \R^N) = C_b([0,T]\times \R^N) $.}. The mild solution to
\begin{equation}
\label{eq:fraz_par}
\left\{\begin{array}{l}
v_t (t,x) + (-\Delta)^{s}v (t, \cdot)(x) =  g(t,x), \quad 0\leq t\leq T, \; x\in \R^N, 
 \\
 \\
v(0, x) = f(x), \quad x\in \R^N, 
 \end{array}\right. 
\end{equation}
belongs to $C^{0, \alpha+2s}_b([0,T]\times \R^N)$, and there is $C>0$, independent of $f$ and $g$, such that 
$$\|v\|_{C^{0, \alpha+2s}_b([0,T]\times \R^N)} \leq C(\|f\|_{C^{\alpha + 2s}_b(\R^N)} + \|g\|_{C^{0, \alpha}_b([0,T]\times \R^N)}). $$

Let  $s\in (0, 1)$, $\alpha \in [0, 1)$ be such that $\alpha + 2s :=k \in \{1, 2\}$. Then for every $f\in Z^k_b(\R^N)$, $g\in C^{0, \alpha}_b([0,T]\times \R^N)$
the mild solution to \eqref{eq:fraz_par} belongs to $Z^{0, k}_b([0,T]\times \R^N)$,  and there is $C>0$, independent of $f$, such that 
$$\|v\|_{Z^{0,k}_b(\R^N)} \leq C(\|f\|_{Z^{k}_b(\R^N)} + \|g\|_{C^{0, \alpha}_b([0,T]\times \R^N)}). $$
\end{Theorem}

In the non-fractional case $s=1$ the first part of the theorem is known since many years (\cite{KCL}). For $s\in (0,1)$ it seems to be new.

 \subsection{Ornstein-Uhlenbeck operators with fractional diffusion}
\label{sect:OU}

Ornstein-Uhlenbeck operators are expressed by
$$(Lu)(x) = \frac{1}{2}( \text{Tr}(QD^2u) )(x)  - \langle Bx, \nabla u(x)\rangle , \quad x\in \R^N, $$
where $Q$ is a symmetric nonnegative definite matrix and $B$ is any matrix. Under ellipticity or hypoellipticity conditions  (respectively, det$\,Q>0$   or  det $\int_0^t e^{-sB}Qe^{-sB^*}ds >0$ for every $t>0$)  we already have maximal H\"older and Zygmund regularity results, first proved in \cite{DPL} in the elliptic case and then in \cite{L} in the hypoelliptic case. 

Here we consider modified Ornstein-Uhlenbeck operators which are the object of very recent studies (e.g., \cite{HMP,AB,CMP}), 
 heuristically given by 
$$(\mathcal L u)(x) =  \frac{1}{2}( {\text Tr}^s(QD^2u) )(x)  - \langle Bx, \nabla u(x)\rangle , \quad x\in \R^N. $$
with $s\in (0,1)$ and $Q>0$.   Tr$^s(QD^2)$ is the pseudo-differential operator with symbol $- \langle Q\xi, \xi\rangle^s$. 

The realization of $\mathcal L$ in $L^2(\R^N)$ has been studied in \cite{AB} even in the hypoelliptic case, using smoothing properties of the relevant semigroup, expressed through Fourier and inverse Fourier transform as 
$$\widehat{P_t f} = e^{ t{\text Tr}B} \exp  \left(- \frac{1}{2}\int_0^t |Q^{1/2}e^{\tau B^*}\cdot|^{2s} d\tau \right) \widehat{f}(e^{tB^*}\cdot), \quad t>0, $$
where $\; \widehat{  }\;$ denotes the Fourier transform $\mathcal F$, 
$$ \widehat{f} (\xi) = (\mathcal Ff)(\xi) = \int_{\R^N} e^{-i\langle x, \xi\rangle} f(x)dx. $$
Now we rewrite $P_t $ in the form \eqref{PtRN}. Applying the inverse Fourier transform we get, for every $f\in L^2(\R^N)$, 
$$ P_t f = e^{ t{\text Tr}B} {\mathcal F}^{-1}  \left(\exp  \left(- \frac{1}{2}\int_0^t |Q^{1/2}e^{\tau B^*}\cdot|^{2s} d\tau \right) \right) \ast  {\mathcal F}^{-1}( {\mathcal F}f(e^{tB^*}\cdot))$$
where
$$ {\mathcal F}^{-1}(g(e^{tB^*}\cdot))(y) = \frac{1}{(2\pi)^N} \int_{\R^N} g(e^{tB^*}\xi)e^{i\langle \xi, y\rangle}d\xi = e^{- t{\text Tr}B}( {\mathcal F}^{-1}g)(e^{-tB}y), $$
so that
$$\begin{array}{lll} P_t f(x) & = & \ds \int_{\R^N} f(e^{-tB}x-e^{-tB}y)  {\mathcal F}^{-1} \left(e^{- \frac{1}{2}\int_0^t |Q^{1/2}e^{\tau B^*}\cdot|^{2s} d\tau }\right)(y)\,dy
\\
\\
& = &  \ds  e^{ t{\text Tr}B} \int_{\R^N} f(e^{-tB}x-z)  {\mathcal F}^{-1}  \left(e^{- \frac{1}{2}\int_0^t |Q^{1/2}e^{\tau B^*}\cdot|^{2s} d\tau }\right)(e^{tB}z)\,dz
\\
\\
& = & \ds \int_{\R^N} f(e^{-tB}x + z) g_t(z)dz, 
\end{array}$$
with
$$\begin{array}{lll}  
g_t(z)  & = & \ds e^{ t{\text Tr}B}  {\mathcal F}^{-1} \left(e^{- \frac{1}{2}\int_0^t |Q^{1/2}e^{\tau B^*}\cdot|^{2s} d\tau }\right)( -e^{tB}z) = {\mathcal F}^{-1}  \left(e^{- \frac{1}{2}\int_0^t |Q^{1/2}e^{(\tau -t) B^*}\cdot|^{2s} d\tau }\right)( - z)
\\
\\
& = & \ds  
  \frac{1}{(2\pi)^N} \int_{\R^N} e^{- \frac{1}{2}\int_0^t |Q^{1/2}e^{-\sigma B^*}\xi |^{2s} d \sigma } e^{ -i\langle \xi, z\rangle}d\xi, 
\end{array}$$
 so that  $P_t$ is represented in the form \eqref{P_t}, with $\mu_t (dx):=g_t(x)dx$.

Setting
$$\varphi_t(\xi) =  e^{- \frac{1}{2}\int_0^t |Q^{1/2}e^{-\sigma B^*}\xi |^{2s} d\sigma},\quad  \xi\in \R^N, $$
 we have
 $g_t =(2\pi)^{-N}{\mathcal F}  (\varphi_t)$.  
Since $\varphi_t(\xi) = \exp (- \int_0^t \lambda (e^{-\sigma B^*}\xi) d\sigma )$, where $\lambda (\xi) = |Q^{1/2}\xi|^{2s}/2$ is a continuous negative definite function such that  $\lambda (0)=0$, then $\mu_t $ is a probability measure, see e.g. \cite[sect. 2.1]{FR}. 
Moreover, since  the function $(t, \xi)\mapsto \varphi_t (\xi)$ is continuous in $[0, +\infty) \times \R^N$, with $ \varphi_0 (\xi)=1$ for every $\xi$, by the L\'evy  Theorem $t\mapsto {\mu}_t$ is weakly continuous, and it weakly converges to $\delta_0$ as $t\to 0$. 
Therefore,  $P_t$ is well defined in $C_b(\R^N)$ and satisfies our assumptions with $\theta = 1/(2s)$ provided
 there exist  $  \partial g_t/\partial x_k \in L^1(\R^N)$, for each $k=1, \ldots N$, and there are $C>0$, $\omega \in \R$ such that
$ \| \partial g_t/\partial x_k \|_{L^1(\R^N)}  \leq Ct^{-1/(2s)}e^{\omega t}$ for every $t>0$. This is shown in the next lemma.

\begin{Lemma} $g_t\in W^{1,1}(\R^N)$ for every $t>0$, and    we have 
\begin{equation}
\label{(0,1]} \sup_{0<t\leq 1} t^{1/(2s)} \left\| \frac{\partial g_t}{\partial x_k}\right\|_{L^1(\R^N)} <+\infty, \quad k=1, \ldots, N, 
\end{equation}
\begin{equation}
\label{>1} \sup_{t>1}  \left\| \frac{\partial g_t}{\partial x_k}\right\|_{L^1(\R^N)} <+\infty, \quad k=1, \ldots, N. 
\end{equation}
\end{Lemma} 
\begin{proof}
The main step is to prove that $g_t\in W^{1,1}(\R^N)$ for every $t\in (0, 1]$ and that \eqref{(0,1]} holds. 
The remaining part of the statement will be a consequence, thanks to the algebraic relations among the functions $g_t$. 

It is convenient to rewrite $g_t$ as
 \begin{equation}
\label{rescaling}
g_t(x)  =  \frac{1}{t^{N/(2s)} }\frac{1}{(2\pi)^N} \int_{\R^N} e^{- \frac{1}{2t}\int_0^t |Q^{1/2}e^{-\sigma B^*}\eta|^{2s} d \sigma } e^{-i\langle t^{-1/(2s)}\eta, x\rangle}d\eta 
= \frac{1}{t^{N/2s} }\,\widetilde{g}_t\left( \frac{x}{t^{1/(2s)}}\right), 
 \end{equation}

where 
$$\widetilde{g}_t :=  \frac{1}{(2\pi)^N}{\mathcal F}\left( e^{- \frac{1}{2t}\int_0^t |Q^{1/2}e^{-\sigma B^*}\cdot |^{2s} d\sigma } \right)=  \frac{1}{(2\pi)^N}{\mathcal F}( \widetilde{\varphi}_t ), \quad t>0, $$
with $ \widetilde{\varphi}_t  = (\varphi_t)^{1/t}$. Our aim is now to show that $\widetilde{g}_t $ is $C^1$, and 
that $ \sup_{t\in (0, 1]}\|\partial \widetilde{g}_t /\partial x_k\|_{L^1(\R^N)}  <+\infty$. In this case,  by \eqref{rescaling} $g_t$ is $C^1$ too, and 
$\partial g_t/\partial x_k (x)= t^{-(N+1)/(2s) } \partial\widetilde{g}_t /\partial x_k(t^{-1/(2s)}x)$, which yields \eqref{(0,1]}.

To prove that  $\widetilde{g}_t $ is continuously differentiable and it has  $L^1$ derivatives it is enough to show that   $\xi\mapsto \xi_k\widetilde{\varphi}_t(\xi)$
belongs to $L^1(\R^N) \cap H^m(\R^N)$ for every $k=1, \ldots, N$, with $m>N/2$. Indeed, in this case $\partial \widetilde{g}_t /\partial x_k  = i {\mathcal F}^{-1}\psi _t$ with 
$$\psi_t(\xi) := \xi_k\widetilde{\varphi}_t(\xi), $$  
and 
$$\begin{array}{lll}
\ds  \int_{\R^N}\bigg| \frac{\partial \widetilde{g}_t}{\partial x_k}\bigg| dx & = & \ds \int_{\R^N} |{\mathcal F}^{-1}\psi_t (x)|dx = \frac{1}{(2\pi)^N}  |{\mathcal F} \psi _t(-x)|dx
\\
\\
& = & \ds  \frac{1}{(2\pi)^N}  \int_{\R^N} |{\mathcal F} \psi _t(x)|(1+|x|^2)^{m/2} \frac{1}{(1+|x|^2)^{m/2} }dx
\\
\\
&
\leq & \ds \frac{1}{(2\pi)^N}\| {\mathcal F} \psi _t (1+|\cdot|^2)^{m/2} \|_{L^2(\R^N)} \bigg( \int_{\R^N}\frac{1}{(1+|x|^2)^{m} }dx\bigg)^{1/2}
\\
\\
& 
\leq & C_N\|\psi_t\|_{H^m(\R^N)}. 
\end{array}$$
So, the rest of the proof of the differentiability of $\widetilde{g}_t$ for $t\in (0, 1]$ and of
 \eqref{(0,1]} is devoted to show that  $\psi_t\in L^1(\R^N) \cap H^m(\R^N)$ with $m>N/2$, and with $H^m$ norm bounded by a constant independent of $t$. As a first step, we observe that there exists $c>0$ such that 
\begin{equation}
\label{stimavarphi}
\widetilde{\varphi}_t(\xi) \leq e^{-c|\xi|^{2s}}, \quad 0<t\leq 1, \;\xi\in \R^N. 
\end{equation}
Indeed, let $M_B>0$, $\omega_B\in \R$ be such that $\|e^{tB^*}\|\leq M_Be^{\omega_B t}$ for every $t>0$. For every $\xi\in \R^N$ and $\sigma \in [0,1]$ we have 
$\xi = e^{\sigma  B^*}Q^{-1/2} Q^{1/2}e^{-\sigma B^*}\xi$, so that $\|\xi\| \leq  M_Be^{\omega_B \sigma }\|Q^{-1/2}\| \, \| Q^{1/2}e^{-\sigma B^*}\xi\|$, and therefore
$ \| Q^{1/2}e^{-\sigma B^*}\xi\| \geq \|\xi\|/\kappa$, with $\kappa = \min_{\tau \in [0,1]}M_Be^{\omega_B \tau}\|Q^{-1/2}\|$. Estimate \eqref{stimavarphi} holds with $c= 1/(2\kappa^{2s})$; it implies that $\psi_t \in L^1(\R^N)\cap L^2(\R^N)$, with $L^1$ and $L^2$ norms bounded by constants independent of $t$. 

To estimate the derivatives of $\psi_t$ we write it as $\psi _t(\xi)  = \xi_k e^{f_t(\xi)}$, where
$$f_t(\xi):= -\frac{1}{2t} \int_0^t|Q^{1/2}e^{-\sigma B^*}\xi |^{2s}   d\sigma    , \quad \xi\in \R^N. $$
The function $\theta(\sigma, \xi) := |Q^{1/2}e^{-\sigma B^*}\xi |^{2s}$ belongs to $C^{\infty}( \R \times (\R^N\setminus\{0\}))$, and therefore $  f_t\in C^{\infty}(  \R^N\setminus\{0\})$ for every $t>0$, and for every multi-index $\alpha$ we have
$$D^{\alpha}f_t = - \frac{1}{2t} \int_0^t D^{\alpha} (|Q^{1/2}e^{-\sigma B^*}\cdot |^{2s} )   d\sigma    . $$
Since for every $\sigma\in \R$ the function $\theta(\sigma, \cdot) $ is homogeneous with degree $2s$, its $j$-th order derivatives are homogeneous with degree $2s-j$; therefore for every multi-index $\alpha$ and $\xi\neq 0$ we have
$$|D^{\alpha}_{\xi}\theta (\sigma, \xi)| = \bigg| D^{\alpha}_{\xi}\theta\left(\sigma,  \frac{\xi}{|\xi|}  \right)\bigg| \, |\xi|^{2s-|\alpha|} \leq  \max \{ |D^{\alpha}_\xi \theta (\sigma, y)|:\; |y|=1\} |\xi|^{2s-|\alpha|} , $$
and consequently, for every $t\in (0, 1]$, 
 $$|D^{\alpha}f_t (  \xi)|  \leq \frac{1}{2} \sup_{0\leq \sigma\leq 1} |D^{\alpha}_{\xi}\theta (\sigma, \xi)| \leq \frac{1}{2} \max \{ |D^{\alpha}_\xi \theta (\sigma,  y)|:\;\sigma\in [0, 1], \; |y|=1\}  |\xi|^{2s-|\alpha|} =: K_{\alpha} |\xi|^{2s-|\alpha|}. $$
For every multi-index $\alpha$, $D^{\alpha} \widetilde{\varphi}_t = D^{\alpha}e^{f_t}$ is a linear combination of functions such as $e^{f_t} D^{\alpha_1}{f _t}\cdot \ldots 	\cdot D^{ \alpha_j}f_t$, where $j\in \{1, \ldots, |\alpha|\}$,  $\alpha_1, \ldots ,\alpha_j \in \N$ and $\sum \alpha_j = |\alpha|$. By the above estimates,  
$$e^{f_t(\xi)} D^{\alpha_1}f (\xi)\cdot \ldots 	\cdot D^{\alpha_j }f_t(\xi)\leq K_{\alpha_1}|\xi|^{2s-\alpha_1}\cdot \ldots \cdot  K_{ \alpha_j}|\xi|^{2s- \alpha_j} \widetilde{\varphi}_t (\xi)$$
and therefore
$$|D^{\alpha}  \widetilde{\varphi}_t (\xi)| \leq \sum_{j=1}^{|\alpha|} c_j |\xi|^{2sj-|\alpha|} e^{-c|\xi|^{2s}}, \quad \xi\neq 0, $$
with suitable coefficients $c_j$. It follows that for every $\eps \in (0, c)$ there exists $c_{\eps, |\alpha|}>0$ such that 
\begin{equation}
\label{derivate_varphi}
|D^{\alpha}  \widetilde{\varphi}_t  (\xi)| \leq c_{\eps, |\alpha| } e^{-\eps |\xi|^{2s}}  |\xi|^{2s-|\alpha|}, \quad \xi\neq 0. 
\end{equation}
Now, $D^{\alpha} \psi_t (\xi)$ is equal to $\xi_k D^{\alpha}  \widetilde{\varphi}_t (\xi)$ plus a linear combination of derivatives of $ \widetilde{\varphi}_t $ of order $|\alpha|-1$. Therefore, 
$$|D^{\alpha} \psi_t (\xi)| \leq  c_{ \eps, |\alpha|} e^{-\eps |\xi|^{2s}}  |\xi|^{2s-|\alpha|+1} +  C e^{-\eps |\xi|^{2s}}  |\xi|^{2s-|\alpha|}|\xi| = \widetilde{C}_{ \eps, |\alpha|} e^{-\eps |\xi|^{2s}}  |\xi|^{2s-|\alpha|+1}. $$
Consequently, $D^{\alpha} \psi _t\in L^2(\R^N)$ provided $\xi\mapsto e^{-\eps |\xi|^{2s}}  |\xi|^{2s-|\alpha|+1}\in L^2(\R^N)$, which is satisfied if $2(2s-|\alpha|+1)>-N$. 
It follows that $\psi_t \in H^m(\R^n)$ if  $2(2s - m+1)>-N$, namely $m < 2s +1 +N/2$. 
We recall that we need $m>N/2$. Since $2s+1>1$, the interval $(N/2, 2s +1 +N/2)$ contains at least one  integer $m$. For such $m$, $\psi_t \in H^m(\R^n)$ 
and $\|\psi_t\|_{H^m(\R^N)} $ is bounded by a constant independent of $t\in (0, 1]$, so that  \eqref{(0,1]} follows.

\vspace{3mm}

To prove \eqref{>1} we argue as in the proof of estimates  \eqref{nth-derivatives1}, \eqref{alpha-alpha} for large $t$. We use the semigroup property $P_t \circ P_s = P_{t+s}$ for $t$, $s>0$, which may be rewritten as 
$$g_{t+s}(x) = \int_{\R^N} g_s(x-e^{-sB}y)g_t(y)dy, \quad t, \; s>0, \; x\in \R^N. $$
In particular, for $t>1$ we get 
$$g_{t}(x) = \int_{\R^N} g_1(x-e^{-B}y) g_{t-1} (y)dy,  \quad x\in \R^N. $$
From the first part of the proof we know that $g_1$ is continuously differentiable. So,  $g_t$ is continuously differentiable and for every $k=1, \ldots, N$ we have 
$$\frac{\partial g_t }{\partial x_k}  (x) = \int_{\R^N} \frac{\partial g_1}{\partial x_k}(x-e^{-B}y)g_{t-1}(y)dy,  \quad x\in \R^N, $$
which implies (recalling that $\|g_{t-1}\|_{L^1(\R^N)} = 1$)
$$\left\| \frac{\partial g_t }{\partial x_k}\right\|_{L^1(\R^N)} \leq \left\| \frac{\partial g_1}{\partial x_k}\right\|_{L^1(\R^N)} , \quad t>1, $$
and \eqref{>1} follows. 

\vspace{3mm}


 \end{proof}

Applying Theorems \ref{Th:Schauder_ell} and \ref{Zygmund_ell} we   extend the results  of Theorem \ref{Th:SchZygfraz}. 

\begin{Theorem}
\label{Thm:OUfraz_ell}
Let $f\in C_b(\R^N)$ and $\lambda >0$, $s\in (0,1 )\setminus \{1/2 \}$. Then the equation
\begin{equation}
\label{eq:OUfraz}
\lambda u - L u = f
\end{equation}
has a unique solution $u\in C^{2s}_b(\R^N)$, and there is $C>0$, independent of $f$, such that 
$$\|u\|_{C^{2s}_b(\R^N)} \leq C\|f\|_{\infty}. $$
If $s=1/2$, equation \eqref{eq:OUfraz} has a unique solution in $Z^1(\R^N)$,  and there is $C>0$, independent of $f$, such that 
$$\|u\|_{Z^1(\R^N)} \leq C\|f\|_{\infty}. $$
If in addition $f\in C^{\alpha}_b(\R^N)$ with $\alpha \in (0, 1)$ and $\alpha + 2s \notin \{1, 2\}$, then $u\in C^{\alpha + 2s}_b(\R^N)$  and there is $C>0$, independent of $f$, such that 
$$\|u\|_{C^{\alpha + 2s}_b(\R^N)} \leq C\|f\|_{ C^{\alpha}_b(\R^N)}. $$
If $\alpha + 2s = k\in \{1, 2\}$, then $u\in Z^{k}_b(\R^N)$  and there is $C>0$, independent of $f$, such that 
$$\|u\|_{Z^{k}_b(\R^N)} \leq C\|f\|_{ C^{\alpha}_b(\R^N)}. $$
\end{Theorem}

Applying   Theorems \ref{Th:Schauder_par} and \ref{Zygmund_par} we extend the results of Theorem \ref{Th:SchZygfraz_par}. 
 
 \begin{Theorem}
\label{Thm:OUfraz_par}
Let $s\in (0, 1)$, $\alpha \in [0, 1)$ be such that $\alpha + 2s\notin \{1, 2\}$, and let 
$f\in C^{\alpha + 2s}_b(\R^N)$, $g\in C^{0, \alpha}_b([0,T]\times \R^N)$. The mild solution to
\begin{equation}
\label{eq:fraz_parOU}
\left\{\begin{array}{l}
v_t (t,x) = L v (t, \cdot)(x) +  g(t,x), \quad 0\leq t\leq T, \; x\in \R^N, 
 \\
 \\
v(0, x) = f(x), \quad x\in \R^N, 
 \end{array}\right. 
\end{equation}
belongs to $C^{0, \alpha+2s}_b([0,T]\times \R^N)$, and there is $C>0$, independent of $f$ and $g$, such that 
$$\|v\|_{C^{0, \alpha+2s}_b([0,T]\times \R^N)} \leq C(\|f\|_{C^{\alpha + 2s}_b(\R^N)} + \|g\|_{C^{0, \alpha}_b([0,T]\times \R^N)}). $$

Let  $s\in (0, 1)$, $\alpha \in [0, 1)$ be such that $\alpha + 2s :=k \in \{1, 2\}$. Then for every $f\in Z^k_b(\R^N)$, $g\in C^{0, \alpha}_b([0,T]\times \R^N)$
the mild solution to \eqref{eq:fraz_parOU} belongs to $Z^{0, k}_b([0,T]\times \R^N)$,  and there is $C>0$, independent of $f$, such that 
$$\|v\|_{Z^{0,k}_b(\R^N)} \leq C(\|f\|_{Z^{k}_b(\R^N)} + \|g\|_{C^{0, \alpha}_b([0,T]\times \R^N)}). $$
\end{Theorem}

The results of Theorem \ref{Thm:OUfraz_ell} seem to be new. A part of them, in the case $\alpha\in (0,1)$, $s\geq 1/2$, $1<  \alpha + 2s <2$, was proved 
in \cite{P} for a similar operator $\mathcal L$, with $Bx$ replaced by $b(x)$ in the drift,  $b\in C^{\alpha}_b(\R^N; \R^N)$. 
Concerning Theorem \ref{Thm:OUfraz_par}, in the case that $\alpha\in (0, 1)$, $s<1/2$, $\alpha + 2s\in (1,2)$, a similar result has been recently obtained in \cite{CMP} for a more general class of operators with suitable nonlinear and time dependent drift coefficients. 

 \section{Examples in infinite dimension}
 \label{sect:infinite}
 
\subsection{Ornstein-Uhlenbeck operators}
\label{sect:OUinf_dim}

In this section we deal with the case that $X$ is an infinite dimensional separable Banach space and the measures $\mu_t$ are Gaussian and centered (i.e.   with zero mean). 

For the general theory of Gaussian measures in Banach spaces we refer to \cite{Boga}. In particular, we recall that  every centered Gaussian measure $\gamma$  is Fomin differentiable along every  $h$ in the Cameron-Martin space $H_{\gamma}$, and the Fomin derivative $\beta^{\gamma}_h$  belongs to $L^{p}(X, \gamma)$ for every $p\in [1, +\infty)$ and satisfies 
%
%
\begin{equation}
\label{FominH}
\|\beta^{\gamma}_h\|_{L^{p}(X, \gamma)} = \left(\frac{1}{\sqrt{2\pi}} \int_{\R} |\xi|^{p} e^{-\xi^2 \|h\|_{H_\gamma}^2 /2} d\xi \right)^{1/p}=:   c_p
\|h\|_{H_\gamma}, 
\end{equation}
 with $c_1= \sqrt{2/\pi}$.  

The first  Schauder type theorems in the literature are in \cite{CDP}, \cite[Ch. 6]{DPZ}, concerning smoothing Ornstein-Uhlenbeck operators in a Hilbert setting.  We recall that if $X$ is a Hilbert space, for every centered Gaussian measure $\gamma$ with covariance $Q$, the relevant Cameron-Martin space $H_\gamma $ is the range of $Q^{1/2}$, with norm $\|h\|_{H_\gamma } = \|Q^{-1/2}h\|$ where $Q^{-1/2} $  is the pseudo-inverse of $Q^{1/2} $.

The assumptions to obtain (in all directions)  smoothing  Ornstein-Uhlenbeck semigroups are the following. 

\begin{Hypothesis}
\label{Hyp_DPZ}
$X$ is a separable Hilbert space,  $A:D(A)\subset X\mapsto X$ is the infinitesimal generator of a strongly continuous  semigroup $e^{tA}$,    and $Q\in \mathcal L(X)$ is a self-adjoint nonnegative operator, such that the operators defined by 
\begin{equation}
\label{Q_t} Q_t:= \int_0^t e^{sA}Qe^{sA^*}ds, \quad t>0
\end{equation}
have finite trace for every $t>0$. Moreover, $e^{tA}$ maps $X$ into $Q_t^{1/2}(X)$ for every $t>0$. \end{Hypothesis}

 The relevant Ornstein-Uhlenbeck semigroup is given by
\begin{equation}
 P_tf(x) = \int_X f(e^{tA}x+y)\mu_t(dy), \quad f\in  \mathcal B_b(X),\;  x\in X,
  \label{sgrOU}
 \end{equation}
  where 
 $$\mu_t = \mathcal N(0, Q_t), \quad t>0, $$
is the Gaussian measure in $X$ with mean $0$ and covariance $Q_t$. In this case $P_t$ is strong Feller, namely it maps $ \mathcal B_b(X)$   into $C_b(X)$. In fact, it maps $ \mathcal B_b(X)$  into $C^k_b(X)$ for every $k\in \N$ (\cite[Thm. 6.2.2]{DPZ}). Our $L$ is a realization of the operator $\mathcal L$ defined by
\begin{equation}\label{HypRealiz}
\mathcal L u(x) = \frac{1}{2} \text{Tr} (QD^2u(x)) + \langle x, A^*\nabla u(x)\rangle ,
\end{equation}
see \cite[Sect. 6.1]{DPZ}. 

Under Hypothesis \ref{Hyp_DPZ}, Hypothesis \ref{Hyp1} is satisfied with $H=X$, $H_t =Q_t^{1/2}(X)$, and  Hypothesis \ref{Hyp2}(i) holds, since $e^{tA}$  is a strongly continuous semigroup  on $X$.  But also   Hypothesis \ref{Hyp2}(ii)   is satisfied provided there exist $\omega \in \R$, $C$, $M$, $\theta >0$ such that 
\begin{equation}
\label{Lambda_t} \|Q_t^{-1/2}e^{tA}\|_{\mathcal L(X)} \leq \frac{Ce^{\omega t}}{t^{\theta}}, \quad t>0. 
\end{equation}
Indeed, in this case for every $h\in X$ and $t>0$, $p\geq 1$ we have $\|\beta_{e^{tA}h}^{\mu_t } \|_{L^p(X, \mu_t)} \leq   
c_p \|e^{tA}h\|_{Q_t^{1/2}(X)} \leq   c_pC  e^{\omega t}t^{-\theta} \|h\|$, thanks to \eqref{FominH} and \eqref{Lambda_t}. Taking $p=1$ yields that  Hypothesis \ref{Hyp2}(ii)   is satisfied; taking $p>1$ by Remark \ref{Rem:Frechet} the space derivatives in the statements of next Theorems  \ref{Th:DPZgen} and \ref{Th:DPZgen_par} are Fr\'echet derivatives instead of mere Gateaux derivatives.

Examples where \eqref{Lambda_t} is satisfied  are in   \cite{DPZ} (see Appendix B and Example 6.2.11).    One of them is considered in the next subsection. 

The corresponding Schauder and Zygmund regularity results in the stationary case are the following.  

\begin{Theorem}
\label{Th:DPZgen}
Let Hypotheses  \ref{Hyp_DPZ} and \eqref{Lambda_t} hold, and assume that $1/\theta \notin \N$. For every 
$f\in C_b(X)$ and $\lambda >0$,  the equation
\begin{equation}
\label{eq:OUsmoothing}
\lambda u - L u = f
\end{equation}
has a unique solution $u\in C^{1/\theta}_b(X)$, and there is $C>0$, independent of $f$, such that 
$$\|u\|_{C^{1/\theta}_b(X)} \leq C\|f\|_{\infty}. $$
If Hypotheses  \ref{Hyp_DPZ} and \eqref{Lambda_t} hold and $1/\theta \in\N$, equation \eqref{eq:OUsmoothing} has a unique solution in $Z^{1/\theta}(X)$,  and there is $C>0$, independent of $f$, such that 
$$\|u\|_{Z^{1/\theta}(X)} \leq C\|f\|_{\infty}. $$
If in addition $f\in C^{\alpha}_b(X)$ with $\alpha \in (0, 1)$ and $\alpha +1/\theta  \notin \N$, then $u\in C^{\alpha + 1/\theta }_b(X)$  and there is $C>0$, independent of $f$, such that 
$$\|u\|_{C^{\alpha + 1/\theta }_b(X)} \leq C\|f\|_{ C^{\alpha}_b(X)}. $$
If $\alpha + 1/\theta = k\in \N$, then $u\in Z^{k}_b(X)$  and there is $C>0$, independent of $f$, such that 
$$\|u\|_{Z^{k}_b(X)} \leq C\|f\|_{ C^{\alpha}_b(X)}. $$
\end{Theorem}

The Schauder part of this result  was stated in \cite{CDP}, \cite[Sect. 6.4.1]{DPZ} in the case $Q=I$, $A$ of negative type, and $\theta= 1/2$; see also \cite{DP} for further estimates in such a case.
It was extended in \cite{C} to Ornstein-Uhlenbeck semigroups arising as transition semigroups of some stochastic PDEs, with $X=L^2(\Omega)$, $\Omega$ being a bounded open set in $\R^N$. In this case, $A$ is the realization of a second order elliptic differential operator in $X$ and  $\theta =1/2$.

In the evolution case  Theorems \ref{Th:Schauder_par} and \ref{Zygmund_par} yield

\begin{Theorem}
\label{Th:DPZgen_par}
Let Hypotheses  \ref{Hyp_DPZ} and \eqref{Lambda_t} hold, and let $T>0$. For every $f\in C_b(X)$, $g\in C_b([0,T]\times X)$ let $v$ be the mild solution to 
$$
\left\{ \begin{array}{l}
v_t(t,x)  = Lv(t,x)  + g(t,x), \quad t\in [0,T], \; x\in X, 
\\
\\
v(0, \cdot) = f.  \end{array}\right. 
$$
 \begin{itemize}
 \item[(i)] If $1/\theta \notin \N$ and $f\in C^{1/\theta}_b(X)$ then $v\in C^{0, 1/\theta} _b([0,T]\times X)$. There exists $C= C(T)>0$, independent of $f$ and $g$, such that 
$$ \|v\|_{ C^{0, 1/\theta}_b ([0,T]\times X)} \leq C ( \|f\|_{C^{1/\theta}_b(X)} + \|g\|_{\infty}). $$
\item[(ii)] If $\alpha\in (0,1)$ and  $\alpha + 1/\theta \notin \N$, $f\in C^{\alpha + 1/\theta}_b(X)$ and $g\in C^{0, \alpha}([0,T]\times X)$
   then $v\in C^{0,\alpha + 1/\theta} _b([0,T]\times X)$. There exists $C= C(T, \alpha)>0$, independent of $f$ and $g$, such that 
$$ \|v\|_{ C^{0,\alpha + 1/\theta}_b([0,T]\times X)} \leq C  ( \|f\|_{ C^{\alpha + 1/\theta }_b(X)} + \|g\|_{C^{0,\alpha  }_b ([0,T]\times X)} ). $$
\end{itemize}
\end{Theorem}

Let us go back to the case where $X$ is a Banach space. 
The classical Ornstein-Uhlenbeck semigroup, 
\begin{equation}\label{eq:ousemi}
P_tf(x) = \int_X f(e^{-t}x + \sqrt{1-e^{-2t}}y)\mu (dy), \quad t>0, \; f\in C_b(X),\; x\in X,
\end{equation}
where $\mu$ is any centered Gaussian measure in $X$, is not strong Feller. It is smoothing only along the directions of the Cameron-Martin space $H_{\mu}$. However, by the changement of variables $z =  \sqrt{1-e^{-2t}}y$ in the integral it may be rewritten in the form \eqref{P_t}, with $T_t= e^{-t}I$ and $\mu_t = \mu \circ ( \sqrt{1-e^{-2t}}I)^{-1}  $, which is the centered Gaussian measure in $X$ with covariance $Q_t = (1-e^{-2t})Q$, if $Q:X^*\mapsto X$ is the covariance of $\mu$.  For the case where $\mu$ is non-Gaussian see Subsection 5.4 below. 

The generator $L$ of $P_t$ is a realization of div$_{\mu}\nabla _{H_\mu}$, where div$_{\mu}$ is the Gaussian divergence and $\nabla _{H_\mu}$ is the gradient along $H_\mu$, see \cite[Sect. 5.8]{Boga}. 

As we mentioned at the beginning of the section, $\mu_t $ is Fomin differentiable along every $h\in H_{\mu_t}$, and Hypothesis  \ref{Hyp1}  is satisfied with $H_t = H_{\mu_t}$. Since $Q_t $ is a multiple of $Q$, the elements of $H_{\mu_t}$ coincide with those of $H_{\mu}$, but the norms of these spaces are different, and precisely we have
$$\|h\|_{H_{\mu_t}} = \frac{1}{  \sqrt{1-e^{-2t}} } \|h\|_{H_{\mu}}, \quad h\in  H_{\mu}, \; t>0. $$
The semigroup $T_t = e^{-t}I$ maps obviously $H_\mu$ into itself and into $H_{\mu_t}$ for every $t>0$; moreover by \eqref{FominH} we have 
$$\|\beta^{\gamma}_{T_th}\|_{L^1(X, \mu_t)} =   \sqrt{\frac{2}{\pi}}  \|e^{-t}h\|_{H_{\mu_t}} =  \sqrt{\frac{2}{\pi}} \frac{e^{-t}}{ \sqrt{1-e^{-2t}}} \|h\|_{H_{\mu}}
\leq   \frac{c\,e^{-t}}{t^{1/2}} \|h\|_{H_{\mu}}, \quad t>0, \; h\in H_{\mu}, $$
with $c= (2/\pi)^{1/2} \sup_{t>0} t^{1/2}/ \sqrt{1-e^{-2t}}$. Therefore, Hypothesis  \ref{Hyp2}  is satisfied with $H=H_{\mu}$, $\omega =-1$, $\theta =1/2$. 
Applying Theorems \ref{Th:Schauder_ell} and \ref{Zygmund_ell}  gives   the same results of \cite{CL}, namely 

\begin{Theorem}
\label{OUclassico} 
Let $\lambda>0$, $f\in C_b(X)$, and set $H=H_{\mu}$. Then the unique solution to
$$\lambda u - Lu =f$$
belongs to $Z^2_H(X)$, and    there is $C >0$ such that 
$$\|u\|_{Z^2_H(X)} \leq C \|f\|_{ \infty}. $$
If in addition $f\in C^{\alpha}_H(X)$ with $0<\alpha <1$, $u$ belongs to $C^{2+\alpha}_H(X)$, and    there is $C  >0$ such that 
$$\|u\|_{C^{2+\alpha}_H(X)} \leq C \|f\|_{   C^{\alpha}_H(X)}. $$
Let now $T>0$, $f\in Z^{2 }_H(X)$, $g\in C_b([0,T]\times X)$. Then the mild solution to 
$$\left\{ \begin{array}{l}
v_t(t,x)  = Lv(t,x)  + g(t,x), \quad t\in [0,T], \; x\in X, 
\\
\\
v(0, \cdot) = f, \end{array}\right. $$
belongs to $Z^{0,2}_H([0,T]\times X)$, and there exists $C=C(T)>0$ such that $\|v\|_{Z^{0,2}_H([0,T]\times X)} \leq C( \|f\|_{Z^{2 }_H(X)} + \|g\|_{\infty})$. 

If in addition 
 $f\in C^{2+\alpha}_H(X)$, $g\in C^{0,\alpha}_H([0,T]\times X)$ with $\alpha\in (0,1)$,  then $v\in C^{0,2+\alpha}_H([0,T]\times X)$ and there exists $C=C(T, \alpha )>0$ such that $\|v\|_{C^{0,2+\alpha }_H([0,T]\times X)} \leq C( \|f\|_{C^{2+\alpha }_H(X)} + \|g\|_{C^{0,\alpha}_H([0,T]\times X)})$. 
\end{Theorem}

Theorem \ref{OUclassico}   can be extended to the wider class of Ornstein-Uhlenbeck operators considered in \cite{VN,GVN}. 
Here,  $A:D(A)\subset X\mapsto X$ is the infinitesimal generator of a strongly continuous  semigroup $e^{tA}$,    and $Q\in \mathcal L (X^*, X)$ is a non-negative (namely, $x^*(Qx^*) \geq 0$ for every $x^* \in X^*$) and symmetric (namely, $y^*(Qx^* )= x^*(Qy^*)$ for every $x^*$, $y^* \in X^*$) operator.    Moreover,  the operators $Q_t$ defined through a Pettis integral, 
$$Q_t x^* := \int_0^{t} e^{sA}Qe^{sA^*}x^*ds, \quad x^*\in X^*, \; t>0, $$
are assumed to be  the covariances of centered Gaussian measures $\mu_t$ in $X$. 
We recall that if  $X$ is a Hilbert space, $Q_t$ is the covariance operator of a Gaussian measure if and only if its trace is finite. If $X$ is just a Banach space,  
 (necessary and) sufficient conditions for $Q_t$   to     be the covariance of a Gaussian measure are in  
\cite[Thm. 7.1]{VNW}. References for sufficient conditions are also in \cite[Remark 2]{VNW2}. 

   Here we choose as $H$ the reproducing kernel Hilbert space $H_Q$ associated to the operator $Q$, see \cite{GVN} and \cite[Chapter III]{VTC}.  
If $B\in \mathcal L(H, X)$ and $Q= BB^*$,      $P_t$ defined by \eqref{P_t} with $T_t= e^{tA}$ is the transition semigroup of a stochastic evolution equation, 
$$\left\{ \begin{array}{l}
dX(t) = AX(t) dt + BdW_H(t), \quad t>0, 
\\
\\
X(0) = x
\end{array}\right. $$
where $W_H(t)$ is a cylindrical     Wiener process with Cameron-Martin space $H $,  see \cite{BVN}    for precise definitions and more details. 
Moreover,  it was proved in  \cite[Thm. 6.2]{GVN} that the semigroup $P_t$ is strongly continuous in the mixed topology on $C_b(X)$, which is the finest locally convex topology on $C_b(X)$ which agrees on every bounded set with the topology of uniform convergence on compact sets.

Hypothesis  \ref{Hyp1} is satisfied with $T_t=e^{tA}$,   $H= H_Q$     (\cite[Thm. 3.4]{GVN}) if  there exists $\omega \in \R$ such that for every $x^*\in D(A^*)$ we have $(A^*-\omega I)x^*(Qx^*)\leq 0$, or equivalently if for every $x^*\in X^*$ the function $t\mapsto \|i^*e^{-\omega t}(e^{tA})^*x^*\|_H$ is nonincreasing in $[0, +\infty)$ (here $i$ is the embedding $i:H\mapsto X $). In this case, by \cite[Thm. 3.5]{GVN}, all the Cameron-Martin spaces $H_{\mu_t}$ coincide and have equivalent norms  for every $t>0$, and 
  $e^{tA}$ maps   $H$     into $H_{\mu_t}$, with 
$$\|e^{tA}h\|_{H_{\mu_t}}^2 \leq \frac{1}{t^2}\int_0^t \|e^{sA}h\|^2_{  H    }ds, \quad t>0. $$
Therefore, if $M$, $\omega$ are such that $\|e^{tA}\|_{\mathcal L (  H    )} \leq Me^{\omega t}$ for $t>0$, we get 
$$\|e^{tA}h\|_{H_{\mu_t}} \leq \frac{M}{t^{1/2}}e^{\max\{\omega , 0\}t}\|h\|_{  H   }, \quad t>0. $$
Recalling \eqref{FominH}, we obtain that Hypothesis  \ref{Hyp2} is satisfied with $\theta = 1/2$ and $\omega$ replaced by $\max\{\omega , 0\}$. 
The statements of Theorem \ref{OUclassico}   hold in this case too.  Notice that, still by \eqref{FominH} and Remark \ref{Rem:Frechet}, the space derivatives in the statements  are   in fact    Fr\'echet derivatives. 
 

\subsection{Nonlocal Ornstein-Uhlenbeck operators}
\label{sect:NLOUO}
As mentioned in the introduction, semigroups of type \eqref{P_t} arise as transition semigroups of Ornstein--Uhlenbeck processes with Levy noise (see \cite{FR}, \cite{LR1}, \cite{LR2}) in finite or infinite dimensional state spaces, i.e., a stochastic process $X(t)$, $t \geq 0$, solving a stochastic differential equation on $X$ of type
\begin{align*}
	d X(t) = A X(t) d t + d Y(t),
\end{align*}
where $Y(t)$, $t \geq 0$, is a Levy process.
We have seen examples of this type to which our results apply in finite dimensions in Subsection \ref{sect:OU}.
In this subsection we shall discuss such a ``nonlocal'' example in infinite dimensions.  More  precisely, in the situation of the previous subsection we take $X= L^2(0,1):=L^2((0,1), dt)$, where $dt$ denotes Lebesgue measure on $(0,1)$.
Let $A$ be the  Laplace operator $\Delta$ on $L^2(0,1)$ with Dirichlet boundary conditions.
Since we do not want to use too much theory of Levy processes  (see \cite{AR,LI,S}), we just mention here that  such a process is determined by a negative  definite   function $\lambda \colon L^2(0,1) \longrightarrow \mathbb{C}$, which in our case we take concretely to be
\begin{align}\label{eq:lambda-5.4.1}
	\lambda(x) := \| x\|^2_{L^2(0,1)} + c \| x \|^{2s}_{L^2(0,1)}, \quad x \in L^2(0,1),
\end{align}
where $c \geq 0$ and $s \in (0,1)$.
The first summand corresponds to the Wiener process part and the second to the pure jump part of $Y(t)$, $t \geq 0$, in its Levy-It\^{o}-decomposition  (see \cite{AR,LI,S}).  
 The corresponding transition semigroup  of $X(t)$, $t\geq0$, is then given by 
\begin{align*}
P_t f(x) = \int f(e^{t \Delta}x+y) \, \mu_t(dy), \quad t>0, \, f \in \mathcal B_b (L^2(0,1)), \;x \in L^2(0,1),
\end{align*}
where $\mu_t$, $t \geq 0$, are probability measures with $\mu_0 = \delta_0$ and with Fourier transforms given by
\begin{align}\label{eq:fourier-transform-5.4.2}
\begin{split}
\hat \mu_t(x) &:= \int_{L^2(0,1)} e^{i \langle x, y \rangle_{L^2(0,1)}} \mu_t(dy)
\\ &\phantom{:}= \operatorname{exp} \bigg\{ - \int_0^t \| e^{r \Delta} x\|^2_{L^2(0,1)} + c \| e^{r \Delta} x \|^{2s}_{L^2(0,1)} \, dr \bigg\},
\end{split}
\end{align}
for $t>0$ and $x \in L^2(0,1)$; see Section 8 in \cite{LR2}.

In fact,  it follows from the proof of Proposition 8.1 in \cite{LR2} that there exists probability measures  $\mu^c_t$ on $L^2(0,1)$ such that 
\begin{align*}
 	\hat{\mu}^c_t(x) = \operatorname{exp} \Big\{ - \int_0^t   c \| e^{r \Delta} x\|^{2s}_{L^2(0,1)} \, dr \Big\}, \quad t>0, 
\end{align*} 
while $ \operatorname{exp}  \{ - \int_0^t \| e^{r \Delta} x\|^2_{L^2(0,1)} \, dr  \}$ is the Fourier transform of the Gaussian measure $\mathcal N_{0,Q_t}$, 
 where 
\begin{align}\label{eq-Qt-5.4.4}
	Q_t = 2\int_0^t e^{2r \Delta} \, dr =  (- \Delta)^{-1} \big(I-e^{2t \Delta}), \quad t>0,
\end{align}
 has finite trace, because the eigenvalues $\lambda_k$, $k \in \mathbb N$, of $\Delta$ are proportional to $-k^2$. 
 Therefore, 
\begin{align}\label{eq:muhat}
	\mu_t =  \mathcal N (0,Q_t)\ast \mu^c_t , \quad t>0. 
\end{align} 
Furthermore, it  follows immediately from the proof of Proposition 
8.1 in \cite{LR2} 
that the functions in \eqref{eq:fourier-transform-5.4.2} are equicontinuous in $0$ with respect to the 
Sazonov
topology on $L^2(0,1)$ (namely,  the topology generated by the seminorms $x\mapsto Tx$, where $T$ is any Hilbert-Schmidt operator in $L^2(0,1)$). This implies that $t\mapsto \mu_t$ is weakly continuous (see 
e.g. Proposition 1.1  in \cite[Chap. IV.1.2]{VTC}). 

The generator $L$ of $P_t$  is  a realization in $C_b(L^2(0,1))$ of the operator $\mathcal{L}$ that reads as 
\begin{align}\label{eq:generator-5.4.3}
  \mathcal{L}  u(x) = \int_{L^2(0,1)} \Big( i \langle x, \Delta y \rangle^{}_{L^2(0,1)} - \lambda(y) \Big) \, e^{i \langle x, y \rangle_{L^2(0,1)}} \, \nu(dy),
\end{align} 
for all  smooth cylindrical functions $u$  such that $u = \hat \nu$ for some probability measure $\nu$ on $L^2(0,1)$.
We refer to \cite{LR1}   for details and a rigorous analysis.

Clearly, if $c=0$, $ \mu_t$ is the Gaussian measure $\mathcal N (0,Q_t)$ above, which is given by  \eqref{Q_t} with   $Q=2I$ and $A=\Delta$. 
In this case, 
$e^{t\Delta}$ maps $L^2(0,1)$ into $Q_t^{1/2}(L^2(0,1))$, and by elementary spectral theory we get 
\begin{align*}
	\|Q_t^{-1/2} e^{t \Delta}\|_{\mathcal L ( L^2(0,1) ) } \leq \frac{c}{t^{1/2}}, \quad t>0, 
\end{align*}
so \eqref{Lambda_t} holds. So, $P_t$ is just the semigroup \eqref{sgrOU}  with $Q=2I$ and $A=\Delta$, and  
$\mathcal{L}$ has the representation  \eqref{HypRealiz}.

For $c>0$ we can apply our approach to  our realization $L$ of   the operator $\mathcal{L}$  in \eqref{eq:generator-5.4.3}.
So, let us check our Hypotheses \ref{Hyp1} and \ref{Hyp2} with $H=X=L^2(0,1)$,  $H_t = Q_t^{1/2}(L^2(0,1))$, and $\theta= 1/2$. 
Obviously the only thing to check is Hypothesis \ref{Hyp2}(ii). 

Let us start with proving the Fomin differentiability of $\mu_t$ along $e^{t\Delta} h$, for every $h \in L^2(0,1)$ and $t >0$.
By the previous subsection we know that  $  \mathcal N(0, Q_t)$ is Fomin differentiable along $e^{t\Delta} h$ for every $t>0$ and $h\in L^2(0,1)$, with 
\begin{align*}
	\big\| \beta_{e^{t\Delta} h}^{ \mathcal N(0, Q_t)} \big\|_{L^1 (L^2(0,1), \, \mathcal N(0, Q_t))} \leq \frac{c}{t^{1/2}} \|h\|_{L^2(0,1)}, \quad t>0, \, h \in L^2(0,1).
\end{align*}
Now \eqref{eq:muhat}  and  the following lemma ensure that 
Hypothesis \ref{Hyp2}(ii) also holds for the measures $\mu_t$, still with $\theta= 1/2$.

\begin{Lemma}
Let $\mu$, $\nu$ be   probability measures on a separable Banach space $X$, such that $\mu$ is Fomin differentiable along $v\in X$. 
Then $\mu \ast  \nu$ is Fomin differentiable along $v$ and  
\begin{align*}
	\big\| \beta_{v}^{\mu \ast  \nu} \big\|_{L^1 (X, \, \mu \ast \nu)} \leq \big\| \beta_{v}^{\mu } \big\|_{L^1 (X, \, \mu) }
\end{align*}
\end{Lemma}
\begin{proof}
Let $f \in C^1_b(X)$. 
Then, defining $\operatorname{Ad} \colon X \times X \longrightarrow X$ by $\operatorname{Ad}(x,y) := x+y$  and $\pi_1: X\times X\rightarrow X$ by $\pi_1(x,y)=x$,   we have
$$\begin{array}{lll}
\ds \int_{X}  \frac{\partial f}{\partial v} \, d(\mu \ast \nu)
 &=& \ds  \int_{X} \int_{X} \frac{\partial f}{\partial v} (x+y) \, \mu (dx) \nu(dy)
\\
\\ 
& = & \ds \int_{X} \int_{X} f(x+y) \, \beta_{v}^{\mu }(x) \, \mu (dx) \nu(dy)
\\
\\
&= & \ds \int_{X} \int_{X} f(\operatorname{Ad}(x,y))  \, \mathbb E_{\mu  \otimes \nu} \Big[ \beta_{v}^{\mu }  \circ\pi_1  \, \Big| \, \sigma(\operatorname{Ad}) \Big] \, \mu (dx) \nu(dy)
\\ 
\\
&= & \ds \int_{X} f(z) \, \mathbb E_{\mu  \otimes \nu} \Big[ \beta_{v}^{\mu }  \circ\pi_1  \, \Big| \, \operatorname{Ad}=z \Big] \, (\mu \ast  \nu)(dz), 
\end{array}$$
  where $\mathbb E_{\mu  \otimes \nu} [  \varphi \,  | \, \sigma(g) ]$ denotes the conditional expectation of $\varphi \in L^1 (X\times X, \mu  \otimes \nu)$ with respect to the sigma-algebra generated by $g:X\times X\mapsto X$.    
Furthermore,
$$\begin{array}{lll}
\ds \int_{X} \Big| \mathbb E_{\mu  \otimes \nu} \Big[ \beta_{v}^{\mu } \circ\pi_1 \, \Big| \, \operatorname{Ad}=z \Big] \Big| \, (\mu \ast \nu)(dz)
& \leq & \ds  \int_{X} \mathbb E_{\mu  \otimes \nu} \Big[ \big| \beta_{v}^{\mu } \circ\pi_1  \big| \, \Big| \, \operatorname{Ad}=z \Big] \, (\mu \ast  \nu)(dz)
\\ 
\\
&= & \ds \int_{X} \int_{X} \mathbb E_{\mu  \otimes \nu} \Big[ \big| \beta_{v}^{\mu }\circ\pi_1 \big| \, \Big| \, \sigma(\operatorname{Ad}) \Big] \, d \mu  \otimes \nu
\\
\\ 
&= & \ds \int_{X} \big| \beta_{v}^{\mu }  \big| \, d \mu .
\end{array}$$
The statement follows, with $ \beta_{v}^{\mu \ast  \nu}(z)  = \mathbb E_{\mu  \otimes \nu} \Big[ \beta_{v}^{\mu } \circ\pi_1 \, \Big| \, \operatorname{Ad}=z \Big] $.
\end{proof}

Applying Theorems  \ref{Th:Schauder_ell} and \ref{Zygmund_ell} yields

\begin{Theorem}
\label{Th:LRell}
For every $f\in C_b(L^2(0,1))$ and $\lambda >0$,  the equation
\begin{equation}
\label{eq:LRell}
\lambda u - L u = f
\end{equation}
has a unique solution $u\in Z^2_b(L^2(0,1))$, and there is $C>0$, independent of $f$, such that 
$$\|u\|_{Z^2_b(L^2(0,1))} \leq C\|f\|_{\infty}. $$
If in addition $f\in C^{\alpha}_b(L^2(0,1))$ with $\alpha \in (0, 1)$, then $u\in C^{2+\alpha}_b(X)$  and there is $C>0$, independent of $f$, such that 
$$\|u\|_{C^{2+\alpha  }_b(L^2(0,1))} \leq C\|f\|_{ C^{\alpha}_b(L^2(0,1))}. $$
\end{Theorem}

Applying  Theorems \ref{Th:Schauder_par} and \ref{Zygmund_par} yields

\begin{Theorem}
\label{Th:LRpar}
Let $T>0$. For every $f\in C_b(L^2(0,1))$, $g\in C_b([0,T]\times L^2(0,1))$,  let $v$ be the mild solution to 
$$
\left\{ \begin{array}{l}
v_t(t,x)  = Lv(t,x)  + g(t,x), \quad t\in [0,T], \; x\in L^2(0,1), 
\\
\\
v(0, \cdot) = f.  \end{array}\right. 
$$
 \begin{itemize}
\item[(i)] If $f\in Z^2_b(L^2(0,1))$ and $g\in C_b([0,T]\times L^2(0,1))$, 
then $v\in Z^{0, 2} _b([0,T]\times L^2(0,1))$, and there exists $C= C(T)>0$, independent of $f$ and $g$, such that 
$$ \|v\|_{Z^{0, 2}_b ([0,T]\times L^2(0,1))} \leq C ( \|f\|_{Z^2_b(L^2(0,1))} + \|g\|_{\infty}). $$
\item[(ii)] If $\alpha\in (0,1)$ and    $f\in C^{2+\alpha  }_b(L^2(0,1))$,  $g\in C^{0, \alpha}_b([0,T]\times L^2(0,1))$, 
   then $v\in C^{0,2 +\alpha } _b([0,T]\times L^2(0,1))$. There exists $C= C(T, \alpha)>0$, independent of $f$ and $g$, such that 
$$ \|v\|_{ C^{0,2+\alpha }_b([0,T]\times L^2(0,1))} \leq C  ( \|f\|_{ C^{ 2+2\alpha }_b(L^2(0,1))} + \|g\|_{C^{0,\alpha  }_b ([0,T]\times L^2(0,1))} ). $$
\end{itemize}
\end{Theorem}


\subsection{The Gross Laplacian and its powers} 
\label{sect:Gross}

Here $X$ is a separable Hilbert space and $Q\in \mathcal L(X)$ is a self-adjoint positive operator with finite trace. The semigroup $P_t$ is defined by 
\eqref{P_t} with $T_t=I$ for every $t>0$, and $\mu_t = \mathcal N(0, tQ)$ is the centered Gaussian measure in $X$ with covariance $tQ$. Therefore we have
\begin{equation}
\label{Gross_sgr}
P_tf(x) = \int_X f(x+y)\mu_t(dy) = \int_X f(x+\sqrt{t}z)\mu (dz), \quad f\in C_b(X), \; t>0.
\end{equation}
with $\mu =  \mu_1 = \mathcal N (0, Q)$. That $P_t$ is a semigroup (namely, $\mu_t \star \mu_s = \mu_{t+s}$ for every $s$, $t>0$) is a consequence of standard properties of Gaussian measures, e.g. \cite[Prop. 2.2.10]{Boga}. The operator $L$ defined in \eqref{L} is a realization of the differential operator 
$$\mathcal{L} u(x) = \frac{1}{2}\text{Tr}\,(QD^2u (x)). $$
See  \cite{G}, \cite[Ch. 3]{DPZ} and the references therein.  
We choose as $H_t $ the Cameron-Martin space of $\mu_t$. So, Hypothesis  \ref{Hyp1} is satisfied.
Moreover we take $H = H_1=$ the  Cameron-Martin space of $\mu$. We have $H_t = Q^{1/2}(X) =H$ for every $t>0$, with norm depending on $t$, 
$$\|h\|_{H_t} = \frac{1}{t^{1/2}}\|h\|_H, \quad h\in H, \; t>0.$$
 Consequently, by \eqref{FominH}, 
\begin{equation}
\label{beta_Gross}
\|\beta^{\mu_t}_{T_th}\|_{L^p(X, \mu_t)} \leq  \frac{c_p}{t^{1/2}}\|h\|_H, \quad h\in H, \; t>0,
\end{equation}
and taking $p=1$, Hypothesis  \ref{Hyp2}  is satisfied with $\theta =1/2$, $\omega =0$. Therefore Theorems  \ref{Th:Schauder_ell} and \ref{Zygmund_ell} yield that the statement of Theorem \ref{OUclassico} holds in this case too, and in this case too  the space derivatives in the statement are Fr\'echet derivatives,  by \eqref{beta_Gross} and Remark \ref{Rem:Frechet}. 

The Schauder part of Theorem \ref{OUclassico} in the stationary case was already stated in \cite{CDP1,DPZ}; see also \cite{ABP} for a related result.

Now let us consider the powers $(-L)^{s}$ with $s\in (0,1)$. As in the finite dimensional case  (see \eqref{eq:Laplace}) we define it as minus the generator of the subordinated semigroup $S_t$  of $P_t$  on $C_b(X)$ with subordinator $\{\eta^{(s)}_t(r)dr,\ t>0\}$, where as in Subsection \ref{subs:Laplacian}, $\eta^{(s)}_t(r)$, $r>0$, is given as the inverse Laplace transform of $[0,\infty) \ni \lambda \mapsto e^{-t\lambda^s}$, i.e.  
$$\ S_tf(x) = \int_0^{\infty}( P_{\sigma}f)(x) \eta^{(s)}_t (\sigma)d\sigma 
= \int_0^{\infty} \int_X f(x+y)  \mathcal N (0, \sigma Q)(dy) \,\eta^{(s)}_t(\sigma)d\sigma
= \int_X f(x+y) \nu_t(dy), \; t>0, $$
where $P_t$ is the semigroup in \eqref{Gross_sgr}, and  the measures $\nu_t$ are defined by 
\begin{equation}
\label{nu_t}
\nu_0(  B )= \delta_0( B ); \;\; \nu_t( B )  = \int_0^{\infty} \mathcal N (0, \sigma Q)(  B )\eta^{(s)}_t(\sigma)d\sigma =  \int_0^{\infty} \mathcal N (0, Q)(  B /\sigma ^{1/2})\eta^{(s)}_t(\sigma)d\sigma , \;\; t>0 ,\ B\in\mathcal{B}(X), 
 \end{equation}
 where $\mathcal{B}(X)$ denotes the Borel $\sigma$-algebra of $X$.  
According to the terminology of \cite[Ch. 4]{BogaDiff}, $ \nu_t$ is called ``mixture of measures".

\begin{Lemma}
\label{Le:Gross_fraz}
$t\mapsto \nu_t$ is weakly continuous in $[0, +\infty)$. The generator of $S_t$ is the operator whose resolvent is given by 
\begin{equation}
\label{Kato}
\frac{\sin(s\pi) }{\pi} 
\int_0^{\infty} R(\xi, L)\frac{\xi^{s}}{ \lambda^2 - 2\xi^{s}\cos(s\pi) + \xi^{2s}}\,d\xi, \quad \lambda >0.
\end{equation}
\end{Lemma}
\begin{proof}
Let us check  that $t\mapsto \nu_t$ is weakly continuous. For every $f\in C_b(X)$ and $t>0$ we have
$$\int_X f(x)\nu_t(dx) = \int_X \int_0^{+\infty} f(x) \mathcal N (0, \sigma Q)(dx)\eta^{(s)}_t(\sigma)d\sigma =  \int_X \int_0^{+\infty} f(t^{1/2s}\tau^{1/2} z) N_{0,  Q}(dz)\eta^{(s)}(\tau) d\tau  $$
For $t_0>0$ the right-hand side goes to  $\int_X f(x)\nu_{t_0}(dx)$ as $t\to t_0$, by the Dominated Convergence Theorem. The same holds for $t_0=0$, recalling that $\int_0^{+\infty}  \eta^{(s)}(\tau) d\tau  = 1$. 
 
Concerning the second assertion, using \eqref{Y} for every $\lambda >0$ and $f\in C_b(X)$ we get
$$\begin{array}{l}
\ds \int_0^{\infty} e^{-\lambda t} S_tf(x)dt = \int_0^{\infty} e^{-\lambda t} \int_0^{\infty} P_\sigma f(x) \frac{t^{-1/s}}{\pi} 
 \int_0^{\infty} e^{-\sigma t^{-1/s} r- r^s\cos(s\pi )}\sin(r^s\sin(s\pi))dr\, d\sigma \,dt 
\\
\\
= \ds  \frac{1}{\pi}  \int_0^{\infty} d\xi \left(\int_0^{\infty}  
e^{-\lambda t- t\xi ^s\cos(s\pi )}\sin(t\xi ^s\sin(s\pi)) dt \int_0^{\infty} 
P_\sigma f(x)  e^{-\sigma \xi }d\sigma  \right)
 \\
 \\
 =  \ds   \frac{\sin(s\pi)}{\pi} \int_0^{\infty} R(\xi, L)f(x) \frac{\xi^{s}}{ \lambda^2 - 2\xi^{s}\cos(s\pi) + \xi^{2s}}\,d\xi  
\end{array} $$
(the last equality follows from $\int_0^{\infty} e^{-at}\sin(bt)dt = b/(b^2+a^2)$). 
\end{proof}

We recall that if $L$ is the infinitesimal generator of a bounded strongly continuous semigroup in a Banach space, 
formula \eqref{Kato} coincides with the Kato representation formula for the resolvent of $-(-L)^{s}$ for $s\in (0, 1)$, which may be taken as a definition of $-(-L)^{s}$ (\cite{Kato}). In our case $P_t$ is a contraction semigroup in $C_b(X)$ but it  is not strongly continuous, whereas it is strongly continuous in $BUC(X)$. 
Therefore, the operator whose resolvent is given by  \eqref{Kato} is an extension to $C_b(X)$ of $-(-L_0)^{s}$, where $L_0$ is the part of $L$ in $BUC(X)$, and it may be called $-(-L)^s$, although our case is not covered by the standard theory of powers of (noninvertible) operators.

The following easy lemma will be used here and in the following. 

\begin{Lemma}
\label{Le:rescaled}
Let $\nu$ be a probability measure in a Banach space $X$ that is Fomin differentiable along some $h$, and let $c>0$. Then the measure $\nu_c := \nu \circ(cI)^{-1}$ (namely, $\nu_c (A) = \nu(A/c)$)
is Fomin differentiable along $h$, and 
\begin{equation}
\label{beta_h_rescaled}
\left\{ \begin{array}{ll}(i) & \beta_h^{\nu_c}(y) = \ds  \frac{1}{c}\beta_h^{\nu}\left( \frac{y}{c} \right), \quad \nu_c-a.e. \;y\in X; 
\\
\\
(ii) & \|\beta_h^{\nu_c}\|_{L^1(X, \nu_c)} = \ds \frac{1}{c}\|\beta_h^{\nu}\|_{L^1(X, \nu)} . 
\end{array}\right.
\end{equation}
\end{Lemma}
\begin{proof}
For every $f\in C^1_b(X)$ and $t>0$ we have
$$\begin{array}{l}
\ds \int_X \frac{\partial f}{\partial h}(y)\nu_c(dy) = \int_X \frac{\partial f}{\partial h}(cz) \nu(dz) = \int_X \frac{1}{c}  \frac{\partial  }{\partial h}f(c\;\cdot)(z)\, \nu(dz) 
\\
\\
\ds =  \frac{1}{c}  \int_X f(cz)\beta^{\nu}_h(z)  \nu(dz) =   \frac{1}{c}  \int_X f(y)\beta^{\nu}_h\left( \frac{y}{c} \right) \nu_c(dy), 
\end{array}$$
and \eqref{beta_h_rescaled}(i)  follows. Moreover, 
$$ \|\beta_h^{\nu_c}\|_{L^1(X, \nu_c)}  =  \frac{1}{c} \int_X \left|\beta^{\nu}_h\left( \frac{y}{c} \right)\right|\nu_c(dy) = \frac{1}{c} \int_X |\beta^{\nu}_h ( y )|\nu (dy)$$
which is \eqref{beta_h_rescaled}(ii). 
\end{proof}

\begin{Proposition}
\label{Pr:Gross}
Let $H =Q^{1/2}(X)$, $T_t=I$ for every $t>0$. The   measures $\nu_t$ defined in \eqref{nu_t} satisfy Hypothesis  \ref{Hyp2}, with $\omega =0$ and $\theta = 1/(2s)$. 
\end{Proposition}
\begin{proof}
We have to check that $\nu_t$ is Fomin differentiable along every $h\in H$, and that there exists $C>0$ such that 
\begin{equation}
\label{Hyp2nu_t}
\|\beta^{\nu_t}_h\|_{L^1(X, \nu_t)} \leq \frac{C}{t^{1/(2s)}}, \quad t>0, \; h\in H. 
\end{equation}
Setting as before  $\mu_t :=  \mathcal N (0, t Q) = \mu \circ (t^{1/2}I)^{-1}$, $\mu:= \mathcal N (0, Q)$ we get from Lemma \ref{Le:rescaled}
$$\beta^{\mu_t}_h(y) = \frac{1}{t^{1/2}} \beta^{\mu}_h\left( \frac{y}{t^{1/2}} \right), \quad t>0, \; h\in H, \;y\in X. $$
Consequently, we get
$$\begin{array}{l}
\ds \int_X \frac{\partial f}{\partial h}(y)\nu_t(dy) = \ds \int_X  \frac{\partial f}{\partial h}(y)\int_0^{\infty} \mathcal N (0, \sigma Q)(dy)\eta^{(s)}_t(\sigma)d\sigma
=  \int_0^{\infty} \int_X  \frac{\partial f}{\partial h}(y) \mathcal N (0, \sigma Q)(dy)\eta^{(s)}_t(\sigma)d\sigma
\\
\\
= \ds \int_0^{\infty} \int_X f(y)  \frac{1}{\sigma^{1/2}}  \beta^{\mu}_h\left( \frac{y}{\sigma^{1/2}} \right) \mu_{\sigma}(dy)\eta^{(s)}_t(\sigma) d\sigma 
  = \int_X f(y) \gamma_t(dy)
\end{array}$$
%
%
where the measures $\gamma_t$ are defined by 
$$\begin{array}{lll}
 \gamma_t(A) &  :=  & \ds \int_0^{\infty}  \frac{1}{\sigma^{1/2}}  \left( \int_A  \beta^{\mu}_h\left( \frac{y}{\sigma^{1/2}} \right)  \mu_{\sigma}(dy)\right) \eta^{(s)}_t(\sigma) d\sigma, 
 \\
 \\
 &=& \ds  \int_0^{\infty}  \frac{1}{\sigma^{1/2}}  \left( \int_{A/\sigma^{1/2}}  \beta^{\mu}_h(z) \mu (dz)\right) \eta^{(s)}_t(\sigma) d\sigma ,
 \quad A\in \mathcal B(X). 
 \end{array}$$

Now we prove that each $ \gamma_t$ is absolutely continuous with respect to $\nu_t$. This will be done showing that the positive and negative parts of $\gamma_t$ are respectively given by 
\begin{equation}
\label{gamma_t^+}
 \gamma_t^+(A) =  \int_0^{\infty}  \frac{1}{\sigma^{1/2}}  \left( \int_{A/\sigma^{1/2}} ( \beta^{\mu}_h)^+(z)  \mu (dz)\right) \eta^{(s)}_t(\sigma) d\sigma, \quad A\in \mathcal B(X),
 \end{equation}
\begin{equation}
\label{gamma_t^-}\gamma_t^-(A) =  \int_0^{\infty}  \frac{1}{\sigma^{1/2}}  \left( \int_{A/\sigma^{1/2}} ( \beta^{\mu}_h)^-(z)  \mu (dz)\right) \eta^{(s)}_t(\sigma) d\sigma, \quad A\in \mathcal B(X). 
 \end{equation}
Such representations yield that both $ \gamma_t^+$ and $ \gamma_t^-$ are absolutely continuous with respect to $\nu_t$, because for every  $\nu_t$-negligible $A$ we have by definition $ \int_0^{\infty}  \mu (A/\sigma^{1/2})   \eta^{(s)}_t(\sigma) d\sigma =0$, and since $  \eta^{(s)}_t(\sigma) >0$ for every $\sigma >0$
we get $\mu (A/\sigma^{1/2})  =0$ for a.e. $\sigma >0$ and therefore $ \gamma_t^+(A)=  \gamma_t^-(A) =0$. 

By \cite[Sect. 2.10]{Boga} there exists a $\mu$-version $f_0$ of $\beta_h^{\mu}$ which is linear on a full measure subspace of $X$. We set
$$X^+:= \{x\in X:\; f_0(x)\geq 0\}, \quad X^-:= \{x\in X:\; f_0(x)< 0\}, $$
and we check that $X = X^+ \cup X^-$ is a Hahn decomposition of $(X, \gamma_t)$, namely $  X^+ \cap X^- =\emptyset$ (which is obvious) and 
$$\gamma_t(A\cap X^+ ) \geq 0,\; \gamma_t(A\cap X^-) \leq 0, \quad A\in \mathcal B(X). $$
Indeed, for every $ A\in \mathcal B(X)$ we have
$$\begin{array}{lll}
\gamma_t(A\cap X^+ ) & = & \ds  \int_0^{\infty}  \frac{1}{\sigma^{1/2}}  \left( \int_{(A\cap X^+)/\sigma^{1/2}} (\beta^{\mu}_h)^+(z)  \mu (dz)\right) \eta^{(s)}_t(\sigma) d\sigma
\\
\\
&=& \ds  \int_0^{\infty}  \frac{1}{\sigma^{1/2}}  \left( \int_{(A\cap X^+)/\sigma^{1/2}} f_0(z)  \mu (dz)\right) \eta^{(s)}_t(\sigma) d\sigma .
\end{array}$$
Since $f_0$ is linear on a $\mu$-full measure subspace, then for every $\sigma >0$ the sets $X^+/\sigma^{1/2}$ and $X^+$ may differ only by a $\mu$-negligible set. Therefore, $ f_0(z)\geq 0$ for $\mu$-a.e. $z\in X^+/\sigma^{1/2}$, so that $\int_{(A\cap X^+)/\sigma^{1/2} } f_0(z)  \mu (dz)\geq 0$ and therefore $\gamma_t(A\cap X^+)\geq 0$. The same argument yields $\gamma_t(A\cap X^-)\leq 0$, and \eqref{gamma_t^+}, \eqref{gamma_t^-} follow. 

Therefore, $\gamma_t$ is absolutely continuous with respect to $\nu_t$ and its density is the Fomin derivative $\beta^{\nu_t}_h$ of $\nu_t$ along $h$. Let us estimate its $L^1(X, \nu_t)$ norm. We have 
$$\|\beta^{\nu_t}_h\|_{L^1(X, \nu_t)} = \sup \left\{ \frac{1}{ \|f\|_{\infty}} \int_X f(y) \beta^{\nu_t}_h(y) \nu_t(dy), \; f\in L^{\infty}(X, \nu_t) \setminus\{0\} \right\}, $$
and for every $f\in L^{\infty}(X, \nu_t)$ we have
$$\begin{array}{l}
\ds \int_X f(y)  \beta^{\nu_t}_h(y) \nu_t(dy) = \int_Xf(y)\gamma_t(dy) =  \int_Xf(y) \int_0^{\infty}  \frac{1}{\sigma^{1/2}}   \beta^{\mu}_h\left( \frac{y}{\sigma^{1/2}}   \right)\mu_{\sigma}(dy) \eta^{(s)}_t(\sigma) d\sigma
\\
\\
= \ds \int_0^{\infty} \left(  \frac{1}{\sigma^{1/2}}   \int_Xf(\sigma^{1/2}x) \beta^{\mu}_h(x) \mu(dx)\right) \eta^{(s)}_t(\sigma)d\sigma
\\
\\
 \leq   \ds \int_0^{\infty} \frac{1}{\sigma^{1/2}} \|f\|_{\infty} \|\beta^{\mu}_h\|_{L^1(X, \mu)}  \eta^{(s)}_t(\sigma)d\sigma
 \leq   \|f\|_{\infty}\|h\|_H \int_0^{\infty}  \frac{1}{\sigma^{1/2}}   \eta^{(s)}_t(\sigma)d\sigma, 
\end{array}$$
where
$$\int_0^{\infty}  \frac{1}{\sigma^{1/2}}   \eta^{(s)}_t(\sigma)d\sigma
 =   \int_0^{\infty}  \frac{1}{\sigma^{1/2}}  \eta^{(s)} \left( \frac{\sigma}{t^{1/s}} \right)d\sigma =    \frac{1}{t^{1/2s}} \int_0^{\infty}  \frac{\eta^{(s)}(\tau)}{\tau^{1/2}} d\tau. $$
 Therefore, \eqref{Hyp2nu_t} follows with $C= \int_0^{\infty}  \eta(\tau) \tau^{-1/2} d\tau$. 
  \end{proof}

Thanks to  Lemma \ref{Le:Gross_fraz} and Proposition \ref{Pr:Gross} we can apply
Theorems \ref{Th:Schauder_ell} and \ref{Zygmund_ell}, that give

\begin{Theorem}
\label{Thm:Grossfraz_ell}
Let $f\in C_b(X)$ and $\lambda >0$, $s\in (0,1 )\setminus \{1/2 \}$. Then the equation
\begin{equation}
\label{eq:Grossfraz}
\lambda u +(-L)^s u = f
\end{equation}
has a unique solution $u\in C^{2s}_H(X)$, and there is $C>0$, independent of $f$, such that 
$$\|u\|_{C^{2s}_H(X)} \leq C\|f\|_{\infty}. $$
If $s=1/2$, equation \eqref{eq:Grossfraz} has a unique solution in $Z^1_H(X)$,  and there is $C>0$, independent of $f$, such that 
$$\|u\|_{Z^1_H(X)} \leq C\|f\|_{\infty}. $$
If in addition $f\in C^{\alpha}_H(X)$ with $\alpha \in (0, 1)$ and $\alpha + 2s \notin \{1, 2\}$, then $u\in C^{\alpha + 2s}_H(X)$  and there is $C>0$, independent of $f$, such that 
$$\|u\|_{C^{\alpha + 2s}_H(X)} \leq C\|f\|_{ C^{\alpha}_H(X)}. $$
If $\alpha + 2s = k\in \{1, 2\}$, then $u\in Z^{k}_H(X)$  and there is $C>0$, independent of $f$, such that 
$$\|u\|_{Z^{k}_H(X)} \leq C\|f\|_{ C^{\alpha}_H(X)}. $$
\end{Theorem}

Applying   Theorems \ref{Th:Schauder_par} and \ref{Zygmund_par} we obtain 
 
 \begin{Theorem}
\label{Thm:Grossfraz_par}
Let $s\in (0, 1)$, $\alpha \in [0, 1)$ be such that $\alpha + 2s\notin \{1, 2\}$, and let 
$f\in C^{\alpha + 2s}_H(X)$, $g\in C^{0, \alpha}_H([0,T]\times X)$ \footnote{For $\alpha =0$ we mean $C^{0,0}_H([0,T]\times X) = C_b([0,T]\times X) $.}. The mild solution to
\begin{equation}
\label{eq:fraz_parGross}
\left\{\begin{array}{l}
v_t (t,x) + (-L)^s v (t, \cdot)(x) +  g(t,x), \quad 0\leq t\leq T, \; x\in X, 
 \\
 \\
v(0, x) = f(x), \quad x\in X, 
 \end{array}\right. 
\end{equation}
belongs to $C^{0, \alpha+2s}_H([0,T]\times X)$, and there is $C>0$, independent of $f$ and $g$, such that 
$$\|v\|_{C^{0, \alpha+2s}_H([0,T]\times X)} \leq C(\|f\|_{C^{\alpha + 2s}_H(X)} + \|g\|_{C^{0, \alpha}_H([0,T]\times X)}). $$
Let  $s\in (0, 1)$, $\alpha \in [0, 1)$ be such that $\alpha + 2s =: k \in \{1, 2\}$. Then for every $f\in Z^k_H(X)$, $g\in C^{0, \alpha}_H([0,T]\times X)$
the mild solution to \eqref{eq:fraz_parGross} belongs to $Z^{0, k}_H([0,T]\times X)$,  and there is $C>0$, independent of $f$, such that 
$$\|v\|_{Z^{0,k}_H(X)} \leq C(\|f\|_{Z^{k}_H(X)} + \|g\|_{C^{0, \alpha}_H([0,T]\times X)}). $$
\end{Theorem}


\subsection{ Non-Gaussian classical Ornstein-Uhlenbeck semigroups }
\label{sub:nonGauss}

In this section, as announced earlier, we come back to \eqref{eq:ousemi}, more precisely to its non-Gaussian analogue considered in \cite[Sect. 7]{FR}, for which the semigroup $P_t$ is given by 
\begin{equation}\label{eq:nGsemi}
P_tf(x) = \int_X f(e^{-t}x + (1-e^{-pt})^{1/p}y)\mu (dy), \quad t>0, \; f\in C_b(X),\; x\in X,
\end{equation}
where $\mu$ is a suitable Borel probability measure in a Hilbert space $X$. $P_t$ may be written in the form \eqref{P_t}, with $T_t = e^{-t}I$ and 
\begin{equation}
\label{rescaled}
\mu_t = \mu\circ [(1-e^{-pt})^{1/p}I]^{-1}, \quad t>0. 
\end{equation}
If $\mu$ is a centered Gaussian measure and $p=2$, $P_t$ is the classical Ornstein-Uhlenbeck semigroup considered before in \eqref{eq:ousemi}. 
For $P_t$  to  be a semigroup, $\mu$ cannot be any Borel measure: indeed,  we need that condition \eqref{semigruppo} is satisfied. It is satisfied provided 
$$\hat{\mu}(a) = e^{-\lambda(a)/p}, \quad a \in X^*, $$
and $\lambda :X^*\mapsto \C$ is a negative definite function, which is Sazonov continuous, and such that 
$$ \lambda( ta) = t^p \lambda(a), \quad a\in X^*, \; t>0. $$
 The weak continuity of $t\mapsto \mu_t$ follows immediately from the equality $\int_X f(y)\mu_t(dy) = \int_X  f( (1-e^{-pt})^{1/p}y)\mu (dy)$, for   every $f\in C_b(X)$ and $t\geq 0$. 

We fix now a Banach space $H\subset X$ such that $\mu$ is Fomin differentiable along every $h\in H$. ($H$ may be the whole space $D(\mu)$ of all $h\in X$ such that $\mu$ is Fomin differentiable along  $h$, or a smaller space continuously embedded in $D(\mu)$). 
  In the case where $X$ is e.g. a separable real Hilbert space, an easy example for such a  probability measure $\mu$ with $D(\mu) \supset Q^\frac{1}{2}(X)$ is the measure $\nu_t$ defined in \eqref{nu_t} with $t=\frac{1}{p}$ and $s=\frac{p}{2}$ (recall that $\nu_t$ in \eqref{nu_t} also depends on $s$); in this case it is easy to check that  $\lambda (a) = \langle Qa, a\rangle^{p/2}/2$   and it is convenient to take $H= Q^\frac{1}{2}(X)$  .

   Going back to the general case  , $T_t= e^{-t}I$ maps obviously $H$ into itself. 
Moreover, by Lemma \ref{Le:rescaled}   $\mu_t$ is Fomin differentiable along every $ h\in H$ and we have
$\|\beta^{\mu_t}_h\|_{L^1(X, \mu_t)}  = (1-e^{-pt})^{-1/p}\|\beta^{\mu}_h\|_{L^1(X, \mu)}$.  Therefore, for every $t>0$ and $ h\in D(\mu)$ we have
$$\|\beta^{\mu}_{T_th}\|_{L^1(X, \mu_t)} =  \frac{e^{-t}}{(1-e^{-pt})^{1/p}} \|\beta^{\mu}_h\|_{L^1(X, \mu)} =  \frac{e^{-t}}{(1-e^{-pt})^{1/p}}\|h\|_{D(\mu)} \leq C\frac{e^{-t}}{t^{1/p}} \|h\|_{D(\mu)}, $$
with $C= \sup_{t>0} t^{1/p}(1-e^{-pt})^{-1/p}$. Since $H$   is continuously embedded in    $D(\mu)$, Hypothesis   \ref{Hyp2} is satisfied with $\omega = -1$, $\theta = 1/p$ and our approach applies. Hence   Theorems \ref{Th:Schauder_ell} and \ref{Zygmund_ell} hold for the generator $L$ of the semigroup in $\eqref{eq:nGsemi}$, with $\theta= 1/p$, as well as Theorems \ref{Th:Schauder_par} and \ref{Zygmund_par}.\bigskip
 

 \textbf{Acknowledgements} This work was supported by  DFG through CRC 1283, and by MIUR through the research project PRIN 2015233N54.  The authors are grateful to Jan Van Neerven for discussions about nonsymmetric Ornstein-Uhlenbeck semigroups.


\end{document}